\documentclass[10pt,reqno]{amsart}

\usepackage[utf8x]{inputenc}
\usepackage[english]{babel}
\usepackage{amsmath,amsfonts,amssymb,amsthm,shuffle}
\usepackage[T1]{fontenc}
\usepackage[math]{anttor}

\usepackage[top=3.4cm,bottom=3.4cm,left=3.3cm,right=3.3cm]{geometry}

\usepackage{xcolor}

\definecolor{ColBlack}{RGB}{0,0,0} 
\definecolor{ColWhite}{RGB}{255,255,255} 
\definecolor{Col1}{RGB}{133,6,6} 
\definecolor{Col2}{RGB}{198,8,0} 
\definecolor{Col3}{RGB}{174,74,52} 
\definecolor{Col4}{RGB}{103,113,121} 
\definecolor{Col5}{RGB}{90,94,107} 
\definecolor{Col6}{RGB}{70,63,50} 

\newcommand{\ColA}[1]{\textcolor{Col1}{#1}}

\newcommand{\ColD}[1]{\textcolor{Col4}{#1}}

\usepackage[hyperindex=true,frenchlinks=true,colorlinks=true,
citecolor=Col2,linkcolor=Col3,urlcolor=Col4,linktocpage,
pagebackref=true]{hyperref}

\usepackage{tikz}

\usepackage{dsfont}
\usepackage{wasysym}
\usepackage{stmaryrd}
\usepackage{mathrsfs}
\usepackage{cite}
\usepackage{caption}
\usepackage{subfig}
\usepackage{enumitem}
\usepackage{multicol}

\renewcommand{\arraystretch}{1.4}

\allowdisplaybreaks

\numberwithin{equation}{subsection}

\makeatletter
\def\l@section{\@tocline{1}{3pt}{1pc}{5pc}{}}
\def\l@subsection{\@tocline{2}{2pt}{2pc}{5pc}{}}
\makeatother

\newtheorem{Theorem}{Theorem}[subsection]
\newtheorem{Proposition}[Theorem]{Proposition}
\newtheorem{Lemma}[Theorem]{Lemma}

\renewcommand{\leq}{\leqslant}
\renewcommand{\geq}{\geqslant}

\title[Bud generating systems]
{Colored operads, \\ series on colored operads, \\ and combinatorial
generating systems}
\date{\today}
\author{Samuele Giraudo}
\address{\scriptsize Université Paris-Est, LIGM (UMR $8049$), CNRS,
    ENPC, ESIEE Paris, UPEM, F-$77454$, Marne-la-Vallée, France}
\email{samuele.giraudo@u-pem.fr}

\newcommand{\Bca}{\mathcal{B}}
\newcommand{\Cca}{\mathcal{C}}

\newcommand{\Gca}{\mathcal{G}}

\newcommand{\Oca}{\mathcal{O}}
\newcommand{\Mca}{\mathcal{M}}

\newcommand{\Tca}{\mathcal{T}}
\newcommand{\Afr}{\mathfrak{a}}
\newcommand{\Bfr}{\mathfrak{b}}

\newcommand{\Sfr}{\mathfrak{s}}
\newcommand{\Tfr}{\mathfrak{t}}

\newcommand{\Rfr}{\mathfrak{R}}
\newcommand{\Fbf}{\mathbf{f}}
\newcommand{\Gbf}{\mathbf{g}}
\newcommand{\Hbf}{\mathbf{h}}
\newcommand{\Ubf}{\mathbf{u}}
\newcommand{\Ibf}{\mathbf{i}}
\newcommand{\Rbf}{\mathbf{r}}
\newcommand{\Sbf}{\mathbf{s}}
\newcommand{\Tbf}{\mathbf{t}}
\newcommand{\Xbf}{\mathbf{x}}
\newcommand{\Ccr}{\mathscr{C}}

\newcommand{\Xbb}{\mathbb{X}}
\newcommand{\Ybb}{\mathbb{Y}}
\newcommand{\Zbb}{\mathbb{Z}}
\newcommand{\VarX}{\mathsf{x}}
\newcommand{\VarY}{\mathsf{y}}
\newcommand{\VarZ}{\mathsf{z}}
\newcommand{\Att}{\mathtt{a}}
\newcommand{\Btt}{\mathtt{b}}
\newcommand{\Ctt}{\mathtt{c}}

\newcommand{\K}{\mathbb{K}}
\newcommand{\N}{\mathbb{N}}
\newcommand{\Q}{\mathbb{Q}}

\newcommand{\Bud}{\mathbf{Bud}}
\newcommand{\As}{\mathbf{As}}
\newcommand{\Dias}{\mathbf{Dias}}
\newcommand{\Mag}{\mathbf{Mag}}
\newcommand{\Motz}{\mathbf{Motz}}
\newcommand{\Tree}{\mathbf{Tree}}

\newcommand{\BBalTree}{\Bca_{\mathrm{bbt}}}
\newcommand{\BMotz}{\Bca_{\mathrm{p}}}
\newcommand{\BDias}[1]{\Bca_{\mathrm{w}, #1}}
\newcommand{\BBTree}[1]{\Bca_{\mathrm{bt}, #1}}
\newcommand{\BUBTree}{\Bca_{\mathrm{bu}}}

\newcommand{\Unit}{\mathds{1}}
\newcommand{\Free}{\mathbf{Free}}
\newcommand{\Eval}{\mathrm{ev}}
\newcommand{\In}{\mathrm{in}}
\newcommand{\Out}{\mathrm{out}}
\newcommand{\Angle}[1]{\left\langle#1\right\rangle}
\newcommand{\AAngle}[1]{\Angle{\Angle{#1}}}
\newcommand{\Supp}{\mathrm{Supp}}
\newcommand{\PreLie}{\curvearrowleft}
\newcommand{\Compo}{\odot}
\newcommand{\Deriv}{\to}
\newcommand{\SyncDeriv}{\leadsto}
\newcommand{\Lang}{\mathrm{L}}
\newcommand{\SyncLang}{\mathrm{L}_{\mathrm{S}}}
\newcommand{\DerivGraph}{\mathrm{G}}
\newcommand{\SyncDerivGraph}{\mathrm{G}_{\mathrm{S}}}
\newcommand{\Hook}{\mathrm{hook}}
\newcommand{\Synt}{\mathrm{synt}}
\newcommand{\Sync}{\mathrm{sync}}
\newcommand{\INodes}{\mathcal{N}}
\newcommand{\Colors}{\mathfrak{col}}
\newcommand{\ColorTypes}{\mathfrak{colt}}
\newcommand{\Prune}{\mathfrak{pru}}
\newcommand{\OpAs}{\star}
\newcommand{\Perfect}{\mathrm{perf}}
\newcommand{\Type}{\mathrm{type}}
\newcommand{\Charac}[1]{\bar{#1}}
\newcommand{\Height}{\mathrm{ht}}
\newcommand{\CFG}{\mathrm{CFG}}
\newcommand{\RTG}{\mathrm{RTG}}
\newcommand{\SG}{\mathrm{SG}}
\newcommand{\MultOutIn}{\chi}
\newcommand{\Types}{\Tca}
\newcommand{\HookSystem}{\mathrm{HS}}
\newcommand{\GenS}{\Sbf}
\newcommand{\Hilb}{\Hbf}

\newcommand{\Hide}[1]{\textcolor{Col4}{\tt [hidden]}}
\newcommand{\Par}[1]{\left(#1\right)}
\newcommand{\Bra}[1]{\left\{#1\right\}}
\newcommand{\Han}[1]{\left[#1\right]}

\newcommand{\Def}[1]{\textcolor{Col3}{\em #1}}
\newcommand{\OEIS}[1]{\href{http://oeis.org/#1}{{\bf #1}}}

\tikzstyle{Centering}=[{baseline={([yshift=-0.5ex]current
    bounding box.center)}}]

\tikzstyle{Node}=[circle,draw=Col1!80,fill=Col1!8,inner sep=1pt,
minimum size=2mm,thick,font=\scriptsize]
\tikzstyle{LabeledNode}=[Node,inner sep=.5pt,minimum size=3mm]
\tikzstyle{NodeST}=[font=\footnotesize]
\tikzstyle{Edge}=[draw=Col2!80,cap=round,thick]
\tikzstyle{Leaf}=[rectangle,draw=ColBlack!70,fill=ColBlack!16,
inner sep=0pt,minimum size=1mm,thick]
\tikzstyle{NodeClear}=[Node,fill=ColWhite!100]
\tikzstyle{EdgeLabel}=[midway,inner sep=1pt,fill=ColWhite!0,
font=\scriptsize]
\tikzstyle{LeafLabel}=[font=\scriptsize,node distance=2mm]
\tikzstyle{EdgeValue}=[regular polygon,regular polygon sides=6,
draw=Col2!80,fill=Col2!2,inner sep=.2pt,line width=.75pt,
minimum size=2.8mm,midway,text=ColBlack,font=\tiny]
\tikzstyle{Grid} = [color=ColBlack!30]
\tikzstyle{PathNode}=[circle,draw=Col1!90,fill=Col1!30,thick,
inner sep=0pt,minimum size=2.0mm]
\tikzstyle{PathStep}=[color=Col1!60,thick]
\tikzstyle{PathGrid}=[color=ColBlack!30]
\tikzstyle{EdgeDeriv}=[Col2!80,cap=round,line width=1pt,->]
\tikzstyle{Mark4}=[draw=Col4!80,fill=Col4!10]
\tikzstyle{Mark6}=[draw=Col6!80,fill=Col6!10]

\newcommand{\TreeLeaf}{\,
\begin{tikzpicture}[xscale=.13,yscale=.12,Centering]
    \node[Leaf](0)at(0.00,-1.50){};
    \draw[Edge](0)--(0,0);
\end{tikzpicture}\,}

\newcommand{\CorollaTwo}[1]{\,
\begin{tikzpicture}[xscale=.13,yscale=.15,Centering]
    \node[Leaf](0)at(0.00,-1.50){};
    \node[Leaf](2)at(2.00,-1.50){};
    \node[Node](1)at(1.00,0.00){\begin{math}#1\end{math}};
    \draw[Edge](0)--(1);
    \draw[Edge](2)--(1);
    \node(r)at(1.00,2){};
    \draw[Edge](r)--(1);
\end{tikzpicture}\,}

\newcommand{\FreeCorollaOne}[1]{\,
\begin{tikzpicture}[xscale=.13,yscale=.25,Centering]
    \node(0)at(1.00,-1.50){};
    \node[NodeST](1)at(1.00,0.00){\begin{math}#1\end{math}};
    \draw[Edge](0)--(1);
    \node(r)at(1.00,1.5){};
    \draw[Edge](r)--(1);
\end{tikzpicture}\,}

\newcommand{\FreeCorollaTwo}[1]{\,
\begin{tikzpicture}[xscale=.11,yscale=.1,Centering]
    \node(0)at(0.00,-1.50){};
    \node(2)at(2.00,-1.50){};
    \node[NodeST](1)at(1.00,0.00){\begin{math}#1\end{math}};
    \draw[Edge](0)--(1);
    \draw[Edge](2)--(1);
    \node(r)at(1.00,2.2){};
    \draw[Edge](r)--(1);
\end{tikzpicture}\,}

\newcommand{\MotzHoriz}{\,
\begin{tikzpicture}[scale=.28,Centering]
    \draw[PathGrid](0,0) grid (1,0);
    \node[PathNode](Motz1)at(0,0){};
    \node[PathNode](Motz2)at(1,0){};
    \draw[PathStep](Motz1)--(Motz2);
\end{tikzpicture}\,}

\newcommand{\MotzPeak}{\,
\begin{tikzpicture}[scale=.23,Centering]
    \draw[PathGrid](0,0) grid (2,1);
    \node[PathNode](Motz1)at(0,0){};
    \node[PathNode](Motz2)at(1,1){};
    \node[PathNode](Motz3)at(2,0){};
    \draw[PathStep](Motz1)--(Motz2);
    \draw[PathStep](Motz2)--(Motz3);
\end{tikzpicture}\,}

\begin{document}

\begin{abstract}
    We introduce bud generating systems, which are used for
    combinatorial generation. They specify sets of various kinds of
    combinatorial objects, called languages. They can emulate
    context-free grammars, regular tree grammars, and synchronous
    grammars, allowing us to work with all these generating systems in
    a unified way. The theory of bud generating systems uses colored
    operads. Indeed, an object is generated by a bud generating system
    if it satisfies a certain equation in a colored operad. To compute
    the generating series of the languages of bud generating systems, we
    introduce formal power series on colored operads and several
    operations on these. Series on colored operads are crucial to
    express the languages specified by bud generating systems and
    allow us to enumerate combinatorial objects with respect to some
    statistics. Some examples of bud generating systems are constructed;
    in particular to specify some sorts of balanced trees and to obtain
    recursive formulas enumerating these.
\end{abstract}

\maketitle

\begin{scriptsize}
    \tableofcontents
\end{scriptsize}

\section*{Introduction}
Coming from theoretical computer science and formal language theory,
formal grammars~\cite{Har78,HMU06} are powerful tools having many
applications in several fields of mathematics. A formal grammar is a
device which describes---more or less concisely and with more or less
restrictions---a set of words, called a language. There are several
variations in the definitions of formal grammars and some types of these
are classified by the Chomsky-Schützenberger hierarchy~\cite{Cho59,CS63}
according to four different categories, taking into account their
expressive power. In an increasing order of power, there are the classes
of Type-$3$ to Type-$0$ grammars, known respectively as regular
grammars, context-free grammars, context-sensitive grammars, and
unrestricted grammars. One of the most striking similarities between all
these variations of formal grammars is that they work by constructing
words by applying rewrite rules~\cite{BN98}. Indeed, a word of the
language described by a formal grammar is obtained by considering a
starting word and by iteratively altering it in accordance with the
production rules of the grammar.
\smallbreak

Similar mechanisms and ideas are translatable into the world of trees,
instead only those of words. Grammars of trees~\cite{CDGJLLTT07} are
thus the natural counterpart of formal grammars to describe sets of
trees, and here also, there exist several very different types of
grammars. One can cite for instance tree grammars, regular tree
grammars~\cite{GS84}, and synchronous grammars~\cite{Gir12}, which
describe sets of various kinds of treelike structures. These grammars,
like the previous ones, work by applying rewrite rules on trees. In
this framework, trees are constructed by growing from the root to the
leaves by replacing some subtrees by other ones.
\smallbreak

In contrast, the theory of operads seems to have no link with formal
grammars. Operads are algebraic structures introduced in the context of
algebraic topology~\cite{May72,BV73} (see also~\cite{Mar08,LV12,Men15}
for a modern conspectus of the theory). Recently, many links between
the theory of operads and combinatorics have been developed (see, for
instance~\cite{CL01,Cha08,CG14}). Many operads involving various sets of
combinatorial objects have been defined, so that almost every classical
object can be seen as an element of at least one operad (see the
previous references and for instance~\cite{Zin12,Gir15,Gir16,FFM18}).
From an intuitive point of view, an operad is a set of abstract
operators with several inputs and one output. More precisely, if $x$ is
an operator with $n$ inputs and $y$ is an operator with $m$ inputs,
$x \circ_i y$ denotes the operator with $n + m - 1$ inputs obtained by
gluing the output of $y$ to the $i$-th input of $x$. Operads are
algebraic structures related to trees in the same way that monoids are
algebraic structures related to words. For this reason, the study of
operads has many connections with combinatorial properties of trees.
\smallbreak

The initial spark of this work was the following simple observation. The
partial composition $x \circ_i y$ of two elements $x$ and $y$ of an
operad $\Oca$ can be regarded as the application of a rewrite rule on
$x$ to obtain a new element of $\Oca$---the rewrite rule being encoded
essentially by $y$. This leads to the idea of using an operad $\Oca$ to
define grammars generating some subsets of $\Oca$. In this way,
according to the nature of the elements of $\Oca$, this provides a way
to define grammars which generate objects different to words and trees.
We use colored operads~\cite{BV73,Yau16}, a generalization of operads.
In a colored operad $\Cca$, every input and every output for the
elements of $\Cca$ has a color. These colors create constraints for the
partial compositions of two elements. Indeed, $x \circ_i y$ is defined
only if the color of the output of $y$ is the same as the color of the
$i$-th input of $x$. Colored operads enable us to define a new kind of
grammar, since the restrictions provided by the colors allow precise
control on how the rewrite rules can be applied.
\smallbreak

We introduce a new kind of grammar, the bud generating system. This is
defined mainly from a ground operad $\Oca$, a set $\Ccr$ of colors, and
a set $\Rfr$ of production rules. A bud generating system describes a
subset of $\Bud_\Ccr(\Oca)$---the colored operad obtained by augmenting
the elements of $\Oca$ with input and output colors taken from $\Ccr$.
An element is generated by iteratively altering an element $x$ of
$\Bud_\Ccr(\Oca)$ by composing it with an element $y$ of $\Rfr$. In this
context, the colors play the role analogous to nonterminal symbols in
formal grammars and in grammars of trees. Any bud generating system
$\Bca$ specifies two sets of objects: its language $\Lang(\Bca)$ and its
synchronous language $\SyncLang(\Bca)$. For instance, they can be used
to describe sets of Motkzin paths with some constraints,  the set of
$\{2, 3\}$-perfect trees~\cite{MPRS79,CLRS09} and some of their
generalizations, and the set of balanced binary trees~\cite{AVL62}.
Indeed, bud generating systems can emulate both context-free grammars
and regular tree grammars, and allow us to see these in a unified
manner.
\smallbreak

This paper is organized as follows. Section~\ref{sec:colored_operads} is
devoted to set our notations and definitions about operads and colored
operads, and to introduce the construction $\Bud_\Ccr(\Oca)$ producing
a colored operad from a noncolored one $\Oca$ and a set $\Ccr$ of
colors. Section~\ref{sec:bud_generating_systems} contains the main
definition: bud generating systems. We establish some properties of
these. Next, we introduce formal power series on colored operads in
Section~\ref{sec:series_colored_operads}, define several products on
these, and explain how these series can be used to obtain enumerative
results from bud generating systems. This article ends with
Section~\ref{sec:examples}, which contains a collection of examples for
most of the notions introduced by this work. We have taken the freedom
to put all the examples in this section. For this reason, the reader is
encouraged to consult this section whilst reading the first ones, by
following the references.
\medbreak

\subsubsection*{Acknowledgements}
The author would like to thank the anonymous referees for their very
valuable suggestions, improving the presentation of this paper.
\medbreak

\subsubsection*{General notations and conventions}
We denote by $\delta_{x, y}$ the Kronecker delta function (that is, for
any elements $x$ and $y$ of a same set, $\delta_{x, y} = 1$ if $x = y$
and $\delta_{x, y} = 0$ otherwise). For any integers $a$ and $c$,
$[a, c]$ denotes the set $\Bra{b \in \N : a \leq b \leq c}$ and $[n]$,
the set $[1, n]$. The cardinality of a finite set $S$ is denoted by
$\# S$. For any finite multiset $S := \lbag s_1, \dots, s_n \rbag$ of
nonnegative integers, we denote by $\sum S$ the sum
\begin{equation}
    \sum S := s_1 + \dots + s_n
\end{equation}
of its elements and by $S!$ the multinomial coefficient
\begin{equation}
    S! := \binom{\sum S}{s_1, \dots, s_n}.
\end{equation}
\medbreak

For any set $A$, $A^*$ denotes the set of all finite sequences, called
words, of elements of $A$. We denote by $A^+$ the subset of $A^*$
consisting in nonempty words. For any $n \geq 0$, $A^n$ is the set of
all words on $A$ of length $n$. If $u$ is a word, its letters are
indexed from left to right from $1$ to its length $|u|$. For any
$i \in [|u|]$, $u_i$ is the letter of $u$ at position $i$. If $a$ is a
letter and $n$ is a nonnegative integer, $a^n$ denotes the word
consisting in $n$ occurrences of $a$. Notice that $a^0$ is the empty
word~$\epsilon$.
\medbreak

In our graphical representations of trees, the uppermost nodes are
always roots. Moreover, internal nodes are represented by circles
\scalebox{1.1}{\begin{tikzpicture}[Centering]
    \node[Node]{};
\end{tikzpicture}},
leaves by squares
\scalebox{1.3}{\begin{tikzpicture}[Centering]
    \node[Leaf]{};
\end{tikzpicture}},
and edges by segments
\begin{tikzpicture}[Centering]
    \draw[Edge](0,0)--(0,.25);
\end{tikzpicture}.
To distinguish trees and syntax trees, we shall draw the latter
without circles for internal nodes and without squares for leaves (only
the labels of the nodes are depicted).
\medbreak

In graphical representations of multigraphs, labels of edges denote
their multiplicities. All unlabeled edges have $1$ as multiplicity.
\medbreak

\section{Colored operads and bud operads} \label{sec:colored_operads}
The aim of this section is to set our notations about operads, colored
operads, and colored syntax trees. We also establish some properties
of treelike expressions in colored operads and present a construction
producing colored operads from operads.
\medbreak

\subsection{Colored operads}
Let us recall here the definitions of colored graded collections and
colored operads.
\medbreak

\subsubsection{Colored graded collections}
Let $\Ccr$ be a finite set, called \Def{set of colors}. A
\Def{$\Ccr$-colored graded collection} is a graded set
\begin{equation}
    C := \bigsqcup_{n \geq 1} C(n)
\end{equation}
together with two maps
\begin{math}
    \Out : C \to \Ccr
\end{math}
and
\begin{math}
    \In : C(n) \to \Ccr^n, n \geq 1,
\end{math}
respectively sending any $x \in C(n)$ to its \Def{output color}
$\Out(x)$ and to its \Def{word of input colors} $\In(x)$. The
\Def{$i$-th input color} of $x$ is the $i$-th letter of $\In(x)$,
denoted by $\In_i(x)$. For any $n \geq 1$ and $x \in C(n)$, the
\Def{arity} $|x|$ of $x$ is $n$. We say that $C$ is \Def{locally finite}
if for all $n \geq 1$, the $C(n)$ are finite sets. A \Def{monochrome
graded collection} is a $\Ccr$-colored graded collection where $\Ccr$
is a singleton. If $C_1$ and $C_2$ are two $\Ccr$-colored graded
collections, a map $\phi : C_1 \to C_2$ is a \Def{$\Ccr$-colored graded
collection morphism} if it preserves arities. Besides, $C_2$ is a
\Def{$\Ccr$-colored graded subcollection} of $C_1$ if for all
$n \geq 1$, $C_2(n) \subseteq C_1(n)$, and $C_1$ and $C_2$ have the same
maps $\Out$ and~$\In$.
\medbreak

\subsubsection{Hilbert series} \label{subsubsec:hilbert_series}
In all this work, we consider that $\Ccr$ has cardinal $k$ and that the
colors of $\Ccr$ are arbitrarily indexed so that
$\Ccr = \Bra{c_1, \dots, c_k}$. Let
$\Xbb_\Ccr := \Bra{\VarX_{c_1}, \dots, \VarX_{c_k}}$ and
$\Ybb_\Ccr := \Bra{\VarY_{c_1}, \dots, \VarY_{c_k}}$ be two alphabets of
mutually commutative parameters and $\N[[\Xbb_\Ccr \sqcup \Ybb_\Ccr]]$
be the set of commutative multivariate series on
$\Xbb_\Ccr \sqcup \Ybb_\Ccr$ with nonnegative integer coefficients.
As usual, if $\Sbf$ is a series of $\N[[\Xbb_\Ccr \sqcup \Ybb_\Ccr]]$,
$\Angle{m, \Sbf}$ denotes the coefficient of the monomial $m$ in $\Sbf$.
\medbreak

For any $\Ccr$-colored graded collection $C$, the \Def{Hilbert series}
$\Hilb_C$ of $C$ is the series of
$\N[[\Xbb_\Ccr \sqcup \Ybb_\Ccr]]$ defined by
\begin{equation} \label{equ:hilbers_series_colored_graded_collection}
    \Hilb_C := \sum_{x \in C}
    \Par{\VarX_{\Out(x)} \prod_{i \in [|x|]} \VarY_{\In_i(x)}}.
\end{equation}
The coefficient of
$\VarX_{a} \VarY_{c_1}^{\alpha_1} \dots \VarY_{c_k}^{\alpha_k}$ in
$\Hilb_C$ thus counts the elements of $C$ having $a$ as output color
and $\alpha_j$ inputs of color $c_j$ for any~$j \in [k]$. Note
that~\eqref{equ:hilbers_series_colored_graded_collection} is defined
only if there are only finitely many such elements for any $a \in \Ccr$
and any $\alpha_j \geq 0$, $j \in [k]$. This is the case when $C$ is
locally finite.
\medbreak

Besides, the \Def{generating series} of $C$ is the series $\GenS_C$ of
$\N[[t]]$ defined as the specialization of $\Hilb_C$ at $\VarX_a := 1$
and $\VarY_a := t$ for all $a \in \Ccr$. Therefore, for any $n \geq 1$
the coefficient $\Angle{t^n, \GenS_C}$ counts the elements of arity $n$
in~$C$.
\medbreak

\subsubsection{Colored operads} \label{subsubsec:colored_operads}
A \Def{nonsymmetric colored set-operad} on $\Ccr$, or a
\Def{$\Ccr$-colored operad} for short, is a colored graded collection
$\Cca$ together with partially defined maps
\begin{equation}
    \circ_i : \Cca(n) \times \Cca(m) \to \Cca(n + m - 1),
    \qquad 1 \leq i \leq n, \enspace 1 \leq m,
\end{equation}
called \Def{partial compositions}, and a subset
$\Bra{\Unit_a : a \in \Ccr}$ of $\Cca(1)$ such that any $\Unit_a$,
$a \in \Ccr$, is called \Def{unit of color $a$} and satisfies
$\Out\Par{\Unit_a} = \In\Par{\Unit_a} = a$. This data has to satisfy,
for any $x, y, z \in \Cca$, the following constraints. First, for any
$i \in [|x|]$, $x \circ_i y$ is defined if and only if
$\Out(y) = \In_i(x)$. Moreover, the relations
\begin{subequations}
\begin{equation} \label{equ:operad_axiom_1}
    \Par{x \circ_i y} \circ_{i + j - 1} z = x \circ_i \Par{y \circ_j z},
    \qquad
    1 \leq i \leq |x|, 1 \leq j \leq |y|,
\end{equation}
\begin{equation} \label{equ:operad_axiom_2}
    \Par{x \circ_i y} \circ_{j + |y| - 1} z
    = \Par{x \circ_j z} \circ_i y,
    \qquad
    1 \leq i < j \leq |x|,
\end{equation}
\begin{equation} \label{equ:operad_axiom_3}
    \Unit_a \circ_1 x = x = x \circ_i \Unit_b,
    \qquad 1 \leq i \leq |x|, \enspace a, b \in \Ccr,
\end{equation}
\end{subequations}
have to hold when they are well-defined.
\medbreak

The \Def{complete composition map} of $\Cca$ is the partially defined
map
\begin{equation}
    \circ : \Cca(n) \times \Cca\Par{m_1}
        \times \dots \times \Cca\Par{m_n}
    \to \Cca\Par{m_1 + \dots + m_n},
\end{equation}
defined from the partial composition maps in the following way. For any
$x \in \Cca(n)$ and $y_1, \dots, y_n \in \Cca$ such that
$\Out\Par{y_i} = \In_i(y)$ for all $i \in [n]$, we set
\begin{equation}
    x \circ \Han{y_1, \dots, y_n}
    := \Par{\dots \Par{\Par{x \circ_n y_n}
        \circ_{n - 1} y_{n - 1}} \dots} \circ_1 y_1.
\end{equation}
\medbreak

Let $\Cca_1$ and $\Cca_2$ are two $\Ccr$-colored operads. A
$\Ccr$-colored graded collection morphism $\phi : \Cca_1 \to \Cca_2$ is
a \Def{$\Ccr$-colored operad morphism} if it sends any unit of color
$a \in \Ccr$ of $\Cca_1$ to the unit of color $a$ of $\Cca_2$,
if it commutes with partial composition maps and, if
for any $x, y \in \Cca_1$ and $i \in [|x|]$, if $x \circ_i y$ is
defined in $\Cca_1$, then $\phi(x) \circ_i \phi(y)$ is defined in
$\Cca_2$. Besides, $\Cca_2$ is a \Def{colored suboperad} of $\Cca_1$ if
$\Cca_2$ is a $\Ccr$-colored graded subcollection of $\Cca_1$ and
$\Cca_1$ and $\Cca_2$ have the same colored units and the same partial
composition maps. If $G$ is a $\Ccr$-colored graded subcollection of
$\Cca$, we denote by $\Cca^G$ the \Def{$\Ccr$-colored operad generated}
by $G$, that is the smallest $\Ccr$-colored suboperad of $\Cca$
containing $G$. When the $\Ccr$-colored operad generated by $G$ is
$\Cca$ itself, $G$ is a \Def{generating $\Ccr$-colored graded
collection} of $\Cca$. Moreover, when $G$ is minimal with respect to
inclusion among the $\Ccr$-colored graded subcollections of $\Cca$
satisfying this property, $G$ is a \Def{minimal generating
$\Ccr$-colored graded collection of~$\Cca$}. We say that $\Cca$ is
\Def{locally finite} if, as a colored graded collection, $\Cca$ is
locally finite.
\medbreak

A \Def{monochrome operad} (or an \Def{operad} for short) $\Oca$ is a
$\Ccr$-colored operad with a monochrome graded collection as underlying
set. In this case, $\Ccr$ is a singleton $\Bra{c_1}$ and, since for all
$x \in \Oca(n)$, we necessarily have $\Out(x) = c_1$ and
$\In(x) = c_1^n$, for all $x, y \in \Oca$ and $i \in [|x|]$, all partial
compositions $x \circ_i y$ are defined. In this case, $\Ccr$ and its
single element $c_1$ do not play any role. For this reason, in the
future definitions of monochrome operads, we shall not define their
set of colors~$\Ccr$.
\medbreak

\subsection{Free colored operads}
Free colored operads and more particularly colored syntax trees play an
important role in this work. We recall here the definitions of these two
notions and establish some of their properties.
\medbreak

\subsubsection{Colored syntax trees}
Unless otherwise specified, we use in the sequel the standard
terminology ({\em i.e.}, \Def{node}, \Def{edge}, \Def{root},
\Def{parent}, \Def{child}, \Def{path}, {\em etc.}) about planar rooted
trees~\cite{Knu97}. Let $\Ccr$ be a set of colors and $C$ be a
$\Ccr$-colored graded collection. A \Def{$\Ccr$-colored $C$-syntax tree}
is a planar rooted tree $\Tfr$ such that, for any $n \geq 1$, any
internal node of $\Tfr$ having $n$ children is labeled by an element of
arity $n$ of $C$ and, for any internal nodes $u$ and $v$ of $\Tfr$ such
that $v$ is the $i$-th child of $u$, $\Out(y) = \In_i(x)$ where $x$
(resp. $y$) is the label of $u$ (resp. $v$). In our graphical
representations of a $\Ccr$-colored $C$-syntax tree $\Tfr$, we write the
colors of the leaves of $\Tfr$ below them and the color of the edge
exiting the root of $\Tfr$ above it (see
Figure~\ref{fig:example_colored_syntax_tree}).
\begin{figure}[ht]
    \centering
    \subfloat[][The degree of this $\Ccr$-colored $C$-syntax tree is
    $5$, its arity is~$8$, and its height is~$3$]{
    \begin{minipage}[c]{.45\textwidth}
    \centering
    \begin{equation*}
        \begin{tikzpicture}[xscale=.37,yscale=.15,Centering]
            \node(0)at(0.00,-6.50){};
            \node(10)at(8.00,-9.75){};
            \node(12)at(10.00,-9.75){};
            \node(2)at(2.00,-6.50){};
            \node(4)at(3.00,-3.25){};
            \node(5)at(4.00,-9.75){};
            \node(7)at(5.00,-9.75){};
            \node(8)at(6.00,-9.75){};
            \node[NodeST](1)at(1.00,-3.25){\begin{math}\Btt\end{math}};
            \node[NodeST](11)at(9.00,-6.50){\begin{math}\Att\end{math}};
            \node[NodeST](3)at(3.00,0.00){\begin{math}\Ctt\end{math}};
            \node[NodeST](6)at(5.00,-6.50){\begin{math}\Ctt\end{math}};
            \node[NodeST](9)at(7.00,-3.25){\begin{math}\Att\end{math}};
            \node(r)at(3.00,2.75){};
            \draw[Edge](0)--(1);
            \draw[Edge](1)--(3);
            \draw[Edge](10)--(11);
            \draw[Edge](11)--(9);
            \draw[Edge](12)--(11);
            \draw[Edge](2)--(1);
            \draw[Edge](4)--(3);
            \draw[Edge](5)--(6);
            \draw[Edge](6)--(9);
            \draw[Edge](7)--(6);
            \draw[Edge](8)--(6);
            \draw[Edge](9)--(3);
            \draw[Edge](r)--(3);
            \node[LeafLabel,above of=r]{\begin{math}1\end{math}};
            \node[LeafLabel,below of=0]{\begin{math}2\end{math}};
            \node[LeafLabel,below of=2]{\begin{math}1\end{math}};
            \node[LeafLabel,below of=4]{\begin{math}2\end{math}};
            \node[LeafLabel,below of=5]{\begin{math}2\end{math}};
            \node[LeafLabel,below of=7]{\begin{math}2\end{math}};
            \node[LeafLabel,below of=8]{\begin{math}1\end{math}};
            \node[LeafLabel,below of=10]{\begin{math}1\end{math}};
            \node[LeafLabel,below of=12]{\begin{math}1\end{math}};
        \end{tikzpicture}
    \end{equation*}
    \end{minipage}
    \label{subfig:example_colored_syntax_tree_1}}
    \subfloat[][A perfect $\Ccr$-colored $C$-syntax tree. The degree of
    this colored syntax tree is $8$, its arity is $11$, and its height
    is~$3$.]{
    \begin{minipage}[c]{.5\textwidth}
    \centering
    \begin{equation*}
        \begin{tikzpicture}[xscale=.29,yscale=.11,Centering]
            \node(0)at(0.00,-14.25){};
            \node(11)at(10.00,-14.25){};
            \node(13)at(11.00,-14.25){};
            \node(14)at(12.00,-14.25){};
            \node(16)at(14.00,-14.25){};
            \node(18)at(16.00,-14.25){};
            \node(2)at(2.00,-14.25){};
            \node(4)at(3.00,-14.25){};
            \node(6)at(5.00,-14.25){};
            \node(7)at(6.00,-14.25){};
            \node(9)at(8.00,-14.25){};
            \node[NodeST](1)at(1.00,-9.50){\begin{math}\Btt\end{math}};
            \node[NodeST](10)at(9.00,0.00){\begin{math}\Att\end{math}};
            \node[NodeST](12)at(11.00,-9.50)
                {\begin{math}\Ctt\end{math}};
            \node[NodeST](15)at(13.00,-4.75)
                {\begin{math}\Att\end{math}};
            \node[NodeST](17)at(15.00,-9.50)
                {\begin{math}\Att\end{math}};
            \node[NodeST](3)at(4.00,-4.75){\begin{math}\Ctt\end{math}};
            \node[NodeST](5)at(4.00,-9.50){\begin{math}\Btt\end{math}};
            \node[NodeST](8)at(7.00,-9.50){\begin{math}\Att\end{math}};
            \draw[Edge](0)--(1);
            \draw[Edge](1)--(3);
            \draw[Edge](11)--(12);
            \draw[Edge](12)--(15);
            \draw[Edge](13)--(12);
            \draw[Edge](14)--(12);
            \draw[Edge](15)--(10);
            \draw[Edge](16)--(17);
            \draw[Edge](17)--(15);
            \draw[Edge](18)--(17);
            \draw[Edge](2)--(1);
            \draw[Edge](3)--(10);
            \draw[Edge](4)--(5);
            \draw[Edge](5)--(3);
            \draw[Edge](6)--(5);
            \draw[Edge](7)--(8);
            \draw[Edge](8)--(3);
            \draw[Edge](9)--(8);
            \node(r)at(9.00,3.25){};
            \draw[Edge](r)--(10);
            \node[LeafLabel,above of=r]{\begin{math}1\end{math}};
            \node[LeafLabel,below of=0]{\begin{math}2\end{math}};
            \node[LeafLabel,below of=2]{\begin{math}1\end{math}};
            \node[LeafLabel,below of=4]{\begin{math}2\end{math}};
            \node[LeafLabel,below of=6]{\begin{math}1\end{math}};
            \node[LeafLabel,below of=7]{\begin{math}1\end{math}};
            \node[LeafLabel,below of=9]{\begin{math}1\end{math}};
            \node[LeafLabel,below of=11]{\begin{math}2\end{math}};
            \node[LeafLabel,below of=13]{\begin{math}2\end{math}};
            \node[LeafLabel,below of=14]{\begin{math}1\end{math}};
            \node[LeafLabel,below of=16]{\begin{math}1\end{math}};
            \node[LeafLabel,below of=18]{\begin{math}1\end{math}};
        \end{tikzpicture}
    \end{equation*}
    \end{minipage}
    \label{subfig:example_colored_syntax_tree_2}}
    \caption{\footnotesize
    Two $\Ccr$-colored $C$-syntax trees, where $\Ccr$ is the set of
    colors $\{1, 2\}$ and $C$ is the $\Ccr$-graded colored collection
    defined by $C := C(2) \sqcup C(3)$ with $C(2) := \{\Att, \Btt\}$,
    $C(3) := \{\Ctt\}$, $\Out(\Att) := 1$, $\Out(\Btt) := 2$,
    $\Out(\Ctt) := 1$, $\In(\Att) := 11$, $\In(\Btt) := 21$, and
    $\In(\Ctt) := 221$.}
    \label{fig:example_colored_syntax_tree}
\end{figure}
\medbreak

Let $\Tfr$ be a $\Ccr$-colored $C$-syntax tree. The \Def{arity} of an
internal node $v$ of $\Tfr$ is its number $|v|$ of children and its
\Def{label} is the element of $C$ labeling it and denoted by $\Tfr(v)$.
The \Def{degree} $\deg(\Tfr)$ (resp. \Def{arity} $|\Tfr|$) of $\Tfr$ is
its number of internal nodes (resp. leaves). We say that $\Tfr$ is
a \Def{corolla} if $\deg(\Tfr) = 1$. The \Def{height} of $\Tfr$ is the
length $\Height(\Tfr)$ of a longest path connecting the root of $\Tfr$
to one of its leaves. For instance, the height of a colored syntax tree
of degree $0$ is $0$ and the one of a corolla is $1$. The set of all
internal nodes of $\Tfr$ is denoted by $\INodes(\Tfr)$. For any
$v \in \INodes(\Tfr)$, $\Tfr_v$ is the subtree of $\Tfr$ rooted at the
node $v$. We say that $\Tfr$ is \Def{perfect} if all paths connecting
the root of $\Tfr$ to its leaves have the same length. Finally, $\Tfr$
is a \Def{monochrome $C$-syntax tree} if $C$ is a monochrome graded
collection.
\medbreak

\subsubsection{Free colored operads}
The \Def{free $\Ccr$-colored operad over $C$} is the operad $\Free(C)$
wherein for any $n \geq 1$, $\Free(C)(n)$ is the set of all
$\Ccr$-colored $C$-syntax trees of arity $n$. For any
$\Tfr \in \Free(C)$, $\Out(\Tfr)$ is the output color of the label of
the root of $\Tfr$ and $\In(\Tfr)$ is the word obtained by reading, from
left to right, the input colors of the leaves of $\Tfr$. For any
$\Sfr, \Tfr \in \Free(C)$, the partial composition $\Sfr \circ_i \Tfr$,
defined if and only if the output color of $\Tfr$ is the input color of
the $i$-th leaf of $\Sfr$, is the tree obtained by grafting the root of
$\Tfr$ to the $i$-th leaf of $\Sfr$. For instance, with the
$\Ccr$-colored graded collection $C$ defined in
Figure~\ref{fig:example_colored_syntax_tree}, one has in~$\Free(C)$,
\begin{equation}
    \begin{tikzpicture}[xscale=.32,yscale=.16,Centering]
        \node(0)at(0.00,-5.33){};
        \node(2)at(2.00,-5.33){};
        \node(4)at(4.00,-5.33){};
        \node(6)at(5.00,-5.33){};
        \node(7)at(6.00,-5.33){};
        \node[NodeST](1)at(1.00,-2.67){\begin{math}\Att\end{math}};
        \node[NodeST](3)at(3.00,0.00){\begin{math}\Att\end{math}};
        \node[NodeST](5)at(5.00,-2.67){\begin{math}\Ctt\end{math}};
        \node(r)at(3.00,2.25){};
        \draw[Edge](0)--(1);
        \draw[Edge](1)--(3);
        \draw[Edge](2)--(1);
        \draw[Edge](4)--(5);
        \draw[Edge](5)--(3);
        \draw[Edge](6)--(5);
        \draw[Edge](7)--(5);
        \draw[Edge](r)--(3);
        \node[LeafLabel,above of=r]{\begin{math}1\end{math}};
        \node[LeafLabel,below of=0]{\begin{math}1\end{math}};
        \node[LeafLabel,below of=2]{\begin{math}1\end{math}};
        \node[LeafLabel,below of=4]{\begin{math}2\end{math}};
        \node[LeafLabel,below of=6]{\begin{math}2\end{math}};
        \node[LeafLabel,below of=7]{\begin{math}1\end{math}};
    \end{tikzpicture}
    \enspace \circ_3 \enspace
    \begin{tikzpicture}[xscale=.3,yscale=.24,Centering]
        \node(0)at(0.00,-1.67){};
        \node(2)at(2.00,-3.33){};
        \node(4)at(4.00,-3.33){};
        \node[NodeST](1)at(1.00,0.00)
            {\begin{math}\Btt\end{math}};
        \node[NodeST](3)at(3.00,-1.67)
            {\begin{math}\Att\end{math}};
        \node(r)at(1.00,1.65){};
        \draw[Edge,Mark4](0)--(1);
        \draw[Edge,Mark4](2)--(3);
        \draw[Edge,Mark4](3)--(1);
        \draw[Edge,Mark4](4)--(3);
        \draw[Edge,Mark4](r)--(1);
        \node[LeafLabel,above of=r]{\begin{math}2\end{math}};
        \node[LeafLabel,below of=0]{\begin{math}2\end{math}};
        \node[LeafLabel,below of=2]{\begin{math}1\end{math}};
        \node[LeafLabel,below of=4]{\begin{math}1\end{math}};
    \end{tikzpicture}
    \enspace = \enspace
    \begin{tikzpicture}[xscale=.28,yscale=.18,Centering]
        \node(0)at(0.00,-4.80){};
        \node(10)at(9.00,-4.80){};
        \node(11)at(10.00,-4.80){};
        \node(2)at(2.00,-4.80){};
        \node(4)at(4.00,-7.20){};
        \node(6)at(6.00,-9.60){};
        \node(8)at(8.00,-9.60){};
        \node[NodeST](1)at(1.00,-2.40){\begin{math}\Att\end{math}};
        \node[NodeST](3)at(3.00,0.00){\begin{math}\Att\end{math}};
        \node[NodeST](5)at(5.00,-4.80){\begin{math}\Btt\end{math}};
        \node[NodeST](7)at(7.00,-7.20){\begin{math}\Att\end{math}};
        \node[NodeST](9)at(9.00,-2.40){\begin{math}\Ctt\end{math}};
        \node(r)at(3.00,2.25){};
        \draw[Edge](0)--(1);
        \draw[Edge](1)--(3);
        \draw[Edge](10)--(9);
        \draw[Edge](11)--(9);
        \draw[Edge](2)--(1);
        \draw[Edge,Mark4](4)--(5);
        \draw[Edge,Mark6](5)--(9);
        \draw[Edge,Mark4](6)--(7);
        \draw[Edge,Mark4](7)--(5);
        \draw[Edge,Mark4](8)--(7);
        \draw[Edge](9)--(3);
        \draw[Edge](r)--(3);
        \node[LeafLabel,above of=r]{\begin{math}1\end{math}};
        \node[LeafLabel,below of=0]{\begin{math}1\end{math}};
        \node[LeafLabel,below of=2]{\begin{math}1\end{math}};
        \node[LeafLabel,below of=4]{\begin{math}2\end{math}};
        \node[LeafLabel,below of=6]{\begin{math}1\end{math}};
        \node[LeafLabel,below of=8]{\begin{math}1\end{math}};
        \node[LeafLabel,below of=10]{\begin{math}2\end{math}};
        \node[LeafLabel,below of=11]{\begin{math}1\end{math}};
    \end{tikzpicture}\,.
\end{equation}
\medbreak

\subsubsection{Treelike expressions and finitely factorizing sets}
For any $\Ccr$-colored operad $\Cca$, the \Def{evaluation map} of $\Cca$
is the map
\begin{math}
    \Eval_\Cca : \Free(\Cca) \to \Cca,
\end{math}
defined as the unique surjective morphism of colored operads satisfying
$\Eval_\Cca(\Tfr) = x$ where $\Tfr$ is a tree of degree $1$ having its
root labeled by $x$. If $S$ is a colored graded subcollection of $\Cca$,
an \Def{$S$-treelike expression} of $x \in \Cca$ is a tree $\Tfr$ of
$\Free(\Cca)$ such that $\Eval_\Cca(\Tfr) = x$ and all internal nodes of
$\Tfr$ are labeled on~$S$.
\medbreak

Besides, when $S$ is such that any $x \in \Cca$ admits finitely many
$S$-treelike expressions, we say that $S$ \Def{finitely factorizes}
$\Cca$. This notion is important in the sequel and is used as sufficient
condition for the well-definition of some formal power series on
colored operads.
\medbreak

\begin{Lemma} \label{lem:finitely_factorizing_sets}
    Let $\Cca$ be a locally finite $\Ccr$-colored operad and $S$ be a
    $\Ccr$-colored graded subcollection of $\Cca$ such that $S(1)$
    finitely factorizes $\Cca$. Then, $S$ finitely factorizes~$\Cca$.
\end{Lemma}
\begin{proof}
    Since $\Cca$ is locally finite and $S(1)$ finitely factorizes
    $\Cca$, there is a nonnegative integer $k$ such that $k$ is the
    degree of a $\Ccr$-colored $S(1)$-syntax tree with a maximal
    number of internal nodes. Let $x$ be an element of $\Cca$ of arity
    $n$ admitting an $S$-treelike expression $\Tfr$. Observe
    first that $\Tfr$ has at most $n - 1$ non-unary internal nodes and
    at most $2n - 1$ edges. Moreover, by the pigeonhole principle, if
    $\Tfr$ would have more than  $(2n - 1)k$ unary internal nodes, there
    would be a chain made of more than $k$ unary internal nodes in
    $\Tfr$. This cannot happen since, by hypothesis, it is not possible
    to form any $\Ccr$-colored $S(1)$-syntax tree with more than $k$
    nodes. Therefore, we have shown that all $S$-treelike expressions
    of $x$ are of degrees at most $n - 1 + (2n - 1)k$. Moreover, since
    $\Cca$ is locally finite and $S$ is a $\Ccr$-colored graded
    subcollection of $\Cca$, all $S(m)$ are finite for all $m \geq 1$.
    Therefore, there are finitely many $S$-treelike expressions of~$x$.
\end{proof}
\medbreak

Assume now that $\Cca$ is locally finite and that $S$ is a
$\Ccr$-colored graded subcollection of $\Cca$ such that $S(1)$ finitely
factorizes $\Cca$. For any element $x$ of $\Cca^S$, the colored
suboperad of $\Cca$ generated by $S$, the \Def{$S$-degree} of $x$ is
defined by
\begin{equation}
    \deg_S(x) := \max \Bra{\deg(\Tfr) : \Tfr \in \Free(S)
        \mbox{ and } \Eval_\Cca(\Tfr) = x}.
\end{equation}
Thanks to the fact that, by hypothesis, $x$ admits at least one
$S$-treelike expression and, by
Lemma~\ref{lem:finitely_factorizing_sets}, the fact that $x$ admits
finitely many $S$-treelike expressions, $\deg_S(x)$ is well-defined.
\medbreak

\subsubsection{Left expressions and hook-length formula}
Let $S$ be a $\Ccr$-colored graded subcollection of $\Cca$ and
$x \in \Cca$. An \Def{$S$-left expression} of $x$ is an expression for
$x$ in $\Cca$ of the form
\begin{equation} \label{equ:left_expression}
    x = \Par{\dots \Par{\Par{\Unit_{\Out(x)} \circ_1 s_1}
        \circ_{i_1} s_2} \circ_{i_2} \dots} \circ_{i_{\ell - 1}} s_\ell
\end{equation}
where $s_1, \dots, s_\ell \in S$ and $i_1, \dots, i_{\ell - 1} \in \N$.
Besides, if $\Tfr$ is an $S$-treelike expression of $x$ such that
$\INodes(\Tfr) = \Bra{e_1, \dots, e_\ell}$, a sequence
$\Par{e_1, \dots, e_\ell}$ is a \Def{linear extension} of $\Tfr$ if the
sequence is a linear extension of the poset induced by $\Tfr$ seen as an
Hasse diagram where the root of $\Tfr$ is the smallest element.
\medbreak

\begin{Lemma} \label{lem:left_expressions_linear_extensions}
    Let $\Cca$ be a locally finite $\Ccr$-colored operad and $S$ be a
    $\Ccr$-colored graded subcollection of $\Cca$. Then, for any
    $x \in \Cca$, the set of all $S$-left expressions of $x$ is in
    one-to-one correspondence with the set of all pairs $(\Tfr, e)$
    where $\Tfr$ is an $S$-treelike expression of $x$ and $e$ is a
    linear extension of~$\Tfr$.
\end{Lemma}
\begin{proof}
    Let $\phi_x$ be the map sending any $S$-left expression of the
    form~\eqref{equ:left_expression} of $x \in \Cca$ to the pair
    $(\Tfr, e)$ where $\Tfr$ is the colored syntax tree of $\Free(S)$
    obtained by interpreting~\eqref{equ:left_expression} in $\Free(S)$,
    {\em i.e.}, by replacing any $s_j$, $j \in [\ell]$,
    in~\eqref{equ:left_expression} by a corolla $\Sfr_j$ of $\Free(S)$
    labeled by $s_j$, and where $e$ is the sequence
    $\Par{e_1, \dots, e_\ell}$ of the internal nodes of $\Tfr$,
    where any $e_j$, $j \in [\ell]$, is the node of $\Tfr$ coming from
    $\Sfr_j$. We then have
    \begin{equation} \label{equ:left_expressions_linear_extensions_2}
        \Tfr =
        \Par{\dots \Par{\Par{\Unit_{\Out(x)} \circ_1 \Sfr_1}
        \circ_{i_1} \Sfr_2}
        \circ_{i_2} \dots} \circ_{i_{\ell - 1}} \Sfr_\ell
    \end{equation}
    and by construction, $\Tfr$ is an $S$-treelike expression of $x$.
    Moreover, immediately from the definition of the partial composition
    in free $\Ccr$-colored operads, $\Par{e_1, \dots, e_\ell}$ is a
    linear extension of $\Tfr$. Therefore, we have shown that $\phi_x$
    sends any $S$-left expression of $x$ to a pair $(\Tfr, e)$ where
    $\Tfr$ is an $S$-treelike expression of $x$ and $e$ is a linear
    extension of~$\Tfr$.
    \smallbreak

    Let $\Tfr$ be an $S$-treelike expression of $x \in \Cca$ and $e$ be
    a linear extension $\Par{e_1, \dots, e_\ell}$ of $\Tfr$. It follows
    by induction on the degree $\ell$ of $\Tfr$ that $\Tfr$ can be
    expressed by an expression of the
    form~\eqref{equ:left_expressions_linear_extensions_2} where any
    $e_j$, $j \in [\ell]$, is the node of $\Tfr$ coming from $\Sfr_j$.
    Now, the interpretation
    of~\eqref{equ:left_expressions_linear_extensions_2} in $\Cca$,
    {\em i.e.}, by replacing any corolla $\Sfr_j$, $j \in [\ell]$,
    in~\eqref{equ:left_expressions_linear_extensions_2} by its label
    $s_j$, is an $S$-left expression of the
    form~\eqref{equ:left_expression} for $x$.
    Since~\eqref{equ:left_expression} is the only antecedent of
    $(\Tfr, e)$ by $\phi_x$, it follows that $\phi_x$, with domain the
    set of all $S$-left expressions of $x$ and with codomain the set of
    all pairs $(\Tfr, e)$ where $\Tfr$ is an $S$-treelike expression of
    $x$ and $e$ is a linear extension of $\Tfr$, is a bijection.
\end{proof}
\medbreak

A famous result of Knuth~\cite{Knu98}, known as the
\Def{hook-length formula for trees}, stated here in our setting, says
that given a $\Ccr$-colored syntax tree $\Tfr$, the number of linear
extensions of $\Tfr$ is
\begin{equation} \label{equ:hook_length_syntax_tree}
    \frac{\deg(\Tfr)!}
        {\prod\limits_{v \in \INodes(\Tfr)} \deg\Par{\Tfr_v}}.
\end{equation}
When $S(1)$ finitely factorizes $\Cca$, by
Lemma~\ref{lem:finitely_factorizing_sets}, the number of $S$-treelike
expressions for any $x \in \Cca$ is finite. Hence, in this case, we
deduce from Lemma~\ref{lem:left_expressions_linear_extensions}
and~\eqref{equ:hook_length_syntax_tree} that the number of $S$-left
expressions of $x$ is
\begin{equation} \label{equ:number_left_expressions}
    \sum_{\substack{
        \Tfr \in \Free(S) \\
        \Eval_\Cca(\Tfr) = x
    }}
    \frac{\deg(\Tfr)!}
        {\prod\limits_{v \in \INodes(\Tfr)} \deg\Par{\Tfr_v}}.
\end{equation}
\medbreak

\subsection{Bud operads}
Let us now present a simple construction producing colored operads from
operads.
\medbreak

\subsubsection{From monochrome operads to colored operads}%
\label{subsubsec:operad_to_colored_operads}
If $\Oca$ is a monochrome operad and $\Ccr$ is a finite set of colors,
we denote by $\Bud_\Ccr(\Oca)$ the $\Ccr$-colored graded collection
defined by
\begin{equation}
    \Bud_\Ccr(\Oca)(n) := \Ccr \times \Oca(n) \times \Ccr^n,
    \qquad n \geq 1,
\end{equation}
and for all $(a, x, u) \in \Bud_\Ccr(\Oca)$, $\Out((a, x, u)) := a$ and
$\In((a, x, u)) := u$. We endow $\Bud_\Ccr(\Oca)$ with the partially
defined partial composition $\circ_i$ satisfying, for all triples
$(a, x, u)$ and $(b, y, v)$ of $\Bud_\Ccr(\Oca)$, and $i \in [|x|]$ such
that $ \Out((b, y, v)) = \In_i((a, x, u))$,
\begin{equation}
    (a, x, u) \circ_i (b, y, v)
    := \Par{a, x \circ_i y, u \mapsfrom_i v},
\end{equation}
where $u \mapsfrom_i v$ is the word obtained by replacing the $i$-th
letter of $u$ by $v$. Besides, if $\Oca_1$ and $\Oca_2$ are two operads
and $\phi : \Oca_1 \to \Oca_2$ is an operad morphism, we denote by
$\Bud_\Ccr(\phi)$ the map
\begin{equation}
    \Bud_\Ccr(\phi) : \Bud_\Ccr\Par{\Oca_1} \to \Bud_\Ccr\Par{\Oca_2}
\end{equation}
defined by
\begin{equation}
    \Bud_\Ccr(\phi)((a, x, u)) := \Par{a, \phi(x), u}.
\end{equation}
\medbreak

\begin{Proposition} \label{prop:functor_bud_operads}
    For any set of colors $\Ccr$, the construction
    \begin{math}
        (\Oca, \phi) \mapsto
        \Par{\Bud_\Ccr(\Oca), \Bud_\Ccr(\phi)}
    \end{math}
    is a functor from the category of monochrome operads to the category
    of $\Ccr$-colored operads.
\end{Proposition}
\medbreak

We omit the proof of Proposition~\ref{prop:functor_bud_operads} since it
is very straightforward. This result shows that $\Bud_\Ccr$ is a
functorial construction producing colored operads from monochrome ones.
We call $\Bud_\Ccr(\Oca)$ the \Def{$\Ccr$-bud operad}
of~$\Oca$%
\footnote{See examples of monochrome operads and their bud operads in
Section~\ref{subsec:examples_monochrome_operads}.}%
.
When $\Ccr$ is a singleton, $\Bud_\Ccr(\Oca)$ is by definition a
monochrome operad isomorphic to~$\Oca$. For this reason, in this case,
we identify $\Bud_\Ccr(\Oca)$ with~$\Oca$.
\medbreak

As a side observation, remark that in general, the bud operad
$\Bud_\Ccr(\Oca)$ of a free operad $\Oca$ is not a free $\Ccr$-colored
operad. Indeed, consider for instance the bud operad
$\Bud_{\{1, 2\}}(\Oca)$, where $\Oca := \Free(C)$ and $C$ is the
monochrome graded collection defined by $C := C(1) := \{\Att\}$. Then, a
minimal generating set of $\Bud_{\{1, 2\}}(\Oca)$ is
\begin{equation}
    \Bra{
    \Par{1, \FreeCorollaOne{\Att}, 1},
    \Par{1, \FreeCorollaOne{\Att}, 2},
    \Par{2, \FreeCorollaOne{\Att}, 1},
    \Par{2, \FreeCorollaOne{\Att}, 2}
    }.
\end{equation}
These elements are subjected to the nontrivial relations
\begin{equation}
    \Par{d, \FreeCorollaOne{\Att}, 1}
    \circ_1
    \Par{1, \FreeCorollaOne{\Att}, e}
    =
    \Par{d,
    \begin{tikzpicture}[xscale=.13,yscale=.15,Centering]
        \node(0)at(1.00,-4.5){};
        \node[NodeST](1)at(1.00,0.00){\begin{math}\Att\end{math}};
        \node[NodeST](2)at(1.00,-2){\begin{math}\Att\end{math}};
        \draw[Edge](1)--(2);
        \draw[Edge](0)--(2);
        \node(r)at(1.00,2.5){};
        \draw[Edge](r)--(1);
    \end{tikzpicture}\,,
    e}
    =
    \Par{d, \FreeCorollaOne{\Att}, 2}
    \circ_1
    \Par{2, \FreeCorollaOne{\Att}, e},
\end{equation}
where $d, e \in \{1, 2\}$, implying that $\Bud_{\{1, 2\}}(\Oca)$ is not
free.
\medbreak

\subsubsection{The associative operad}
The \Def{associative operad} $\As$ is the monochrome operad defined by
$\As(n) := \Bra{\OpAs_n}$, $n \geq 1$, and wherein partial composition
maps are defined by
\begin{equation}
    \OpAs_n \circ_i \OpAs_m := \OpAs_{n + m - 1},
    \qquad 1 \leq i \leq n, \enspace 1 \leq m.
\end{equation}
For any set of colors $\Ccr$, the bud operad $\Bud_\Ccr(\As)$ is the
set of all triples
\begin{equation}
    \Par{a, \OpAs_n, u_1 \dots u_n}
\end{equation}
where $a \in \Ccr$ and $u_1, \dots, u_n \in \Ccr$. For
$\Ccr := \{1, 2, 3\}$, one has for instance the partial composition
\begin{equation}
    \Par{2, \OpAs_4, \ColA{3{\bf 1}12}}
    \circ_2
    \Par{1, \OpAs_3, \ColD{233}} =
    \Par{2, \OpAs_6, \ColA{3}\ColD{233} \ColA{12}}.
\end{equation}
\medbreak

The associative operad and its bud operads will play an important role
in the sequel. For this reason, to gain readability, we shall simply
denote by $(a, u)$ any element $\Par{a, \OpAs_{|u|}, u}$ of
$\Bud_\Ccr(\As)$ without any loss of information.
\medbreak

\subsubsection{Pruning map}
Here, we use the fact that any monochrome operad $\Oca$ can be seen as
a $\Ccr$-colored operad where all output and input colors of its
elements are equal to $c_1$, where $c_1$ is the first color of $\Ccr$
(see Section~\ref{subsubsec:colored_operads}). Let
\begin{equation}
    \Prune : \Bud_\Ccr(\Oca) \to \Oca
\end{equation}
be the morphism of $\Ccr$-colored operads defined, for any
$(a, x, u) \in \Bud_\Ccr(\Oca)$, by
\begin{equation}
    \Prune((a, x, u)) := x.
\end{equation}
We call $\Prune$ the \Def{pruning map} on $\Bud_\Ccr(\Oca)$.
\medbreak

\section{Bud generating systems and combinatorial generation}%
\label{sec:bud_generating_systems}
In this section, we introduce bud generating systems. A bud generating
system relies on an operad $\Oca$, a set of colors $\Ccr$, and the bud
operad $\Bud_\Ccr(\Oca)$. The principal interest of these objects is
that they allow us to specify sets of objects of $\Bud_\Ccr(\Oca)$. We
shall also establish some first properties of bud generating systems by
showing that they can emulate context-free grammars, regular tree
grammars, and synchronous grammars.
\medbreak

\subsection{Bud generating systems}
We introduce here the main definitions and the main tools about bud
generating systems.
\medbreak

\subsubsection{Bud generating systems}
A \Def{bud generating system} is a tuple
$\Bca := (\Oca, \Ccr, \Rfr, I, T)$ where $\Oca$ is an operad called
\Def{ground operad}, $\Ccr$ is a finite set of colors, $\Rfr$ is a
finite $\Ccr$-colored graded subcollection of $\Bud_\Ccr(\Oca)$ called
\Def{set of rules}, $I$ is a subset of $\Ccr$ called \Def{set of initial
colors}, and $T$ is a subset of $\Ccr$ called \Def{set of terminal
colors}.
\medbreak

A \Def{monochrome bud generating system} is a bud generating system
whose set $\Ccr$ of colors contains a single color, and whose sets of
initial and terminal colors are equal to $\Ccr$. In this case, as
explained in Section~\ref{subsubsec:operad_to_colored_operads},
$\Bud_\Ccr(\Oca)$ and $\Oca$ are identified. These particular
generating systems are thus simply denoted by pairs $(\Oca, \Rfr)$.
\medbreak

Let us explain how bud generating systems specify, in two different
ways, two $\Ccr$-colored graded subcollections of $\Bud_\Ccr(\Oca)$. In
what follows, $\Bca := (\Oca, \Ccr, \Rfr, I, T)$ is a bud generating
system.
\medbreak

\subsubsection{Generation} \label{subsubsec:generation}
We say that $x_2 \in \Bud_\Ccr(\Oca)$ is \Def{derivable in one step}
from $x_1 \in \Bud_\Ccr(\Oca)$ if there is a rule $r \in \Rfr$ and an
integer $i$ such that $x_2 = x_1 \circ_i r$. We denote this property by
$x_1 \Deriv x_2$. When $x_1, x_2 \in \Bud_\Ccr(\Oca)$ are such that
$x_1 = x_2$ or there are $y_1, \dots, y_{\ell - 1} \in \Bud_\Ccr(\Oca)$,
$\ell \geq 1$, satisfying
\begin{equation}
    x_1 \Deriv y_1 \Deriv \cdots \Deriv y_{\ell - 1} \Deriv x_2,
\end{equation}
we say that $x_2$ is \Def{derivable} from $x_1$. Moreover, $\Bca$
\Def{generates} $x \in \Bud_\Ccr(\Oca)$ if there is a color $a$ of $I$
such that $x$ is derivable from $\Unit_a$ and all colors of $\In(x)$
are in $T$. The \Def{language} $\Lang(\Bca)$ of $\Bca$ is the set of all
the elements of $\Bud_\Ccr(\Oca)$ generated by~$\Bca$.
\medbreak

The \Def{derivation graph} of $\Bca$ is the oriented multigraph
$\DerivGraph(\Bca)$ with the set of elements derivable from $\Unit_a$,
$a \in I$, as set of vertices. In $\DerivGraph(\Bca)$, for any
$x_1, x_2 \in \Lang(\Bca)$ such that $x_1 \Deriv x_2$, there are $\ell$
edges from $x_1$ to $x_2$, where $\ell$ is the number of pairs
$(i, r) \in \N \times \Rfr$ such that $x_2 = x_1 \circ_i r$%
\footnote{See examples of bud generating systems and derivation
graphs in Sections~\ref{subsubsec:bud_generating_system_dias_gamma},
\ref{subsubsec:example_bmotz}, and~\ref{subsubsec:example_bubtree}.}%
.
\medbreak

\subsubsection{Synchronous generation}%
\label{subsubsec:synchronous_generation}
We say that $x_2 \in \Bud_\Ccr(\Oca)$ is \Def{synchronously derivable
in one step} from $x_1 \in \Bud_\Ccr(\Oca)$ if there are rules $r_1$,
\dots, $r_{|x_1|}$ of $\Rfr$ such that
\begin{math}
    x_2 = x_1 \circ \Han{r_1, \dots, r_{|x_1|}}.
\end{math}
We denote this
property by $x_1 \SyncDeriv x_2$. When $x_1, x_2 \in \Bud_\Ccr(\Oca)$
are such that $x_1 = x_2$ or there are
\begin{math}
    y_1, \dots, y_{\ell - 1} \in \Bud_\Ccr(\Oca)$, $\ell \geq 1,
\end{math}
satisfying
\begin{equation}
    x_1 \SyncDeriv y_1 \SyncDeriv \cdots
    \SyncDeriv y_{\ell - 1} \SyncDeriv x_2,
\end{equation}
we say that $x_2$ is \Def{synchronously derivable} from $x_1$. Moreover,
$\Bca$ \Def{synchronously generates} $x \in \Bud_\Ccr(\Oca)$ if there
is a color $a$ of $I$ such that $x$ is synchronously derivable from
$\Unit_a$ and all colors of $\In(x)$ are in $T$. The \Def{synchronous
language} $\SyncLang(\Bca)$ of $\Bca$ is the set of all the elements of
$\Bud_\Ccr(\Oca)$ synchronously generated by~$\Bca$.
\medbreak

The \Def{synchronous derivation graph} of $\Bca$ is the oriented
multigraph $\SyncDerivGraph(\Bca)$ with the set of elements
synchronously derivable from $\Unit_a$, $a \in I$, as set of vertices.
In $\SyncDerivGraph(\Bca)$, for any $x_1, x_2 \in \SyncLang(\Bca)$ such
that $x_1 \SyncDeriv x_2$, there are $\ell$ edges from $x_1$ to $x_2$,
where $\ell$ is the number of tuples
\begin{math}
    \Par{r_1, \dots, r_{|x_1|}} \in \Rfr^{|x_1|}
\end{math}
such that
\begin{math}
    x_2 = x_1 \circ \Han{r_1, \dots, r_{|x_1|}}.
\end{math}%
\footnote{See examples of bud generating systems and synchronous
derivation graphs in Sections~\ref{subsubsec:example_bbtree}
and~\ref{subsubsec:example_bbaltree}.}%
.
\medbreak

\subsection{First properties}
We state now two properties about the languages and the synchronous
languages of bud generating systems.
\medbreak

\begin{Lemma} \label{lem:language_treelike_expressions}
    Let $\Bca := (\Oca, \Ccr, \Rfr, I, T)$ be a bud generating system.
    Then, for any $x \in \Bud_\Ccr(\Oca)$, $x$ belongs to $\Lang(\Bca)$
    if and only if $x$ admits an $\Rfr$-treelike expression with output
    color in $I$ and all input colors in~$T$.
\end{Lemma}
\begin{proof}
    Assume that $x$ belongs to $\Lang(\Bca)$. Then, by definition of the
    derivation relation $\Deriv$, $x$ admits an $\Rfr$-left expression.
    Lemma~\ref{lem:left_expressions_linear_extensions} implies in
    particular that $x$ admits an $\Rfr$-treelike expression $\Tfr$.
    Moreover, since $\Tfr$ is a treelike expression for $x$, $\Tfr$ has
    the same output and input colors as those of $x$. Hence, because $x$
    belongs to $\Lang(\Bca)$, its output color is in $I$ and all its
    input colors are in $T$. Thus, $\Tfr$ satisfies the required
    properties.
    \smallbreak

    Conversely, assume that $x$ is an element of $\Bud_\Ccr(\Oca)$
    admitting an $\Rfr$-treelike expression $\Tfr$ with output color in
    $I$ and all input colors in $T$.
    Lemma~\ref{lem:left_expressions_linear_extensions} implies in
    particular that $x$ admits an $\Rfr$-left expression. Hence, by
    definition of the derivation relation $\Deriv$, $x$ is derivable
    from $\Unit_{\Out(x)}$ and all its input colors are in $T$.
    Therefore, $x$ belongs to~$\Lang(\Bca)$.
\end{proof}
\medbreak

\begin{Lemma} \label{lem:synchronous_language_treelike_expressions}
    Let $\Bca := (\Oca, \Ccr, \Rfr, I, T)$ be a bud generating system.
    Then, for any $x \in \Bud_\Ccr(\Oca)$, $x$ belongs to
    $\SyncLang(\Bca)$ if and only if $x$ admits an $\Rfr$-treelike
    expression with output color in $I$ and all input colors in $T$ and
    which is a perfect tree.
\end{Lemma}
\begin{proof}
    The proof of the statement of the lemma is very similar to the one
    of Lemma~\ref{lem:language_treelike_expressions}. The only
    difference lies on the fact that the definition of synchronous
    languages uses the complete composition map $\circ$ instead of
    partial composition maps $\circ_i$, intervening in the definition of
    languages. Hence, in this context, $\Rfr$-treelike expressions are
    perfect trees.
\end{proof}
\medbreak

\begin{Proposition} \label{prop:language_sub_operad_generated}
    Let $\Bca := (\Oca, \Ccr, \Rfr, I, T)$ be a bud generating system.
    Then, the language of $\Bca$ satisfies
    \begin{equation}
        \Lang(\Bca) =
        \Bra{x \in \Bud_\Ccr(\Oca)^\Rfr :
            \Out(x) \in I \mbox{ and } \In(x) \in T^+},
    \end{equation}
    where $\Bud_\Ccr(\Oca)^\Rfr$ is the colored suboperad of
    $\Bud_\Ccr(\Oca)$ generated by $\Rfr$.
\end{Proposition}
\begin{proof}
    By definition of suboperads generated by a set, as a $\Ccr$-colored
    graded collection, $\Bud_\Ccr(\Oca)^\Rfr$ consists in all the
    elements obtained by evaluating in $\Bud_\Ccr(\Oca)$ all
    $\Ccr$-colored $\Rfr$-syntax trees. Therefore, the statement of the
    proposition is a consequence of
    Lemma~\ref{lem:language_treelike_expressions}.
\end{proof}
\medbreak

\begin{Proposition} \label{prop:synchronous_language_subset_language}
    Let $\Bca := (\Oca, \Ccr, \Rfr, I, T)$ be a bud generating system.
    Then, the synchronous language of $\Bca$ is a subset of the
    language of~$\Bca$. Moreover, when $\Rfr$ contains all the colored
    units of $\Bud_\Ccr(\Oca)$, these two languages are equal.
\end{Proposition}
\begin{proof}
    By Lemma~\ref{lem:language_treelike_expressions} the language of
    $\Bca$ is the set of the elements obtained by evaluating in
    $\Bud_\Ccr(\Oca)$ all $\Ccr$-colored $\Rfr$-syntax trees satisfying
    some conditions for their output and input colors.
    Lemma~\ref{lem:synchronous_language_treelike_expressions} says
    that the synchronous language of $\Bca$ is the set of the elements
    obtained by evaluating in $\Bud_\Ccr(\Oca)$ some $\Ccr$-colored
    $\Rfr$-syntax trees satisfying at least the previous conditions.
    Hence, this implies the statement of the proposition.
    \smallbreak

    The second part of the proposition follows from the fact that, if
    $x_1 \Deriv x_2$ for two elements $x_1$ and $x_2$ of
    $\Bud_\Ccr(\Oca)$, there is by definition $r \in \Rfr$ and an
    integer $i$ such that $x_2 = x_1 \circ_i r$. Then, one has
    \begin{equation}
        x_2 = x_1 \circ
        \Han{
            \Unit_{\In_1\Par{x_1}}, \dots, \Unit_{\In_{i - 1}\Par{x_1}},
            r,
            \Unit_{\In_{i + 1}\Par{x_1}}, \dots,
            \Unit_{\In_{|x_1|}\Par{x_1}}
        }.
    \end{equation}
    Since by hypothesis, all the colored units of
    $\Bud_\Ccr(\Oca)$ are in $\Rfr$, this implies $x_1 \SyncDeriv x_2$.
    Hence, as binary relations, $\Deriv$ and $\SyncDeriv$ are equal,
    establishing the second part of the statement of the proposition.
\end{proof}
\medbreak

\subsection{Links with other generating systems}%
\label{subsec:other_generating_systems}
Context-free grammars, regular tree grammars, and synchronous grammars
are already existing generating systems describing sets of words for the
first, and sets of trees for the last two. We show here that any of
these grammars can be emulated by bud generating systems. In the first
case, context-free grammars are emulated by bud generating systems
with the associative operad $\As$ as ground operad, and in the second
and third cases, regular tree grammars and synchronous grammars are
emulated by bud generating systems with free operads $\Free(C)$ as
ground operads, where $C$ are suitable set of generators.
\medbreak

\subsubsection{Context-free grammars}
A \Def{context-free grammar}~\cite{Har78,HMU06} is a tuple
$\Gca := (V, T, P, s)$ where $V$ is a finite alphabet of
\Def{variables}, $T$ is a finite alphabet of \Def{terminal symbols},
$P$ is a finite subset of $V \times (V \sqcup T)^*$ called \Def{set of
productions}, and $s$ is a variable of $V$ called \Def{start symbol}. If
$x_1$ and $x_2$ are two words of $(V \sqcup T)^*$, $x_2$ is
\Def{derivable in one step} from $x_1$ if $x_1$ is of the form
$x_1 = u a v$ and $x_2$ is of the form $x_2 = u w v$ where
$u, v \in (V \sqcup T)^*$ and $(a, w)$ is a production of $P$. This
property is denoted by $x_1 \Deriv x_2$, so that $\Deriv$ is a binary
relation on $(V \sqcup T)^*$. The reflexive and transitive closure of
$\Deriv$ is the \Def{derivation relation}. A word $x \in T^*$ is
\Def{generated} by $\Gca$ if $x$ is derivable from the word~$s$. The
\Def{language} of $\Gca$ is the set of all words generated by $\Gca$.
We say that $\Gca$ is \Def{proper} if, for any $(a, w) \in P$, $w$ is
not the empty word.
\medbreak

If $\Gca := (V, T, P, s)$ is a proper context-free grammar, we denote
by $\CFG(\Gca)$ the bud generating system
\begin{equation}
    \CFG(\Gca) := (\As, V \sqcup T, \Rfr, \{s\}, T)
\end{equation}
wherein $\Rfr$ is the set of rules
\begin{equation}
    \Rfr := \Bra{(a, u) \in \Bud_{V \sqcup T}(\As) : (a, u) \in P}.
\end{equation}
\medbreak

\begin{Proposition} \label{prop:emulation_grammars}
    Let $\Gca$ be a proper context-free grammar. Then, the restriction
    of the map $\In$, sending any $(a, u) \in \Bud_{V \sqcup T}(\As)$
    to $u$, on the domain $\Lang(\CFG(\Gca))$ is a bijection between
    $\Lang(\CFG(\Gca))$ and the language of~$\Gca$.
\end{Proposition}
\begin{proof}
    Let us denote by $V$ the set of variables, by $T$ the set of
    terminal symbols, by $P$ the set of productions, and by $s$ the
    start symbol of $\Gca$.
    \smallbreak

    Let $\Par{a, x} \in \Bud_{V \sqcup T}(\As)$, $\ell \geq 1$, and
    $y_1, \dots, y_{\ell - 1} \in (V \sqcup T)^*$. Then, by definition
    of $\CFG$, there are in $\CFG(\Gca)$ the derivations
    \begin{equation}
        \Unit_s \Deriv \Par{s, y_1} \Deriv \cdots
        \Deriv \Par{s, y_{\ell - 1}}
        \Deriv \Par{a, x}
    \end{equation}
    if and only if $a = s$ and there are in $\Gca$ the derivations
    \begin{equation}
        s \Deriv y_1 \Deriv \cdots \Deriv y_{\ell - 1} \Deriv x.
    \end{equation}
    Then, $\Par{a, x}$ belongs to $\Lang(\CFG(\Gca))$ if and only if
    $a = s$ and $x$ belongs to the language of $\Gca$. The fact that
    $\In\Par{\Par{s, x}} = x$ completes the proof.
\end{proof}
\medbreak

\subsubsection{Regular tree grammars}
Let $V := V(0)$ be a finite graded alphabet of \Def{variables} and
$T := \bigsqcup_{n \geq 0} T(n)$ be a finite graded alphabet of
\Def{terminal symbols}. For any $n \geq 0$ and $a \in T(n)$, the arity
$|a|$ of $a$ is $n$. The tuple $(V, T)$ is called a \Def{signature}.
\medbreak

A \Def{$(V, T)$-tree} is an element of
$\Bud_{V \sqcup T(0)}(\Free(T \setminus T(0)))$, where
$T \setminus T(0)$ is seen as a monochrome graded collection. In other
words, a $(V, T)$-tree is a planar rooted tree $\Tfr$ such that, for any
$n \geq 1$, any internal node of $\Tfr$ having $n$ children is labeled
by an element of arity $n$ of $T$, and the output and all leaves of
$\Tfr$ are labeled on $V \sqcup T(0)$.
\medbreak

A \Def{regular tree grammar}~\cite{GS84,CDGJLLTT07} is a tuple
$\Gca := (V, T, P, s)$ where $(V, T)$ is a signature, $P$ is a set of
pairs of the form $(v, \Sfr)$ called \Def{productions} where $v \in V$
and $\Sfr$ is a $(V, T)$-tree, and $s$ is a variable of $V$ called
\Def{start symbol}. If $\Tfr_1$ and $\Tfr_2$ are two $(V, T)$-trees,
$\Tfr_2$ is \Def{derivable in one step} from $\Tfr_1$ if $\Tfr_1$ has a
leaf $y$ labeled by $a$ and the tree obtained by replacing $y$ by the
root of $\Sfr$ in $\Tfr_1$ is $\Tfr_2$, provided that $(a, \Sfr)$ is a
production of $P$. This property is denoted by $\Tfr_1 \Deriv \Tfr_2$,
so that $\Deriv$ is a binary relation on the set of all $(V, T)$-trees.
The reflexive and transitive closure of $\Deriv$ is the derivation
relation. A $(V, T)$-tree $\Tfr$ is \Def{generated} by $\Gca$ if $\Tfr$
is derivable from the tree $\Unit_s$ consisting in one leaf labeled by
$s$ and all leaves of $\Tfr$ are labeled on $T(0)$. The \Def{language}
of $\Gca$ is the set of all $(V, T)$-trees generated by~$\Gca$.
\medbreak

If $\Gca := (V, T, P, s)$ is a regular tree grammar, we denote by
$\RTG(\Gca)$ the bud generating system
\begin{equation}
    \RTG(\Gca) :=
    \Par{\Free(T \setminus T(0)), V \sqcup T(0), \Rfr, \{s\}, T(0)}
\end{equation}
wherein $\Rfr$ is the set of rules
\begin{equation}
    \Rfr :=
    \Bra{(a, \Tfr, u) \in \Bud_{V \sqcup T(0)}(\Free(T \setminus T(0)))
    : \Par{a, \Tfr_{a, u}} \in P},
\end{equation}
where, for any $\Tfr \in \Free(T \setminus T(0))$,
$a \in V \sqcup T(0)$, and $u \in (V \sqcup T(0))^{|\Tfr|}$,
$\Tfr_{a, u}$ is the $(V, T)$-tree obtained by labeling the output of
$\Tfr$ by $a$ and by labeling from left to right the leaves of $\Tfr$ by
the letters of~$u$.
\medbreak

\begin{Proposition} \label{prop:emulation_tree_grammars}
    Let $\Gca$ be a regular tree grammar. Then, the map
    $\phi : \Lang(\RTG(\Gca)) \to L$ defined by
    $\phi((a, \Tfr, u)) := \Tfr_{a, u}$ is a bijection between the
    language of $\RTG(\Gca)$ and the language $L$ of~$\Gca$.
\end{Proposition}
\begin{proof}
    Let us denote by $(V, T)$ the underlying signature and by $s$ the
    start symbol of~$\Gca$.
    \smallbreak

    Let
    $(a, \Tfr, u) \in \Bud_{V \sqcup T(0)}(\Free(T \setminus T(0)))$,
    $\ell \geq 1$, and
    $\Sfr^{(1)}, \dots, \Sfr^{(\ell - 1)} \in \Free(T \setminus T(0))$,
    and $v_1, \dots, v_{\ell - 1} \in (V \sqcup T(0))^+$. Then, by
    definition of $\RTG$, there are in $\RTG(\Gca)$ the derivations
    \begin{equation}
        \Unit_s \Deriv \Par{s, \Sfr^{(1)}, v_1}
        \Deriv \cdots \Deriv
        \Par{s, \Sfr^{(\ell - 1)}, v_{\ell - 1}}
        \Deriv (a, \Tfr, u)
    \end{equation}
    if and only if $a = s$ and there are in $\Gca$ the derivations
    \begin{equation}
        \Unit_s \Deriv {\Sfr^{(1)}}_{s, v_1} \Deriv \cdots \Deriv
        {\Sfr^{(\ell - 1)}}_{s, v_{\ell - 1}}
        \Deriv
        \Tfr_{a, u}.
    \end{equation}
    Then, $(a, \Tfr, u)$ belongs to $\Lang(\RTG(\Gca))$ if and only if
    $a = s$ and $\Tfr_{a, u}$ belongs to the language of $\Gca$. The
    fact that $\phi((a, \Tfr, u)) = \Tfr_{a, u}$ completes the proof.
\end{proof}
\medbreak

\subsubsection{Synchronous grammars}
\label{subsubsec:synchronous_grammars}
In this section, we shall denote by $\Tree$ the monochrome operad
defined as the free operad generated by one operation $\Att_n$ of arity
$n$ for all $n \geq 1$. The elements of this operad are planar rooted
trees where internal nodes have an arbitrary arity. Observe by the way
that $\Tree$ is not locally finite.
\medbreak

Let $B$ be a finite alphabet. A \Def{$B$-bud tree} is an element of
$\Bud_B(\Tree)$. In other words, a $B$-bud tree is a planar rooted tree
$\Tfr$ such that the output and all leaves of $\Tfr$ are labeled on $B$.
The leaves of a $B$-bud tree are indexed from $1$ from left to right.
\medbreak

A \Def{synchronous grammar}~\cite{Gir12} is a tuple $\Gca := (B, a, R)$
where $B$ is a finite alphabet of \Def{bud labels}, $a$ is an element of
$B$ called \Def{axiom}, and $R$ is a finite set of pairs of the form
$(b, \Sfr)$ called \Def{substitution rules} where $b \in B$ and $\Sfr$
is a $B$-bud tree. If $\Tfr_1$ and $\Tfr_2$ are two $B$-bud trees such
that $\Tfr_1$ is of arity $n$, $\Tfr_2$ is \Def{derivable in one step}
from $\Tfr_1$ if there are substitution rules
$\Par{b_1, \Sfr_1}, \dots, \Par{b_n, \Sfr_n}$ of $R$ such that for all
$i \in [n]$, the $i$-th leaf of $\Tfr_1$ is labeled by $b_i$ and
$\Tfr_2$ is obtained by replacing the $i$-th leaf of $\Tfr_1$ by
$\Sfr_i$ for all $i \in [n]$. This property is denoted by
$\Tfr_1 \SyncDeriv \Tfr_2$, so that $\SyncDeriv$ is a binary relation on
the set of all $B$-bud trees. The reflexive and transitive closure of
$\SyncDeriv$ is the derivation relation. A $B$-bud tree $\Tfr$ is
\Def{generated} by $\Gca$ if $\Tfr$ is derivable from the tree $\Unit_a$
consisting is one leaf labeled by $a$. The \Def{language} of $\Gca$ is
the set of all $B$-bud trees generated by~$\Gca$.
\medbreak

If $\Gca := (B, a, R)$ is a synchronous grammar, we denote by
$\SG(\Gca)$ the bud generating system
\begin{equation}
    \SG(\Gca) := (\Tree, B, \Rfr, \{a\}, B)
\end{equation}
wherein $\Rfr$ is the set of rules
\begin{equation}
    \Rfr :=
    \Bra{(b, \Tfr, u) \in \Bud_B(\Tree) : \Par{b, \Tfr_{b, u}} \in R},
\end{equation}
where, for any $\Tfr \in \Bud_B(\Tree)$, $b \in B$, and $u \in B^+$,
$\Tfr_{b, u}$ is the $B$-bud tree obtained by labeling the output of
$\Tfr$ by $b$ and by labeling from left to right the leaves of $\Tfr$ by
the letters of~$u$.
\medbreak

\begin{Proposition} \label{prop:emulation_synchronous_grammars}
    Let $\Gca$ be a synchronous grammar. Then, the map
    $\phi : \SyncLang(\SG(\Gca)) \to L$ defined by
    $\phi((b, \Tfr, u)) := \Tfr_{b, u}$ is a bijection between the
    synchronous language of $\SG(\Gca)$ and the language $L$ of~$\Gca$.
\end{Proposition}
\begin{proof}
    Let us denote by $B$ the set of bud labels and by $a$ the axiom
    of~$\Gca$.
    \smallbreak

    Let $(b, \Tfr, u) \in \Bud_B(\Tree)$, $\ell \geq 1$, and
    $\Sfr^{(1)}, \dots, \Sfr^{(\ell - 1)} \in \Tree$, and
    $v_1, \dots, v_{\ell - 1} \in B^+$. Then, by definition of $\SG$,
    there are in $\SG(\Gca)$ the synchronous derivations
    \begin{equation}
        \Unit_a \SyncDeriv \Par{a, \Sfr^{(1)}, v_1}
        \SyncDeriv \cdots \SyncDeriv
        \Par{a, \Sfr^{(\ell - 1)}, v_{\ell - 1}}
        \SyncDeriv (b, \Tfr, u)
    \end{equation}
    if and only if $b = a$ and there are in $\Gca$ the derivations
    \begin{equation}
        \Unit_a \SyncDeriv {\Sfr^{(1)}}_{a, v_1}
        \SyncDeriv \cdots \SyncDeriv
        {\Sfr^{(\ell - 1)}}_{a, v_{\ell - 1}}
        \SyncDeriv
        \Tfr_{b, u}.
    \end{equation}
    Then, $(a, \Tfr, u)$ belongs to $\SyncLang(\SG(\Gca))$ if and only
    if $b = a$ and $\Tfr_{b, u}$ belongs to the language of $\Gca$. The
    fact that $\phi((b, \Tfr, u)) = \Tfr_{b, u}$ completes the proof.
\end{proof}
\medbreak

\section{Series on colored operads and bud generating systems}
\label{sec:series_colored_operads}
A very normal combinatorial question consists, given a bud generating
system $\Bca$, in computing the generating series
$\GenS_{\Lang(\Bca)}(t)$ and $\GenS_{\SyncLang(\Bca)}(t)$, respectively
counting the elements of the language and of the synchronous language of
$\Bca$ with respect to the arity of the elements. To achieve this
objective, we develop a new generalization of formal power series,
namely series on colored operads, and define several operation on these.
Any bud generating system $\Bca$ leads to the definition of three
series on colored operads: its hook generating series $\Hook(\Bca)$, its
syntactic generating series $\Synt(\Bca)$, and its synchronous
generating series $\Sync(\Bca)$. The hook generating series allows us to
define analogues of the hook-length statistic of binary trees for
objects belonging to the language of $\Bca$, possibly different than
trees. The syntactic (resp. synchronous) generating series leads to
obtain functional equations and recurrence formulas to compute the
coefficients of $\GenS_{\Lang(\Bca)}(t)$
and~$\GenS_{\SyncLang(\Bca)}(t)$.
\medbreak

From now, $\K$ is a field of characteristic zero. Moreover, in all this
section, $\Cca$ is a $\Ccr$-colored operad. Recall that the set $\Ccr$
of colors is always considered on the form $\Ccr = \{c_1, \dots, c_k\}$.
Besides, $\Bca := (\Oca, \Ccr, \Rfr, I, T)$ is a bud generating system.
\medbreak

\subsection{Series on colored operads}
We introduce here the main definitions about series on colored operads.
We also introduce specific notions about series on bud operads.
\medbreak

\subsubsection{Series on colored operads}
The linear span of the underlying set of $\Cca$ is denoted by
$\K \Angle{\Cca}$. Let $\K \AAngle{\Cca}$ be the dual space of
$\K \Angle{\Cca}$. By definition, the elements of $\K \AAngle{\Cca}$ are
maps
\begin{math}
    \Fbf : \Cca \to \K,
\end{math}
called \Def{$\Cca$-formal power series} (or \Def{$\Cca$-series} for
short). Let $\Fbf \in \K \AAngle{\Cca}$. The coefficient of any
$x \in \Cca$ in $\Fbf$ is denoted by $\Angle{x, \Fbf}$. The
\Def{support} of $\Fbf$ is the set
\begin{math}
    \Supp(\Fbf) := \Bra{x \in \Cca : \Angle{x, \Fbf} \ne 0}.
\end{math}
For any $\Ccr$-colored graded subcollection $S$ of $\Cca$, the
\Def{characteristic series} of $S$ is the $\Cca$-series $\Charac{S}$
defined for any $x \in \Cca$ by $\Angle{x, \Charac{S}} := 1$ if
$x \in S$, and by $\Angle{x, \Charac{S}} := 0$ otherwise. The
\Def{series of colored units} of $\K \AAngle{\Cca}$ is the series $\Ubf$
defined as the characteristic of $\Bra{\Unit_a : a \in \Ccr}$. This
series will play a special role in the sequel. Since $\Ccr$ is finite,
$\Ubf$ is a polynomial. By exploiting the vector space structure of
$\K \AAngle{\Cca}$, any $\Cca$-series $\Fbf$ expresses as
\begin{equation} \label{equ:definition_series_as_sums}
    \Fbf = \sum_{x \in \Cca} \Angle{x, \Fbf} x.
\end{equation}
This notation using potentially infinite sums of elements of $\Cca$
accompanied with coefficients of $\K$ is common in the context of formal
power series. In the sequel, we shall define and handle some
$\Cca$-series using the notation~\eqref{equ:definition_series_as_sums}.
\medbreak

Let us provide here some bibliographical information about various types
of series. Since the introduction of formal power series, a lot of
generalizations were proposed in order to extend the range of problems
they can help to solve. The most obvious ones are multivariate series
allowing us to count objects not only with respect to their sizes but
also with respect to various other statistics. Another one consists in
considering noncommutative series on words~\cite{Eil74,SS78,BR10}, or
even, pushing the generalization one step further, on elements of a
monoid~\cite{Sak09}. Besides, as another generalization, series on trees
have been considered~\cite{BR82,Boz01}. Series on (noncolored) operads
increase the list of these generalizations. Chapoton is the first to
have considered such series on operads~\cite{Cha02,Cha08,Cha09}. Several
authors have contributed to this field by considering slight variations
in the definitions of these series. Among these, one can cite van der
Laan~\cite{VDL04}, Frabetti~\cite{Fra08}, and Loday and
Nikolov~\cite{LN13}. Our notion of series on colored operads developed
here is a natural generalization of series on operads.
\medbreak

Observe that $\Cca$-series are defined here on fields $\K$ instead on
the much more general structures of semirings, as it is the case for
series on monoids. We choose to tolerate this loss of generality
because this considerably simplifies the theory. Furthermore, we shall
use in the sequel $\Cca$-series as devices for combinatorial
enumeration, so that it is sufficient to pick $\K$ as the field
$\Q(q_0, q_1, q_2, \dots)$ of rational functions in an infinite number
of commuting parameters with rational coefficients. The parameters
$q_0, q_1, q_2, \dots$ intervene in the enumeration of colored graded
subcollections of $\Cca$ with respect to several statistics%
\footnote{See examples of series on the bud operad of the operad $\Motz$
of Motzkin paths in
Section~\ref{subsubsec:example_series_motzkin_paths}.}%
.
\medbreak

\subsubsection{Colored operad morphisms and series}
If $\Cca_1$ and $\Cca_2$ are two $\Ccr$-colored operads and
$\phi : \Cca_1 \to \Cca_2$ is a morphism of colored operads,
$\bar{\phi}$ is the map
\begin{equation}
    \bar{\phi} : \K \AAngle{\Cca_1} \to \K \AAngle{\Cca_2}
\end{equation}
defined, for any $\Fbf \in \K \AAngle{\Cca_1}$ and $y \in \Cca_2$, by
\begin{equation} \label{equ:series_morphism}
    \Angle{y, \bar{\phi}(\Fbf)} :=
    \sum_{\substack{x \in \Cca_1 \\ \phi(x) = y}} \Angle{x, \Fbf}.
\end{equation}
\medbreak

Observe first that $\bar{\phi}$ is a linear map. Moreover, notice
that~\eqref{equ:series_morphism} could be undefined for arbitrary
colored operads $\Cca_1$ and $\Cca_2$, and an arbitrary morphism of
colored operads $\phi$. However, when all fibers of $\phi$ are finite,
for any $y \in \Cca_2$, the right member
of~\eqref{equ:series_morphism} is well-defined since the sum has a
finite number of terms.
\medbreak

\subsubsection{Pruned series and faithfulness}
\label{subsubsec:pruned_series}
Let $\Fbf$ be a $\Bud_\Ccr(\Oca)$-series. By a slight abuse of notation,
we denote by
\begin{equation}
    \Prune : \K \AAngle{\Bud_\Ccr(\Oca)} \to \K \AAngle{\Oca}
\end{equation}
the map $\bar{\Prune}$. Since $\Ccr$ is finite, the series
$\Prune(\Fbf)$ is well-defined and is called \Def{pruned series} of
$\Fbf$. Intuitively, the series $\Prune(\Fbf)$ can be seen as a version
of $\Fbf$ wherein the colors of the elements of its support are
forgotten%
\footnote{See an example of pruned series in
Section~\ref{subsubsec:example_series_motzkin_paths}.}%
. Besides, $\Fbf$ is said \Def{faithful} if all coefficients of
$\Prune(\Fbf)$ are equal to $0$ or to~$1$.
\medbreak

We say that $\Bca$ is \Def{faithful} (resp.
\Def{synchronously faithful}) if the characteristic series of
$\Lang(\Bca)$ (resp. $\SyncLang(\Bca)$) is faithful. Observe that all
monochrome bud generating systems are faithful (resp. synchronously
faithful). One of the reasons for requiring faithfulness (resp.
synchronous faithfulness) for bud generating systems appears when $\Bca$
is utilized for specifying objects of $\Oca$ by pruning the objects of
$\Lang(\Bca)$ (resp. $\SyncLang(\Bca)$). In this case, if $\Bca$ is not
faithful (resp. synchronously faithful), there would be several distinct
elements $(a, x, u)$ of $\Bud_\Ccr(\Oca)$ generated (resp. synchronously
generated) by $\Bca$ whose image by $\Prune$ is $x$. This could make
very hard the enumeration of the pruned version of the language (resp.
synchronous language) of~$\Bca$.
\medbreak

\subsubsection{Series of colors} \label{subsubsec:series_of_colors}
Let
\begin{equation}
    \Colors : \Cca \to \Bud_\Ccr(\As)
\end{equation}
be the morphism of colored operads defined for any $x \in \Cca$ by
\begin{equation}
    \Colors(x) := \Par{\Out(x), \In(x)}.
\end{equation}
By a slight abuse of notation, we denote by
\begin{equation}
    \Colors : \K \AAngle{\Cca} \to \K \AAngle{\Bud_\Ccr(\As)}
\end{equation}
the map $\bar{\Colors}$. If $\Fbf$ is a $\Cca$-series, we call
$\Colors(\Fbf)$ the \Def{series of colors} of $\Fbf$. Intuitively, the
series $\Colors(\Fbf)$ can be seen as a version of $\Fbf$ wherein only
the colors of the elements of its support are taken into account%
\footnote{See examples of series of colors in
Section~\ref{subsubsec:example_series_trees}.}%
.
\medbreak

\subsubsection{Series of color types}%
\label{subsubsec:series_of_color_types}
The \Def{$\Ccr$-type} of a word $u \in \Ccr^+$ is the word $\Type(u)$ of
$\N^k$ defined by
\begin{equation}
    \Type(u) := |u|_{c_1} \dots |u|_{c_k},
\end{equation}
where for any $a \in \Ccr$, $|u|_a$ is the number of occurrences of $a$
in~$u$. By extension, we shall call \Def{$\Ccr$-type} any word of $\N^k$
with at least a nonzero letter and we denote by $\Types_\Ccr$ the set of
all $\Ccr$-types. The \Def{degree} $\deg(\alpha)$ of
$\alpha \in \Types_\Ccr$ is the sum of the letters of $\alpha$. We
denote by $\Ccr^\alpha$ the word
\begin{math}
    c_1^{\alpha_1} \dots c_k^{\alpha_k}.
\end{math}
\medbreak

Assume that $\Zbb_\Ccr :=\Bra{\VarZ_{c_1}, \dots, \VarZ_{c_k}}$ is any
alphabet of commutative letters. For any type $\alpha$, we denote by
$\Zbb_\Ccr^\alpha$ the monomial
\begin{math}
    \VarZ_{c_1}^{\alpha_1} \dots \VarZ_{c_k}^{\alpha_k}
\end{math}
of $\K \Han{\Zbb_\Ccr}$. Moreover, for any two types $\alpha$ and
$\beta$, the \Def{sum} $\alpha \dot{+} \beta$ of $\alpha$ and $\beta$ is
the type satisfying $(\alpha \dot{+} \beta)_i := \alpha_i + \beta_i$ for
all $i \in [k]$. Observe that with this notation,
\begin{math}
    \Zbb_\Ccr^\alpha \Zbb_\Ccr^\beta = \Zbb_\Ccr^{\alpha \dot{+} \beta}.
\end{math}
\medbreak

Consider now the map
\begin{equation}
    \ColorTypes :
    \K \AAngle{\Cca} \to
    \K \Han{\Han{\Xbb_\Ccr \sqcup \Ybb_\Ccr}},
\end{equation}
defined for all $\alpha, \beta \in \Tca_\Ccr$ by
\begin{equation} \label{equ:definition_color_types}
    \Angle{\Xbb_\Ccr^\alpha \Ybb_\Ccr^\beta, \ColorTypes(\Fbf)}
    :=
    \sum_{\substack{
        (a, u) \in \Bud_\Ccr(\As) \\
        \Type(a) = \alpha \\
        \Type(u) =  \beta
    }}
    \Angle{(a, u), \Colors(\Fbf)}.
\end{equation}
By the definition of the map $\Colors$,
\begin{equation}
    \ColorTypes(\Fbf) =
    \sum_{x \in \Cca}
    \Angle{x, \Fbf}\enspace
    \Xbb_\Ccr^{\Type(\Out(x))}\enspace
    \Ybb_\Ccr^{\Type(\In(x))}.
\end{equation}
\medbreak

Observe that for all $\alpha, \beta \in \Types_\Ccr$ such that
$\deg(\alpha) \ne 1$, the coefficients of
$\Xbb_\Ccr^\alpha \Ybb_\Ccr^\beta$ in $\ColorTypes(\Fbf)$ are zero. In
intuitive terms, the series $\ColorTypes(\Fbf)$, called
\Def{series of color types} of $\Fbf$, can be seen as a version of
$\Colors(\Fbf)$ wherein only the output colors and the types of the
input colors of the elements of its support are taken into account, the
variables of $\Xbb_\Ccr$ encoding output colors and the variables of
$\Ybb_\Ccr$ encoding input colors%
\footnote{See an example of a series of color types in
Section~\ref{subsubsec:example_series_trees}.}%
. In the sequel, we shall be concerned by
the computation of the coefficients of $\ColorTypes(\Fbf)$ for some
$\Cca$-series~$\Fbf$.
\medbreak

\subsubsection{Elementary series of bud generating systems}
\label{subsubsec:elementary_series_bud_generating_systems}
We assume here that $\Oca$ is a locally finite monochrome operad. We
shall denote by $\Rbf$ the characteristic series of $\Rfr$, by $\Ibf$
the characteristic series of $\Bra{\Unit_a : a \in I}$, and by $\Tbf$
the characteristic series of $\Bra{\Unit_a : a \in T}$. For all colors
$a \in \Ccr$ and types $\alpha \in \Types_\Ccr$, let
\begin{equation}
    \MultOutIn_{a, \alpha} :=
    \# \Bra{r \in \Rfr : (\Out(r), \Type(\In(r))) = (a, \alpha)}.
\end{equation}
For any $a \in \Ccr$, let $\Gbf_a\Par{\VarY_{c_1}, \dots, \VarY_{c_k}}$
be the series of $\K\Han{\Han{\Ybb_\Ccr}}$ defined by
\begin{equation}
    \Gbf_a\Par{\VarY_{c_1}, \dots, \VarY_{c_k}}
    :=
    \sum_{\gamma \in \Types_\Ccr}
    \MultOutIn_{a, \gamma} \enspace \Ybb_\Ccr^\gamma
    =
    \sum_{\substack{
        r \in \Rfr \\
        \Out(r) = a
    }}
    \Ybb_\Ccr^{\Type(\In(r))}.
\end{equation}
Since $\Rfr$ is finite, this series is a polynomial%
\footnote{See examples of these definitions in
Sections~\ref{subsubsec:colt_synt_bubtree},
\ref{subsubsec:colt_sync_bbtree},
and~\ref{subsubsec:colt_sync_bbaltree}.}%
.
\medbreak

In the sequel, we shall use maps $\phi : \Ccr \times \Types_\Ccr \to \N$
such that $\phi(a, \gamma) \ne 0$ for a finite number of pairs
$(a, \gamma) \in \Ccr \times \Types_\Ccr$, to express in a concise
manner some recurrence relations for the coefficients of series on
colored operads. We shall consider the two following notations. If
$\phi$ is such a map and $a \in \Ccr$, we define $\phi^{(a)}$ as the
natural number
\begin{equation}
    \phi^{(a)} :=
    \sum_{\substack{
        b \in \Ccr \\
        \gamma \in \Types_\Ccr
    }}
    \phi(b, \gamma) \gamma_a
\end{equation}
and $\phi_a$ as the finite multiset
\begin{equation}
    \phi_a := \lbag \phi(a, \gamma) : \gamma \in \Types_\Ccr \rbag.
\end{equation}
\medbreak

\subsection{Products on series}
Two binary products $\PreLie$ and $\Compo$ on the space of $\Cca$-series
are presented. The product $\PreLie$ is a generalization to series and
to colored operads of a known product on monochrome operads, and
$\Compo$ is a generalization to colored operads of a known product on
series on monochrome operads.
\medbreak

\subsubsection{Pre-Lie product}
Given two $\Cca$-series $\Fbf, \Gbf \in \K \AAngle{\Cca}$, the
\Def{pre-Lie product} of $\Fbf$ and $\Gbf$ is the $\Cca$-series
$\Fbf \PreLie \Gbf$ defined, for any $x \in \Cca$, by
\begin{equation} \label{equ:pre_lie_product}
    \Angle{x, \Fbf \PreLie \Gbf}
    := \sum_{\substack{
            y, z \in \Cca \\
            i \in [|y|] \\
            x = y \circ_i z}}
    \Angle{y, \Fbf} \Angle{z, \Gbf}.
\end{equation}
Observe that $\Fbf \PreLie \Gbf$ could be undefined for arbitrary
$\Cca$-series $\Fbf$ and $\Gbf$ on an arbitrary colored operad $\Cca$.
Besides, notice from~\eqref{equ:pre_lie_product} that $\PreLie$ is
bilinear and that $\Ubf$ is a left unit of~$\PreLie$. However, since
\begin{equation}
    \Fbf \PreLie \Ubf =
    \sum_{x \in \Cca} |x| \Angle{x, \Fbf} x,
\end{equation}
the $\Cca$-series $\Ubf$ is not a right unit of $\PreLie$. This product
is also nonassociative in the general case since we have, for instance
in~$\K \AAngle{\As}$,
\begin{equation}
    \Par{\OpAs_2 \PreLie \OpAs_2} \PreLie \OpAs_2 =
    6 \OpAs_4
    \ne
    4 \OpAs_4 =
    \OpAs_2 \PreLie \Par{\OpAs_2 \PreLie \OpAs_2}.
\end{equation}
\medbreak

Recall that a \Def{$\K$-pre-Lie algebra}~\cite{Vin63,Ger63} (see
also~\cite{CL01,Man11}) is a $\K$-vector space $V$ endowed with a
bilinear product $\PreLie$ satisfying, for all $x, y, z \in V$, the
relation
\begin{equation} \label{equ:pre_lie_relation}
    (x \PreLie y) \PreLie z - x \PreLie (y \PreLie z) =
    (x \PreLie z) \PreLie y - x \PreLie (z \PreLie y).
\end{equation}
In this case, we say that $\PreLie$ is a \Def{pre-Lie product}. Observe
that any associative product satisfies~\eqref{equ:pre_lie_relation}, so
that associative algebras are pre-Lie algebras.
\medbreak

\begin{Proposition} \label{prop:functor_series_pre_lie_algebras}
    For any locally finite colored operad $\Cca$, the space
    $\K \AAngle{\Cca}$ endowed with the binary product $\PreLie$ is a
    pre-Lie algebra.
\end{Proposition}
\medbreak

This product $\PreLie$ is a generalization of a pre-Lie product defined
in~\cite{Ger63} (see also~\cite{VDL04,Cha08}), endowing the $\K$-linear
span of the underlying monochrome graded collection of a monochrome
operad with a pre-Lie algebra structure.
Proposition~\ref{prop:functor_series_pre_lie_algebras} is based on
similar arguments as the ones contained in the previous references.
\medbreak

\subsubsection{Composition product}
Given two $\Cca$-series $\Fbf, \Gbf \in \K \AAngle{\Cca}$, the
\Def{composition product} of $\Fbf$ and $\Gbf$ is the $\Cca$-series
$\Fbf \Compo \Gbf$ defined, for any $x \in \Cca$, by
\begin{equation} \label{equ:associative_product}
    \Angle{x, \Fbf \Compo \Gbf}
    :=
    \sum_{\substack{
        y, z_1, \dots, z_{|y|} \in \Cca \\
        x = y \circ \Han{z_1, \dots, z_{|y|}}
    }}
    \Angle{y, \Fbf}
    \prod_{i \in [|y|]} \Angle{z_i, \Gbf}.
\end{equation}
Observe that $\Fbf \Compo \Gbf$ could be undefined for arbitrary
$\Cca$-series $\Fbf$ and $\Gbf$ on an arbitrary colored operad $\Cca$.
Besides, notice from~\eqref{equ:associative_product} that $\Compo$ is
linear on the left and that the series $\Ubf$ is the left and right unit
of $\Compo$. However, this product is not linear on the right since we
have, for instance in~$\K \AAngle{\As}$,
\begin{equation}
    \OpAs_2 \Compo (\OpAs_2 + \OpAs_3) =
    \OpAs_4 + 2 \OpAs_5 + \OpAs_6
    \ne \OpAs_4 + \OpAs_6 =
    \OpAs_2 \Compo \OpAs_2 + \OpAs_2 \Compo \OpAs_3.
\end{equation}
\medbreak

\begin{Proposition} \label{prop:functor_series_monoids}
    For any locally finite colored operad $\Cca$, the space
    $\K \AAngle{\Cca}$ endowed with the binary product $\Compo$ and the
    unit $\Ubf$ is a monoid.
\end{Proposition}
\medbreak

This product $\Compo$ is a generalization of the composition product of
series on operads of~\cite{Cha02,Cha09} (see
also~\cite{VDL04,Fra08,Cha08,LV12,LN13}). In the case where $\Cca$ is a
monochrome operad concentrated in arity $1$, $\Compo$ coincides with the
Cauchy product on series of monoids considered in~\cite{Sak09}.
Proposition~\ref{prop:functor_series_monoids} is based on similar
arguments as the ones contained in the previous references.
\medbreak

\begin{Lemma} \label{lem:initial_terminal_composition}
    Let $\Bca := (\Oca, \Ccr, \Rfr, I, T)$ be a bud generating system
    and $\Fbf$ be a $\Bud_\Ccr(\Oca)$-series. Then,
    $\Ibf \Compo \Fbf \Compo \Tbf$ is the $\Bud_\Ccr(\Oca)$-series
    satisfying, for all $x \in \Bud_\Ccr(\Oca)$,
    \begin{equation}
        \Angle{x, \Ibf \Compo \Fbf \Compo \Tbf} =
        \begin{cases}
            \Angle{x, \Fbf}
                & \mbox{if } \Out(x) \in I \mbox{ and }
                    \In(x) \in T^+, \\
            0 & \mbox{otherwise}.
        \end{cases}
    \end{equation}
\end{Lemma}
\begin{proof}
    By definition of the operation $\Compo$, composing $\Fbf$ with
    $\Ibf$ to the left and with $\Tbf$ to the right with respect to
    $\Compo$ amounts to annihilate the coefficients of the terms of
    $\Fbf$ that have an output color which is not in $I$ or an input
    color which is not in $T$. This implies the statement of the lemma.
\end{proof}
\medbreak

\subsection{Series and languages} \label{subsec:series_languages}
We introduce the Kleene star operation of the pre-Lie product~$\PreLie$
in order to define the hook generating series of a bud generating
system $\Bca$. We also study the inverse of the composition product
$\Compo$ in order to define the syntactic generating series of $\Bca$.
We relate both of these series with the language of $\Bca$ and provide
ways to compute its coefficients.
\medbreak

\subsubsection{Pre-Lie star product}
For any $\Cca$-series $\Fbf \in \K \AAngle{\Cca}$ and any $\ell \geq 0$,
let $\Fbf^{\PreLie_\ell}$ be the $\Cca$-series recursively defined by
\begin{equation}
    \Fbf^{\PreLie_\ell} :=
    \begin{cases}
        \Ubf & \mbox{if } \ell = 0, \\
        \Fbf^{\PreLie_{\ell - 1}} \PreLie \Fbf & \mbox{otherwise}.
    \end{cases}
\end{equation}
Immediately from this definition and the definition of the pre-Lie
product $\PreLie$, the coefficients of $\Fbf^{\PreLie_\ell}$,
$\ell \geq 0$, satisfy for any $x \in \Cca$,
\begin{equation} \label{equ:induction_pre_lie_powers}
    \Angle{x, \Fbf^{\PreLie_\ell}} =
    \begin{cases}
        \delta_{x, \Unit_{\Out(x)}} & \mbox{if } \ell = 0, \\
        \sum\limits_{\substack{
            y, z \in \Cca \\
            i \in [|y|] \\
            x = y \circ_i z}}
        \Angle{y, \Fbf^{\PreLie_{\ell - 1}}} \Angle{z, \Fbf}
        & \mbox{otherwise}.
    \end{cases}
\end{equation}
\medbreak

\begin{Lemma} \label{lem:pre_lie_powers}
    Let $\Cca$ be a locally finite $\Ccr$-colored operad and $\Fbf$ be a
    series of $\K \AAngle{\Cca}$. Then, the coefficients of
    $\Fbf^{\PreLie_{\ell + 1}}$, $\ell \geq 0$, satisfy for any
    $x \in \Cca$,
    \begin{equation} \label{equ:pre_lie_powers}
        \Angle{x, \Fbf^{\PreLie_{\ell + 1}}} =
        \sum_{\substack{
            y_1, \dots, y_{\ell + 1} \in \Cca \\
            i_1, \dots, i_\ell \in \N \\
            x = \Par{\dots \Par{y_1 \circ_{i_1} y_2}
                \circ_{i_2} \dots} \circ_{i_\ell} y_{\ell + 1}
        }}
        \prod_{j \in [\ell + 1]} \Angle{y_j, \Fbf}.
    \end{equation}
\end{Lemma}
\begin{proof}
    By Proposition~\ref{prop:functor_series_pre_lie_algebras}, since
    $\Cca$ is locally finite, $\Fbf^{\PreLie_{\ell + 1}}$ is a
    well-defined $\Cca$-series. The statement of the lemma follows by
    induction on $\ell$ and by
    using~\eqref{equ:induction_pre_lie_powers}.
\end{proof}
\medbreak

The \Def{$\PreLie$-star} of $\Fbf$ is the series
\begin{equation}
    \Fbf^{\PreLie_*} := \sum_{\ell \geq 0} \Fbf^{\PreLie_\ell}
    = \Ubf + \Fbf + \Fbf \PreLie \Fbf + (\Fbf \PreLie \Fbf) \PreLie \Fbf +
        ((\Fbf \PreLie \Fbf) \PreLie \Fbf) \PreLie \Fbf + \cdots.
\end{equation}
Observe that $\Fbf^{\PreLie_*}$ could be undefined for an arbitrary
$\Cca$-series~$\Fbf$.
\medbreak

\begin{Proposition} \label{prop:pre_lie_star}
    Let $\Cca$ be a locally finite $\Ccr$-colored operad and $\Fbf$ be a
    series of $\K \AAngle{\Cca}$ such that $\Supp(\Fbf)(1)$ finitely
    factorizes $\Cca$. Then,
    \begin{enumerate}[label = ({\it {\roman*})}]
        \item \label{item:pre_lie_star_1}
        the series $\Fbf^{\PreLie_*}$ is well-defined;
        \item \label{item:pre_lie_star_2}
        for any $x \in \Cca$, the coefficient of $x$ in
        $\Fbf^{\PreLie_*}$ satisfies
        \begin{equation} \label{equ:coefficients_pre_lie_star}
            \Angle{x, \Fbf^{\PreLie_*}} =
            \delta_{x, \Unit_{\Out(x)}} +
            \sum_{\substack{
                y, z \in \Cca \\
                i \in [|y|] \\
                x = y \circ_i z}}
            \Angle{y, \Fbf^{\PreLie_*}} \Angle{z, \Fbf};
        \end{equation}
        \item \label{item:pre_lie_star_3}
        the equation
        \begin{equation} \label{equ:equation_pre_lie_star}
            \Xbf - \Xbf \PreLie \Fbf = \Ubf
        \end{equation}
        admits $\Xbf = \Fbf^{\PreLie_*}$ as unique solution.
    \end{enumerate}
\end{Proposition}
\begin{proof}
    Let $x \in \Cca$ and let us show that the coefficient
    $\Angle{x, \Fbf^{\PreLie_*}}$ is well-defined. Since $\Cca$ is
    locally finite and $\Supp(\Fbf)(1)$ finitely factorizes $\Cca$, by
    Lemma~\ref{lem:finitely_factorizing_sets}, there are finitely many
    $\Supp(\Fbf)$-treelike expressions for $x$. Thus, for all
    $\ell \geq \deg_{\Supp(\Fbf)}(x) + 1$ there is in particular no
    expression for $x$ of the form
    \begin{math}
        x = \Par{\dots \Par{y_1 \circ_{i_1} y_2}
                \circ_{i_2} \dots} \circ_{i_{\ell - 1}} y_\ell
    \end{math}
    where $y_1, \dots, y_\ell \in \Supp(\Fbf)$ and
    $i_1, \dots, i_{\ell - 1} \in \N$. This implies, together with
    Lemma~\ref{lem:pre_lie_powers}, that
    $\Angle{x, \Fbf^{\PreLie_\ell}} = 0$. Therefore, by virtue of this
    observation and by definition of the $\PreLie$-star operation, the
    coefficient of $x$ in $\Fbf^{\PreLie_*}$ is
    \begin{equation}
        \Angle{x, \Fbf^{\PreLie_*}} =
        \sum_{\ell \geq 0} \Angle{x, \Fbf^{\PreLie_\ell}}
        = \sum_{0 \leq \ell \leq \deg_{\Supp(\Fbf)}(x)}
        \Angle{x, \Fbf^{\PreLie_\ell}},
    \end{equation}
    showing that $\Angle{x, \Fbf^{\PreLie_*}}$ is a sum of a finite
    number of terms, all well-defined by Lemma~\ref{lem:pre_lie_powers}.
    Thus, $\Fbf^{\PreLie_*}$ is well-defined, so
    that~\ref{item:pre_lie_star_1} holds.
    \smallbreak

    Point~\ref{item:pre_lie_star_2} follows straightforwardly from the
    definition of the $\PreLie$-star operation
    and~\eqref{equ:induction_pre_lie_powers}.
    \smallbreak

    By~\eqref{equ:equation_pre_lie_star}, we have
    $\Xbf = \Ubf + \Xbf \PreLie \Fbf$ so that the coefficients of $\Xbf$
    satisfy, for any $x \in \Cca$,
    \begin{equation}
        \Angle{x, \Xbf}
        = \Angle{x, \Ubf} + \Angle{x, \Xbf \PreLie \Fbf}
        = \delta_{x, \Unit_{\Out(x)}} +
        \sum_{\substack{
            y, z \in \Cca \\
            i \in [|y|] \\
            x = y \circ_i z}}
        \Angle{y, \Xbf}
        \Angle{z, \Fbf}.
    \end{equation}
    By~\ref{item:pre_lie_star_2}, this implies
    $\Xbf =  \Fbf^{\PreLie_*}$ and the uniqueness of this solution, so
    that~\ref{item:pre_lie_star_3} is established.
\end{proof}
\medbreak

In particular, Point~\ref{item:pre_lie_star_2} of
Proposition~\ref{prop:pre_lie_star} gives a way, given a $\Cca$-series
$\Fbf$ satisfying the stated constraints, to compute recursively the
coefficients of its $\PreLie$-star~$\Fbf^{\PreLie_*}$.
\medbreak

\subsubsection{Hook generating series}
\label{subsubsec:hook_generating_series}
The \Def{hook generating series} of $\Bca$ the $\Bud_\Ccr(\Oca)$-series
$\Hook(\Bca)$ defined by
\begin{equation} \label{equ:definition_hook_series}
    \Hook(\Bca) :=
    \Ibf \Compo \Rbf^{\PreLie_*} \Compo \Tbf.
\end{equation}
Observe that~\eqref{equ:definition_hook_series} could be undefined for
an arbitrary set of rules $\Rfr$ of $\Bca$. Nevertheless, when $\Rbf$
satisfies the conditions of Proposition~\ref{prop:pre_lie_star}, that
is, when $\Oca$ is a locally finite operad and $\Rfr(1)$ finitely
factorizes $\Bud_\Ccr(\Oca)$, $\Hook(\Bca)$ is well-defined.
\medbreak

The aim of the following is to provide an expression to compute the
coefficients of~$\Hook(\Bca)$.
\medbreak

\begin{Lemma} \label{lem:pre_lie_star_hook_series}
    Let $\Bca := (\Oca, \Ccr, \Rfr, I, T)$ be a bud generating system
    such that $\Oca$ is a locally finite operad and $\Rfr(1)$ finitely
    factorizes $\Bud_\Ccr(\Oca)$. Then, for any $x \in \Bud_\Ccr(\Oca)$,
    \begin{equation} \label{equ:pre_lie_star_hook_series}
        \Angle{x, \Rbf^{\PreLie_*}} =
        \delta_{x, \Unit_{\Out(x)}} +
        \sum_{\substack{
            y \in \Bud_\Ccr(\Oca) \\
            z \in \Rfr \\
            i \in [|y|] \\
            x = y \circ_i z}}
        \Angle{y, \Rbf^{\PreLie_*}}.
    \end{equation}
\end{Lemma}
\begin{proof}
    Since $\Rfr(1)$ finitely factorizes $\Bud_\Ccr(\Oca)$, by
    Point~\ref{item:pre_lie_star_1} of
    Proposition~\ref{prop:pre_lie_star}, $\Rbf^{\PreLie_*}$ is a
    well-defined series. Now, \eqref{equ:pre_lie_star_hook_series} is a
    consequence of Point~\ref{item:pre_lie_star_2} of
    Proposition~\ref{prop:pre_lie_star} together with the fact that all
    coefficients of $\Rbf$ are equal to $0$ or to~$1$.
\end{proof}
\medbreak

\begin{Proposition} \label{prop:hook_series_derivation_graph}
    Let $\Bca := (\Oca, \Ccr, \Rfr, I, T)$ be a bud generating system
    such that $\Oca$ is a locally finite operad and $\Rfr(1)$ finitely
    factorizes $\Bud_\Ccr(\Oca)$. Then, for any $x \in \Bud_\Ccr(\Oca)$
    such that $\Out(x) \in I$, the coefficient
    $\Angle{x, \Rbf^{\PreLie_*}}$ is the number of multipaths from
    $\Unit_{\Out(x)}$ to $x$ in the derivation graph of~$\Bca$.
\end{Proposition}
\begin{proof}
    First, since $\Rfr(1)$ finitely factorizes $\Bud_\Ccr(\Oca)$, by
    Point~\ref{item:pre_lie_star_1} of
    Proposition~\ref{prop:pre_lie_star}, $\Rbf^{\PreLie_*}$ is a
    well-defined series. If $x = \Unit_a$ for an $a \in I$, since
    $\Angle{\Unit_a, \Rbf^{\PreLie_*}} = 1$, the statement of the
    proposition holds. Let us now assume that $x$ is different from a
    colored unit and let us denote by $\lambda_x$ the number of
    multipaths from $\Unit_{\Out(x)}$ to $x$ in the derivation graph
    $\DerivGraph(\Bca)$ of $\Bca$. By definition of $\DerivGraph(\Bca)$,
    by denoting by $\mu_{y, x}$ the number of edges from
    $y \in \Bud_\Ccr(\Oca)$ to $x$ in $\DerivGraph(\Bca)$, we have
    \begin{equation}\begin{split}
        \label{equ:hook_series_derivation_graph}
        \lambda_x
        & = \sum_{y \in \Bud_\Ccr(\Oca)} \mu_{y, x}\, \lambda_y
        =
        \sum_{y \in \Bud_\Ccr(\Oca)}
        \# \Bra{(i, r) \in \N \times \Rfr : x = y \circ_i r}\,
        \lambda_y
        =
        \sum_{\substack{
            y \in \Bud_\Ccr(\Oca) \\
            i \in [|y|] \\
            r \in \Rfr \\
            x = y \circ_i r
        }}
        \lambda_y.
    \end{split}\end{equation}
    We observe that Relation~\eqref{equ:hook_series_derivation_graph}
    satisfied by the $\lambda_x$ is the same as
    Relation~\eqref{equ:pre_lie_star_hook_series} of
    Lemma~\ref{lem:pre_lie_star_hook_series} satisfied by the
    $\Angle{x, \Rbf^{\PreLie_*}}$. This implies the statement of the
    proposition.
\end{proof}
\medbreak

\begin{Theorem} \label{thm:hook_series}
    Let $\Bca := (\Oca, \Ccr, \Rfr, I, T)$ be a bud generating system
    such that $\Oca$ is a locally finite operad and $\Rfr(1)$
    finitely factorizes $\Bud_\Ccr(\Oca)$. Then, the hook generating
    series of $\Bca$ satisfies
    \begin{equation} \label{equ:hook_series}
        \Hook(\Bca)
        = \sum_{\substack{
            \Tfr \in \Free(\Rfr) \\
            \Out(\Tfr) \in I \\
            \In(\Tfr) \in T^+
        }}
        \enspace
        \frac{\deg(\Tfr)!}
            {\prod\limits_{v \in \INodes(\Tfr)} \deg\Par{\Tfr_v}}
        \enspace
        \Eval_{\Bud_\Ccr(\Oca)}(\Tfr).
    \end{equation}
\end{Theorem}
\begin{proof}
    By definition of $\Lang(\Bca)$ and $\DerivGraph(\Bca)$, any
    $x \in \Lang(\Bca)$ can be reached from $\Unit_{\Out(x)}$ by a
    multipath
    \begin{equation}
        \Unit_{\Out(x)} \Deriv y_1 \Deriv y_2 \Deriv \cdots
        \Deriv y_{\ell - 1} \Deriv x
    \end{equation}
    in $\DerivGraph(\Bca)$, where $y_1, \dots, y_{\ell - 1}$ are
    elements of $\Bud_\Ccr(\Oca)$ and $\Unit_{\Out(x)} \in I$. Hence,
    by definition of~$\Deriv$, $x$ admits an $\Rfr$-left expression
    \begin{equation}
        x = \Par{\dots \Par{\Par{\Unit_{\Out(x)} \circ_1 r_1}
        \circ_{i_1} r_2} \circ_{i_2} \dots} \circ_{i_{\ell - 1}} r_\ell
    \end{equation}
    where for any $j \in [\ell]$, $r_j \in \Rfr$, and for any
    $j \in [\ell - 1]$,
    \begin{equation}
        y_j = \Par{\dots \Par{\Par{\Unit_{\Out(x)} \circ_1 r_1}
        \circ_{i_1} r_2} \circ_{i_2} \dots} \circ_{i_{j - 1}} r_j
    \end{equation}
    and $i_{j} \in [|y_j|]$. This shows that the set of all multipaths
    from $\Unit_{\Out(x)}$ to $x$ in $\DerivGraph(\Bca)$ is in
    one-to-one correspondence with the set of all $\Rfr$-left
    expressions for $x$. Now, observe that since $\Rfr(1)$ finitely
    factorizes $\Bud_\Ccr(\Oca)$, by Point~\ref{item:pre_lie_star_1} of
    Proposition~\ref{prop:pre_lie_star}, $\Rbf^{\PreLie_*}$ is a
    well-defined series. If $x = \Unit_a$ for an $a \in I$, since
    $\Angle{\Unit_a, \Rbf^{\PreLie_*}} = 1$, the statement of the
    proposition holds. Let us now assume that $x$ is different from a
    colored unit and let us denote by $\lambda_x$ the number of
    multipaths from $\Unit_{\Out(x)}$ to $x$ in the derivation graph
    $\DerivGraph(\Bca)$ of $\Bca$. By d, $\Rbf^{\PreLie_*}$ is a
    well-defined series. By
    Proposition~\ref{prop:hook_series_derivation_graph},
    Lemmas~\ref{lem:finitely_factorizing_sets}
    and~\ref{lem:left_expressions_linear_extensions},
    and~\eqref{equ:number_left_expressions}, we obtain that
    \begin{equation} \label{equ:hook_series_demo}
        \Angle{x, \Rbf^{\PreLie_*}} =
        \sum_{\substack{
            \Tfr \in \Free(\Rfr) \\
            \Eval_{\Bud_\Ccr(\Oca)}(\Tfr) = x
        }}
        \frac{\deg(\Tfr)!}
            {\prod\limits_{v \in \INodes(\Tfr)} \deg(\Tfr_v)}.
    \end{equation}
    Finally, by Lemma~\ref{lem:initial_terminal_composition}, for any
    $x \in \Bud_\Ccr(\Oca)$ such that $\Out(x) \in I$ and
    $\In(x) \in T^+$, we have
    $\Angle{x, \Hook(\Bca)} = \Angle{x, \Rbf^{\PreLie_*}}$. This shows
    that the right member of~\eqref{equ:hook_series} is equal
    to~$\Hook(\Bca)$.
\end{proof}
\medbreak

An alternative way to understand $\Hook(\Bca)$ thus offered by
Theorem~\ref{thm:hook_series} consists is seeing the coefficient
$\Angle{x, \Hook(\Bca)}$, $x \in \Bud_\Ccr(\Oca)$, as the number of
$\Rfr$-left expressions of~$x$.
\medbreak

The following result establishes a link between the hook generating
series of $\Bca$ and its language.
\medbreak

\begin{Proposition} \label{prop:support_hook_series_language}
    Let $\Bca := (\Oca, \Ccr, \Rfr, I, T)$ be a bud generating system
    such that $\Oca$ is a locally finite operad and $\Rfr(1)$
    finitely factorizes $\Bud_\Ccr(\Oca)$. Then, the support of the
    hook generating series of $\Bca$ is the language of~$\Bca$.
\end{Proposition}
\begin{proof}
    This is an immediate consequence of Theorem~\ref{thm:hook_series}
    and Lemma~\ref{lem:language_treelike_expressions}.
\end{proof}
\medbreak

Bud generating systems lead to the definition of analogues of the
\Def{hook-length statistic}~\cite{Knu98} for combinatorial objects
possibly different than trees in the following way. Let $\Oca$ be a
monochrome operad, $G$ be a generating set of $\Oca$, and
$\HookSystem_{\Oca, G} := (\Oca, G)$ be a monochrome bud generating
system depending on $\Oca$ and $G$, called \Def{hook bud generating
system}. Since $G$ is a generating set of $\Oca$, by
Propositions~\ref{prop:language_sub_operad_generated}
and~\ref{prop:support_hook_series_language}, the support of
$\Hook\Par{\HookSystem_{\Oca, G}}$ is equal to
$\Lang\Par{\HookSystem_{\Oca, G}}$. We define the \Def{hook-length
coefficient} of any element $x$ of $\Oca$ as the coefficient
$\Angle{x, \Hook\Par{\HookSystem_{\Oca, G}}}$%
\footnote{See examples of definitions of a hook-length statistics for
binary trees, words of $\Dias_\gamma$, and for Motzkin paths in
Sections~\ref{subsubsec:example_hook_coeff_binary_trees},
\ref{subsubsec:example_hook_coeff_dias},
and~\ref{subsubsec:example_hook_coeff_motz}.}%
.
\medbreak

\subsubsection{Invertible elements for the composition product}
Since by Proposition~\ref{prop:functor_series_monoids}, $\Compo$ is an
associative product and $\Ubf$ is its unit, the \Def{$\Compo$-inverse}
of a $\Cca$-series $\Fbf$ is defined as the unique $\Cca$-series $\Xbf$
satisfying
\begin{equation} \label{equ:equation_composition_inverse}
    \Fbf \Compo \Xbf = \Ubf = \Xbf \Compo \Fbf.
\end{equation}
This series $\Xbf$ could be undefined for an arbitrary $\Cca$-series
$\Fbf$. The $\Compo$-inverse of $\Fbf$ is denoted by
$\Fbf^{\Compo_{-1}}$ when it is well-defined.
\medbreak

Immediately from this definition and the definition of the composition
product $\Compo$, the coefficients of $\Fbf^{\Compo_{-1}}$ satisfy
for any $x \in \Cca$,
\begin{equation} \label{equ:induction_composition_inverse}
    \Angle{x, \Fbf^{\Compo_{-1}}}
    =
    \frac{\delta_{x, \Unit_{\Out(x)}}}{\Angle{\Unit_{\Out(x)}, \Fbf}}
     -
    \frac{1}{\Angle{\Unit_{\Out(x)}, \Fbf}}
    \sum_{\substack{
        y, z_1, \dots, z_{|y|} \in \Cca \\
        y \ne \Unit_{\Out(x)} \\
        x = y \circ \Han{z_1, \dots, z_{|y|}}}}
    \Angle{y, \Fbf}
    \prod_{i \in [|y|]} \Angle{z_i, \Fbf^{\Compo_{-1}}}.
\end{equation}
\medbreak

\begin{Proposition} \label{prop:composition_inverse}
    Let $\Cca$ be a locally finite colored $\Ccr$-operad and $\Fbf$ be a
    series of $\K \AAngle{\Cca}$ such that
    $\Supp(\Fbf) = \Bra{\Unit_a : a \in \Ccr} \sqcup S$ where $S$ is
    a $\Ccr$-colored graded subcollection of $\Cca$ such that $S(1)$ is
    finitely factorizes $\Cca$. Then,
    \begin{enumerate}[label = ({\it {\roman*})}]
        \item \label{item:composition_inverse_1}
        the series $\Fbf^{\Compo_{-1}}$ is well-defined;
        \item \label{item:composition_inverse_2}
        for any $x \in \Cca$, the coefficient of $x$ in
        $\Fbf^{\Compo_{-1}}$ satisfies
        \begin{equation} \label{equ:composition_inverse_expression}
            \Angle{x, \Fbf^{\Compo_{-1}}} =
            \frac{1}{\Angle{\Unit_{\Out(x)}, \Fbf}}
            \enspace
            \sum_{\substack{
                \Tfr \in \Free(S) \\
                \Eval_\Cca(\Tfr) = x
            }}
            (-1)^{\deg(\Tfr)}
            \prod_{v \in \INodes(\Tfr)}
            \enspace
            \frac{\Angle{\Tfr(v), \Fbf}}
                {\prod\limits_{j \in [|v|]}
                \Angle{\Unit_{\In_j(v)}, \Fbf}}.
        \end{equation}
    \end{enumerate}
\end{Proposition}
\begin{proof}
    Let us first assume that $x$ does not belong to $\Cca^S$, the
    colored suboperad of $\Cca$ generated by $S$. Hence, since there is
    no $\Tfr \in \Free(S)$ such that $\Eval_\Cca(\Tfr) = x$, the right
    member of~\eqref{equ:composition_inverse_expression} is equal to
    zero. Moreover, since $x$ does not belong to $\Cca^S$, for any
    $y \in \Cca$ and $z_1, \dots, z_{|y|} \in \Cca$ such that
    $y \ne \Unit_{\Out(x)}$ and
    $x = y \circ \Han{z_1, \dots, z_{|y|}}$, we have necessarily
    $y \notin S$ or $z_i \notin \Cca^S$ for at least one $i \in [|y|]$.
    By~\eqref{equ:induction_composition_inverse}, this implies that
    $\Angle{x, \Fbf^{\Compo_{-1}}} = 0$. This hence shows
    that~\eqref{equ:composition_inverse_expression} holds
    when~$x \notin \Cca^S$.
    \smallbreak

    Let us assume that $x$ belongs to $\Cca^S$. By
    Lemma~\ref{lem:finitely_factorizing_sets}, since $\Cca$ is locally
    finite and $S(1)$ finitely factorizes $\Cca$, the $S$-degree
    $\deg_S(x)$ of $x$ is well-defined. To
    prove~\eqref{equ:composition_inverse_expression}, we proceed by
    induction on $\deg_S(x)$ and by
    using~\eqref{equ:induction_composition_inverse}. A straightforward
    computation shows that~\eqref {equ:composition_inverse_expression}
    holds so that~\ref{item:composition_inverse_2} checks out.
    \smallbreak

    Finally, by Lemma~\ref{lem:finitely_factorizing_sets}, there is a
    finite number of $S$-treelike expressions of $x$. This shows
    that~\eqref{equ:composition_inverse_expression} is well-defined and
    then $\Fbf^{\Compo_{-1}}$ also is. Hence,
    \ref{item:composition_inverse_1} holds.
\end{proof}
\medbreak

Given a $\Cca$-series $\Fbf$ such that $\Fbf^{\Compo_{-1}}$ is
well-defined, Equation~\eqref{equ:induction_composition_inverse} (resp.
Proposition~\ref{prop:composition_inverse}) provides a recursive (resp.
direct) way to compute the coefficients of $\Fbf^{\Compo_{-1}}$.
\medbreak

Besides, the set of all the $\Cca$-series satisfying the conditions of
Proposition~\ref{prop:composition_inverse} forms a submonoid of
$\K \AAngle{C}$ for the composition product which is also a group. This
group is a generalization of the groups constructed from operads
of~\cite{Cha02,Cha09} (see also~\cite{VDL04,Fra08,Cha08,LV12,LN13}).
\medbreak

\subsubsection{Syntactic generating series}%
\label{subsubsec:syntactic_generating_series}
The \Def{syntactic generating series} of $\Bca$ the
$\Bud_\Ccr(\Oca)$-series $\Synt(\Bca)$ defined by
\begin{equation} \label{equ:definition_syntactic_series}
    \Synt(\Bca) :=
    \Ibf \Compo (\Ubf - \Rbf)^{\Compo_{-1}} \Compo \Tbf.
\end{equation}
Observe that~\eqref{equ:definition_syntactic_series} could be undefined
for an arbitrary set of rules $\Rfr$ of $\Bca$. Nevertheless, when
$\Ubf - \Rbf$ satisfies the conditions of
Proposition~\ref{prop:composition_inverse}, $\Synt(\Bca)$ is
well-defined. Remark that this condition is satisfied whenever
$\Oca$ is locally finite and $\Rfr(1)$ factorizes
finitely~$\Bud_\Ccr(\Oca)$.
\medbreak

The aim of this section is to provide an expression to compute the
coefficients of~$\Synt(\Bca)$.
\medbreak

\begin{Lemma} \label{lem:inverse_composition_syntactic_series}
    Let $\Bca := (\Oca, \Ccr, \Rfr, I, T)$ be a bud generating system
    such that $\Oca$ is a locally finite operad and $\Rfr(1)$
    finitely factorizes $\Bud_\Ccr(\Oca)$. Then, for any
    $x \in \Bud_\Ccr(\Oca)$,
    \begin{equation} \label{equ:inverse_composition_syntactic_series}
        \Angle{x, (\Ubf - \Rbf)^{\Compo_{-1}}} =
        \delta_{x, \Unit_{\Out(x)}} +
        \sum_{\substack{
            y \in \Rfr \\
            z_1, \dots, z_{|y|} \in \Bud_\Ccr(\Oca) \\
            x = y \circ \Han{z_1, \dots, z_{|y|}}
        }}
        \enspace
        \prod_{i \in [|y|]} \Angle{z_i, (\Ubf - \Rbf)^{\Compo_{-1}}}.
    \end{equation}
\end{Lemma}
\begin{proof}
    Since $\Rfr(1)$ finitely factorizes $\Bud_\Ccr(\Oca)$, by
    Point~\ref{item:composition_inverse_1} of
    Proposition~\ref{prop:composition_inverse},
    $(\Ubf - \Rbf)^{\Compo_{-1}}$ is a well-defined series. Now,
    \eqref{equ:inverse_composition_syntactic_series} is a consequence of
    Point~\ref{item:composition_inverse_2} of
    Proposition~\ref{prop:composition_inverse} and
    Equation~\eqref{equ:induction_composition_inverse} for the
    $\Compo$-inverse, together with the fact that all coefficients of
    $\Rbf$ are equal to $0$ or to~$1$.
\end{proof}
\medbreak

\begin{Theorem} \label{thm:syntactic_series}
    Let $\Bca := (\Oca, \Ccr, \Rfr, I, T)$ be a bud generating system
    such that $\Oca$ is a locally finite operad and $\Rfr(1)$
    finitely factorizes $\Bud_\Ccr(\Oca)$. Then, the syntactic
    generating series of $\Bca$ satisfies
    \begin{equation} \label{equ:syntactic_series}
        \Synt(\Bca)
        =
        \sum_{\substack{
            \Tfr \in \Free(\Rfr) \\
            \Out(\Tfr) \in I \\
            \In(\Tfr) \in T^+
        }}
        \Eval_{\Bud_\Ccr(\Oca)}(\Tfr).
    \end{equation}
\end{Theorem}
\begin{proof}
    Let, for any $x \in \Bud_\Ccr(\Oca)$, $\lambda_x$ be the number of
    $\Rfr$-treelike expressions for $x$. Since $\Rfr(1)$ finitely
    factorizes $\Bud_\Ccr(\Oca)$, by
    Lemma~\ref{lem:finitely_factorizing_sets},
    all $\lambda_x$ are well-defined integers. Moreover, since $\Rfr(1)$
    finitely factorizes $\Bud_\Ccr(\Oca)$, by
    Point~\ref{item:composition_inverse_1} of
    Proposition~\ref{prop:composition_inverse},
    $(\Ubf - \Rbf)^{\Compo_{-1}}$ is a well-defined series. Let us show
    that $\Angle{x, (\Ubf - \Rbf)^{\Compo_{-1}}} = \lambda_x$. First,
    when $x$ does not belong to $\Bud_\Ccr(\Oca)^\Rfr$, by
    Point~\ref{item:composition_inverse_2} of
    Proposition~\ref{prop:composition_inverse},
    $\Angle{x, (\Ubf - \Rbf)^{\Compo_{-1}}} = 0$. Since, in this case
    $\lambda_x = 0$, the property holds. Let us now assume that $x$
    belongs to $\Bud_\Ccr(\Oca)^\Rfr$. Again by
    Lemma~\ref{lem:finitely_factorizing_sets}, the $\Rfr$-degree of $x$
    is well-defined. Therefore, we proceed by induction on
    $\deg_\Rfr(x)$. By
    Lemma~\ref{lem:inverse_composition_syntactic_series}, when $x$ is a
    colored unit $\Unit_a$, $a \in \Ccr$, one has
    $\Angle{x, (\Ubf - \Rbf)^{\Compo_{-1}}} = 1$. Since there is exactly
    one $\Rfr$-treelike expression for $\Unit_a$, namely the syntax tree
    consisting in one leaf of output and input color $a$,
    $\lambda_{\Unit_a} = 1$ so that the base case holds. Otherwise,
    again by Lemma~\ref{lem:inverse_composition_syntactic_series}, we
    have, by using induction hypothesis,
    \begin{equation} \label{equ:syntactic_series_1}
        \Angle{x, (\Ubf - \Rbf)^{\Compo_{-1}}} =
        \sum_{\substack{
            y \in \Rfr \\
            z_1, \dots, z_{|y|} \in \Bud_\Ccr(\Oca) \\
            x = y \circ \Han{z_1, \dots, z_{|y|}}}}
        \prod_{i \in [|y|]} \lambda_{z_i}
        = \lambda_x.
    \end{equation}
    Notice that one can apply the induction hypothesis to
    state~\eqref{equ:syntactic_series_1} since one has
    $\deg_\Rfr(x) \geq 1 + \deg_\Rfr(z_i)$ for all $i \in [|y|]$.
    \smallbreak

    Now, from~\eqref{equ:syntactic_series_1} and by using
    Lemma~\ref{lem:initial_terminal_composition}, we obtain that for all
    $x \in \Bud_\Ccr(\Oca)$ such that $\Out(x) \in I$ and
    $\In(x) \in T^+$, $\Angle{x, \Synt(\Bca)} = \lambda_x$. Denoting by
    $\Fbf$ the series of the right member
    of~\eqref{equ:syntactic_series}, we have
    $\Angle{x, \Fbf} = \lambda_x$ if $\Out(x) \in I$ and
    $\In(x) \in T^+$, and $\Angle{x, \Fbf} = 0$ otherwise. This shows
    that this expression is equal to~$\Synt(\Bca)$.
\end{proof}
\medbreak

Theorem~\ref{thm:syntactic_series} explains the name of syntactic
generating series for $\Synt(\Bca)$ because this series can be
expressed following~\eqref{equ:syntactic_series} as a sum of evaluations
of syntax trees. An alternative way to see $\Synt(\Bca)$ is that for any
$x \in \Bud_\Ccr(\Oca)$, the coefficient $\Angle{x, \Synt(\Bca)}$ is the
number of $\Rfr$-treelike expressions for~$x$.
\medbreak

The following result establishes a link between the syntactic
generating series of $\Bca$ and its language.
\medbreak

\begin{Proposition} \label{prop:support_syntactic_series_language}
    Let $\Bca := (\Oca, \Ccr, \Rfr, I, T)$ be a bud generating system
    such that $\Oca$ is a locally finite operad and $\Rfr(1)$
    finitely factorizes $\Bud_\Ccr(\Oca)$. Then, the support of the
    syntactic generating series of $\Bca$ is the language of~$\Bca$.
\end{Proposition}
\begin{proof}
    This is an immediate consequence of
    Theorem~\ref{thm:syntactic_series} and
    Lemma~\ref{lem:language_treelike_expressions}.
\end{proof}
\medbreak

By Propositions~\ref{prop:support_hook_series_language}
and~\ref{prop:support_syntactic_series_language}, the series
$\Hook(\Bca)$ and $\Synt(\Bca)$ have the same support. The main
difference between these two series is that the coefficient of an
$x \in \Bud_\Ccr(\Oca)$ in $\Synt(\Bca)$ is the number of
$\Rfr$-treelike expressions for $x$, while in $\Hook(\Bca)$ this
coefficient is the number of ways to generate $x$ in $\Bca$.
\medbreak

We say that $\Bca$ is \Def{unambiguous} if all coefficients of
$\Synt(\Bca)$ are equal to $0$ or to $1$. This property is important
from a combinatorial and enumerative point of view. Indeed, when $\Bca$
is unambiguous, its syntactic generating series is the characteristic
series of its language. As a consequence, by definition of the series of
colors $\Colors$ (see Section~\ref{subsubsec:series_of_colors}) and
Proposition~\ref{prop:support_syntactic_series_language}, the
coefficient of $(a, u) \in \Bud_\Ccr(\As)$ in the series
$\Colors(\Synt(\Bca))$ is the number of elements $x$ of $\Lang(\Bca)$
such that $(\Out(x), \In(x)) = (a, u)$.
\medbreak

As a side remark, observe that Theorem~\ref{thm:syntactic_series}
implies in particular that for any bud generating system of the form
$\Bca := (\Oca, \Ccr, \Rfr, \Ccr, \Ccr)$, if $\Synt(\Bca)$ is
unambiguous, then the colored suboperad of $\Bud_\Ccr(\Oca)$ generated
by $\Rfr$ is free. The converse property does not hold.
\medbreak

Let us now describe the coefficients of
$\ColorTypes(\Synt(\Bca))$, the series of color types of the syntactic
series of $\Bca$, in the particular case when $\Bca$ is unambiguous.
We shall give two descriptions: a first one involving a system of
equations of series of $\K \Han{\Han{\Ybb_\Ccr}}$, and a second one
involving a recurrence relation on the coefficients of a series
of~$\K \Han{\Han{\Xbb_\Ccr \sqcup \Ybb_\Ccr}}$.
\medbreak

\begin{Lemma} \label{lem:colt_synt_coefficients_description}
    Let $\Bca := (\Oca, \Ccr, \Rfr, I, T)$ be an unambiguous bud
    generating system such that $\Oca$ is a locally finite operad and
    $\Rfr(1)$ finitely factorizes $\Bud_\Ccr(\Oca)$. Then, for all
    colors $a \in I$ and all types $\alpha \in \Types_\Ccr$ such that
    $\Ccr^\alpha \in T^+$, the coefficients
    $\Angle{\VarX_a \Ybb_\Ccr^\alpha, \ColorTypes(\Synt(\Bca))}$ count
    the number of elements $x$ of $\Lang(\Bca)$ such that
    $(\Out(x), \Type(\In(x))) = (a, \alpha)$.
\end{Lemma}
\begin{proof}
    By Proposition~\ref{prop:support_syntactic_series_language} and
    since $\Bca$ is unambiguous, $\Synt(\Bca)$ is the characteristic
    series of $\Lang(\Bca)$. The statement of the lemma follows
    immediately from the definition~\eqref{equ:definition_color_types}
    of~$\ColorTypes$.
\end{proof}
\medbreak

\begin{Proposition} \label{prop:functional_equation_synt}
    Let $\Bca := (\Oca, \Ccr, \Rfr, I, T)$ be an unambiguous bud
    generating system such that $\Oca$ is a locally finite operad and
    $\Rfr(1)$ finitely factorizes $\Bud_\Ccr(\Oca)$. For all
    $a \in \Ccr$, let $\Fbf_a\Par{\VarY_{c_1}, \dots, \VarY_{c_k}}$ be
    the series of $\K \Han{\Han{\Ybb_\Ccr}}$ satisfying
    \begin{equation}
        \Fbf_a\Par{\VarY_{c_1}, \dots, \VarY_{c_k}} =
        \VarY_a +
        \Gbf_a\Par{\Fbf_{c_1}\Par{\VarY_{c_1}, \dots, \VarY_{c_k}},
        \dots,
        \Fbf_{c_k}\Par{\VarY_{c_1}, \dots, \VarY_{c_k}}}.
    \end{equation}
    Then, for any color $a \in I$ and any type $\alpha \in \Types_\Ccr$
    such that $\Ccr^\alpha \in T^+$, the coefficients
    $\Angle{\VarX_a \Ybb_\Ccr^\alpha, \ColorTypes(\Synt(\Bca))}$ and
    $\Angle{\Ybb_\Ccr^\alpha, \Fbf_a}$ are equal.
\end{Proposition}
\begin{proof}
    Let us set $\Hbf := (\Ubf - \Rbf)^{\Compo_{-1}}$ and, for all
    $a \in \Ccr$, $\Hbf_a := \Unit_a \Compo \Hbf$. Since $\Rfr(1)$
    finitely factorizes $\Bud_\Ccr(\Oca)$, by
    Point~\ref{item:composition_inverse_1} of
    Proposition~\ref{prop:composition_inverse}, $\Hbf$ and $\Hbf_a$ are
    well-defined series.
    Equation~\eqref{equ:equation_composition_inverse} implies that any
    $\Hbf_a$, $a \in \Ccr$, satisfies the relation
    \begin{equation} \label{equ:functional_equation_synt_proof_1}
        \Hbf_a = \Unit_a + \Rbf_a \Compo \Hbf
    \end{equation}
    where $\Rbf_a := \Unit_a \Compo \Rbf$. Observe that for any
    $a \in \Ccr$,
    \begin{math}
        \ColorTypes\Par{\Rbf_a} =
        \Gbf_a\Par{\VarY_{c_1}, \dots, \VarY_{c_k}}.
    \end{math}
    Moreover, from the definitions of $\ColorTypes$ and the operation
    $\Compo$, we obtain that $\ColorTypes\Par{\Rbf_a \Compo \Hbf}$ can
    be computed by a functional composition of the series
    $\Gbf_a\Par{\VarY_{c_1}, \dots, \VarY_{c_k}}$ with
    $\Fbf_{c_1}\Par{\VarY_{c_1}, \dots, \VarY_{c_k}}$, \dots,
    $\Fbf_{c_k}\Par{\VarY_{c_1}, \dots, \VarY_{c_k}}$. Hence,
    Relation~\eqref{equ:functional_equation_synt_proof_1} leads to
    \begin{equation}\begin{split}
        \ColorTypes(\Hbf_a)
            & = \ColorTypes\Par{\Unit_a}
                + \ColorTypes\Par{\Rbf_a \Compo \Hbf} \\
            & = \VarY_a + \Gbf_a\Par{\Fbf_{c_1}\Par{\VarY_{c_1},
                \dots, \VarY_{c_k}},
                \dots,
                \Fbf_{c_k}\Par{\VarY_{c_1}, \dots,
                \VarY_{c_k}}} \\
            & = \Fbf_a\Par{\VarY_{c_1}, \dots, \VarY_{c_k}}.
    \end{split}\end{equation}
    Finally, Lemma~\ref{lem:initial_terminal_composition} implies that,
    when $a \in I$ and $\Ccr^\alpha \in T^+$,
    $\Angle{\VarX_a \Ybb_\Ccr^\alpha, \ColorTypes(\Synt(\Bca))}$ and
    $\Angle{\Ybb_\Ccr^\alpha, \Fbf_a}$ are equal.
\end{proof}
\medbreak

When $\Bca$ is a bud generating system satisfying the conditions of
Proposition~\ref{prop:functional_equation_synt}, the generating series
of the language of $\Bca$ satisfies
\begin{equation} \label{equ:generating_series_synt_functional_equation}
    \GenS_{\Lang(\Bca)} = \sum_{a \in I} \Fbf_a^{T},
\end{equation}
where $\Fbf_a^{T}$ is the specialization of the series
$\Fbf_a\Par{\VarY_{c_1}, \dots, \VarY_{c_k}}$ at $\VarY_b := t$ for all
$b \in T$ and at $\VarY_c := 0$ for all $c \in \Ccr \setminus T$.
Therefore, the resolution of the system of equations given by
Proposition~\ref{prop:functional_equation_synt} provides a way to
compute the coefficients of~$\GenS_{\Lang(\Bca)}$.
\medbreak

\begin{Theorem} \label{thm:algebraic_generating_series_languages}
    Let $\Bca := (\Oca, \Ccr, \Rfr, I, T)$ be an unambiguous bud
    generating system such that $\Oca$ is a locally finite operad and
    $\Rfr(1)$ finitely factorizes $\Bud_\Ccr(\Oca)$. Then, the
    generating series $\GenS_{\Lang(\Bca)}$ of the language of $\Bca$ is
    algebraic.
\end{Theorem}
\begin{proof}
    Proposition~\ref{prop:functional_equation_synt} shows that each
    series $\Fbf_a$ satisfies an algebraic equation involving variables
    of $\Ybb_\Ccr$ and series $\Fbf_b$, $b \in \Ccr$. Hence, $\Fbf_a$ is
    algebraic. Moreover, the fact that,
    by~\eqref{equ:generating_series_synt_functional_equation},
    $\GenS_{\Lang(\Bca)}$ is a specialized sum of some $\Fbf_a$ implies
    the statement of the theorem.
\end{proof}
\medbreak

\begin{Theorem} \label{thm:series_color_types_synt}
    Let $\Bca := (\Oca, \Ccr, \Rfr, I, T)$ be an unambiguous bud
    generating system such that $\Oca$ is a locally finite operad and
    $\Rfr(1)$ finitely factorizes $\Bud_\Ccr(\Oca)$ Let $\Fbf$ be
    the series of $\K \Han{\Han{\Xbb_\Ccr \sqcup \Ybb_\Ccr}}$
    satisfying, for any $a \in \Ccr$ and any type
    $\alpha \in \Types_\Ccr$,
    \begin{equation} \label{equ:series_color_types_synt}
        \Angle{\VarX_a \Ybb_\Ccr^\alpha, \Fbf}
        =
        \delta_{\alpha, \Type(a)} +
        \sum_{\substack{
            \phi : \Ccr \times \Types_\Ccr \to \N \\
            \alpha = \phi^{(c_1)} \dots \phi^{(c_k)}
        }}
        \MultOutIn_{a, \sum \phi_{c_1} \dots \sum \phi_{c_k}}
        \Par{\prod_{b \in \Ccr} \phi_b !}
        \Par{
        \prod_{\substack{
            b \in \Ccr \\
            \gamma \in \Types_\Ccr
        }}
        \Angle{\VarX_b \Ybb_\Ccr^\gamma, \Fbf}^{\phi(b, \gamma)}
        }.
    \end{equation}
    Then, for any color $a \in I$ and any type $\alpha \in \Types_\Ccr$
    such that $\Ccr^\alpha \in T^+$, the coefficients
    $\Angle{\VarX_a \Ybb_\Ccr^\alpha, \ColorTypes(\Synt(\Bca))}$
    and $\Angle{\VarX_a \Ybb_\Ccr^\alpha, \Fbf}$ are equal.
\end{Theorem}
\begin{proof}
    First, since $\Rfr(1)$ finitely factorizes $\Bud_\Ccr(\Oca)$, by
    Point~\ref{item:composition_inverse_1} of
    Proposition~\ref{prop:composition_inverse},
    $(\Ubf - \Rbf)^{\Compo_{-1}}$ is a well-defined series. Moreover,
    by~\eqref{equ:equation_composition_inverse},
    $(\Ubf - \Rbf)^{\Compo_{-1}}$ satisfies the identity of series
    \begin{equation} \label{equ:series_color_types_synt_proof_1}
        (\Ubf - \Rbf) \Compo (\Ubf - \Rbf)^{\Compo_{-1}} = \Ubf.
    \end{equation}
    Since the map $\Colors$ commutes with the addition of series, with
    the composition product $\Compo$, and with the inverse with respect
    to $\Compo$, \eqref{equ:series_color_types_synt_proof_1} leads to
    the equation
    \begin{equation} \label{equ:series_color_types_synt_proof_2}
        \Colors(\Ubf - \Rbf) \Compo
        \Colors(\Ubf - \Rbf)^{\Compo_{-1}}
        = \Colors(\Ubf).
    \end{equation}
    By Point~\ref{item:composition_inverse_2} of
    Proposition~\ref{prop:composition_inverse},
    by~\eqref{equ:induction_composition_inverse}, and by definition of
    the composition map of $\Bud_\Ccr(\As)$, the coefficients of
    $\Colors(\Ubf - \Rbf)^{\Compo_{-1}}$ satisfy, for all
    $(a, u) \in \Bud_\Ccr(\As)$, the recurrence relation
    \begin{equation} \label{equ:series_color_types_synt_proof_3}
        \Angle{(a, u), \Colors(\Ubf - \Rbf)^{\Compo_{-1}}} =
        \delta_{u, a} +
        \sum_{\substack{
            w \in \Ccr^+ \\
            w \ne a
        }}
        \enspace
        \lambda_{a, w}
        \sum_{\substack{
            v^{(1)}, \dots, v^{(|w|)} \in \Ccr^+ \\
            u = v^{(1)} \dots v^{(|w|)}
        }}
        \enspace
        \prod_{i \in [|w|]}
        \Angle{\Par{w_i, v^{(i)}},
        \Colors(\Ubf - \Rbf)^{\Compo_{-1}}},
    \end{equation}
    where $\lambda_{a, w}$ denotes the number of rules $r \in \Rfr$ such
    that $\Out(r) = a$ and $\In(r) = w$. By definition of $\ColorTypes$
    and by~\eqref{equ:series_color_types_synt_proof_3}, a
    straightforward computation shows that the coefficients of
    $\ColorTypes\Par{(\Ubf - \Rbf)^{\Compo_{-1}}}$ express for any
    $\alpha \in \Types_\Ccr$, as
    \begin{multline} \label{equ:series_color_types_synt_proof_4}
        \Angle{\VarX_a \Ybb_\Ccr^\alpha,
        \ColorTypes\Par{(\Ubf - \Rbf)^{\Compo_{-1}}}} \\
        =
        \delta_{\alpha, \Type(a)} +
        \sum_{\substack{
            \gamma \in \Types_\Ccr \\
            \gamma \ne \Type(a)
        }}
        \enspace
        \MultOutIn_{a, \gamma}
        \sum_{\substack{
            \beta^{(1)}, \dots, \beta^{(\deg(\gamma))}
                \in \Types_\Ccr \\
            \alpha = \beta^{(1)} \dot{+} \cdots
                \dot{+} \beta^{(\deg(\gamma))}
        }}
        \enspace
        \prod_{i \in [\deg(\gamma)]}
        \Angle{\VarX_{\Ccr^\gamma_i} \Ybb_\Ccr^{\beta^{(i)}},
        \ColorTypes\Par{(\Ubf - \Rbf)^{\Compo_{-1}}}}.
    \end{multline}
    Therefore, \eqref{equ:series_color_types_synt_proof_4} provides a
    recurrence relation for the coefficients of
    $\ColorTypes\Par{(\Ubf - \Rbf)^{\Compo_{-1}}}$. By using the
    notations introduced in
    Section~\ref{subsubsec:elementary_series_bud_generating_systems}
    about mappings $\phi : \Ccr \times \Types_\Ccr \to \N$, we obtain
    that the coefficients of
    $\ColorTypes\Par{(\Ubf - \Rbf)^{\Compo_{-1}}}$ satisfy the same
    recurrence relation~\eqref{equ:series_color_types_synt} as the ones
    of $\Fbf$. Finally, Lemma~\ref{lem:initial_terminal_composition}
    implies that, when $a \in I$ and $\Ccr^\alpha \in T^+$,
    $\Angle{\VarX_a \Ybb_\Ccr^\alpha, \ColorTypes(\Synt(\Bca))}$ and
    $\Angle{\VarX_a \Ybb_\Ccr^\alpha, \Fbf}$ are equal.
\end{proof}
\medbreak

When $\Bca$ is a bud generating system satisfying the conditions of
Theorem~\ref{thm:series_color_types_synt} (which are the same as
the ones required by Proposition~\ref{prop:functional_equation_synt}),
one has for any $n \geq 1$,
\begin{equation} \label{equ:generating_series_synt_recurrence}
    \Angle{t^n, \GenS_{\Lang(\Bca)}}
    =
    \sum_{a \in I}
    \sum_{\substack{
        \alpha \in \Types_\Ccr \\
        \alpha_i = 0, c_i \in \Ccr \setminus T
    }}
    \Angle{\VarX_a \Ybb_\Ccr^\alpha, \Fbf}.
\end{equation}
Therefore, this provides an alternative and recursive way to compute the
coefficients of $\GenS_{\Lang(\Bca)}$, different from the one of
Proposition~\ref{prop:functional_equation_synt}%
\footnote{See an example of computation of a series
$\GenS_{\Lang(\Bca)}$ in Section~\ref{subsubsec:colt_synt_bubtree}.}
.
\medbreak

\subsection{Series and synchronous languages}
We introduce the Kleene star operation of the composition
product~$\Compo$ in order to define the synchronous generating series of
a bud generating system $\Bca$. We relate this series with the
synchronous language of $\Bca$ and provide ways to compute its
coefficients.
\medbreak

Proofs of some results of this section are very similar to ones of
Section~\ref{subsec:series_languages}. For this reason, some proofs
are sketched here.
\medbreak

\subsubsection{Composition star product}
For any $\Cca$-series $\Fbf \in \K \AAngle{\Cca}$ and any $\ell \geq 0$,
let $\Fbf^{\Compo_\ell}$ be the series defined by
\begin{equation} \label{equ:definition_composition_power}
    \Fbf^{\Compo_\ell} :=
    \prod_{1 \leq i \leq \ell}
    \Fbf,
\end{equation}
where the product of~\eqref{equ:definition_composition_power} denotes
the iterated version of the composition product $\Compo$. Observe
that since $\Compo$ is associative (see
Proposition~\ref{prop:functor_series_monoids}), this definition is
consistent. Immediately from this definition and the definition of the
composition product $\Compo$, the coefficient of $\Fbf^{\Compo_\ell}$,
$\ell \geq 0$, satisfies for any $x \in \Cca$,
\begin{equation} \label{equ:induction_composition_powers}
    \Angle{x, \Fbf^{\Compo_\ell}} =
    \begin{cases}
        \delta_{x, \Unit_{\Out(x)}} & \mbox{if } \ell = 0, \\
        \sum\limits_{
        \substack{
            y, z_1, \dots, z_{|y|} \in \Cca \\
            x = y \circ \Han{z_1, \dots, z_{|y|}}}}
        \Angle{y, \Fbf^{\Compo_{\ell - 1}}}
        \prod\limits_{i \in [|y|]} \Angle{z_i, \Fbf}
        & \mbox{otherwise}.
    \end{cases}
\end{equation}
\medbreak

\begin{Lemma} \label{lem:composition_powers}
    Let $\Cca$ be a locally finite $\Ccr$-colored operad and $\Fbf$ be a
    series of $\K \AAngle{\Cca}$. Then, the coefficients of
    $\Fbf^{\Compo_{\ell + 1}}$, $\ell \geq 0$, satisfy for any
    $x \in \Cca$,
    \begin{equation} \label{equ:composition_powers}
        \Angle{x, \Fbf^{\Compo_{\ell + 1}}} =
        \sum_{\substack{
            \Tfr \in \Free_\Perfect(\Cca) \\
            \Height(\Tfr) = \ell + 1 \\
            \Eval_\Cca(\Tfr) = x
        }}
        \enspace
        \prod_{v \in \INodes(\Tfr)} \Angle{\Tfr(v), \Fbf}.
    \end{equation}
\end{Lemma}
\begin{proof}
    By Proposition~\ref{prop:functor_series_monoids}, since $\Cca$ is
    locally finite, $\Fbf^{\Compo_{\ell + 1}}$ is a well-defined
    $\Cca$-series. The statement of the lemma follows by induction on
    $\ell$ and by using~\eqref{equ:induction_composition_powers}.
\end{proof}
\medbreak

The \Def{$\Compo$-star} of $\Fbf$ is the series
\begin{equation}
    \Fbf^{\Compo_*} := \sum_{\ell \geq 0} \Fbf^{\Compo_\ell}
    = \Ubf + \Fbf + \Fbf \Compo \Fbf + \Fbf \Compo \Fbf \Compo \Fbf +
         \Fbf \Compo \Fbf \Compo \Fbf \Compo \Fbf + \cdots.
\end{equation}
Observe that $\Fbf^{\Compo_*}$ could be undefined for an arbitrary
$\Cca$-series~$\Fbf$.
\medbreak

\begin{Proposition} \label{prop:composition_star}
    Let $\Cca$ be a locally finite $\Ccr$-colored operad and $\Fbf$ be a
    series of $\K \AAngle{\Cca}$ such that $\Supp(\Fbf)(1)$
    finitely factorizes $\Cca$. Then,
    \begin{enumerate}[label = ({\it {\roman*})}]
        \item \label{item:composition_star_1}
        the series $\Fbf^{\Compo_*}$ is well-defined;
        \item \label{item:composition_star_2}
        for any $x \in \Cca$, the coefficient of $x$ in
        $\Fbf^{\Compo_*}$ satisfies
        \begin{equation} \label{equ:coefficients_composition_star}
            \Angle{x, \Fbf^{\Compo_*}} =
            \delta_{x, \Unit_{\Out(x)}} +
            \sum_{\substack{
                y, z_1, \dots, z_{|y|} \in \Cca \\
                x = y \circ \Han{z_1, \dots, z_{|y|}}
            }}
            \Angle{y, \Fbf^{\Compo_*}}
            \prod_{i \in [|y|]} \Angle{z_i, \Fbf};
        \end{equation}
        \item \label{item:composition_star_3}
        the equation
        \begin{equation} \label{equ:equation_composition_star}
            \Xbf - \Xbf \Compo \Fbf = \Ubf
        \end{equation}
        admits $\Xbf = \Fbf^{\Compo_*}$ as unique solution.
    \end{enumerate}
\end{Proposition}
\begin{proof}
    The proof is similar to the one of
    Proposition~\ref{prop:pre_lie_star} and
    uses~\eqref{equ:induction_composition_powers} and
    Lemmas~\ref{lem:finitely_factorizing_sets}
    and~\ref{lem:composition_powers}.
\end{proof}
\medbreak

In particular, Point~\ref{item:composition_star_2} of
Proposition~\ref{prop:composition_star} gives a way, given a
$\Cca$-series $\Fbf$ satisfying the stated constraints, to compute
recursively the coefficients of its $\Compo$-star~$\Fbf^{\Compo_*}$.
\medbreak

\subsubsection{Synchronous generating series}%
\label{subsubsec:synchronous_generating_series}
The \Def{synchronous generating series} of $\Bca$ the
$\Bud_\Ccr(\Oca)$-series $\Sync(\Bca)$ defined by
\begin{equation} \label{equ:definition_synchronous_series}
    \Sync(\Bca) := \Ibf \Compo \Rbf^{\Compo_*} \Compo \Tbf.
\end{equation}
Observe that~\eqref{equ:definition_synchronous_series} could be
undefined for an arbitrary set of rules $\Rfr$ of $\Bca$. Nevertheless,
when $\Rbf$ satisfies the conditions of
Proposition~\ref{prop:composition_star}, that is, when $\Oca$ is a
locally finite operad and $\Rfr(1)$ finitely factorizes
$\Bud_\Ccr(\Oca)$, $\Sync(\Bca)$ is well-defined.
\medbreak

The aim of the following is to provide an expression to compute the
coefficients of~$\Sync(\Bca)$.
\medbreak

\begin{Lemma} \label{lem:composition_star_synchronous_series}
    Let $\Bca := (\Oca, \Ccr, \Rfr, I, T)$ be a bud generating system
    such that $\Oca$ is a locally finite operad and $\Rfr(1)$
    finitely factorizes $\Bud_\Ccr(\Oca)$. Then, for any
    $x \in \Bud_\Ccr(\Oca)$,
    \begin{equation} \label{equ:composition_star_synchronous_series}
        \Angle{x, \Rbf^{\Compo_*}} =
        \delta_{x, \Unit_{\Out(x)}} +
        \sum_{\substack{
            y \in \Bud_\Ccr(\Oca) \\
            z_1, \dots, z_{|y|} \in \Rfr \\
            x = y \circ \Han{z_1, \dots, z_{|y|}}
        }}
        \Angle{y, \Rbf^{\Compo_*}}.
    \end{equation}
\end{Lemma}
\begin{proof}
    The proof is similar to the one of
    Lemma~\ref{lem:pre_lie_star_hook_series} and uses
    Proposition~\ref{prop:composition_star}.
\end{proof}
\medbreak

\begin{Theorem} \label{thm:synchronous_series}
    Let $\Bca := (\Oca, \Ccr, \Rfr, I, T)$ be a bud generating system
    such that $\Oca$ is a locally finite operad and $\Rfr(1)$
    finitely factorizes $\Bud_\Ccr(\Oca)$. Then, the synchronous
    generating series of $\Bca$ satisfies
    \begin{equation} \label{equ:synchronous_series}
        \Sync(\Bca)
        = \sum_{
            \substack{\Tfr \in \Free_{\Perfect}(\Rfr) \\
            \Out(\Tfr) \in I \\
            \In(\Tfr) \in T^+}}
        \Eval_{\Bud_\Ccr(\Oca)}(\Tfr).
    \end{equation}
\end{Theorem}
\begin{proof}
    Let, for any $x \in \Bud_\Ccr(\Oca)$, $\lambda_x$ be the number of
    perfect $\Rfr$-treelike expressions for $x$. Since $\Rfr(1)$
    finitely factorizes $\Bud_\Ccr(\Oca)$, by
    Lemma~\ref{lem:finitely_factorizing_sets}, all $\lambda_x$ are
    well-defined integers. Moreover, since $\Rfr(1)$ finitely factorizes
    $\Bud_\Ccr(\Oca)$, by Point~\ref{item:composition_star_1} of
    Proposition~\ref{prop:composition_star}, $\Rbf^{\Compo_*}$ is a
    well-defined series. Let us show that
    $\Angle{x, \Rbf^{\Compo_*}} = \lambda_x$. First, when $x$ does not
    belong to $\Bud_\Ccr(\Oca)^\Rfr$, by
    Lemma~\ref{lem:composition_star_synchronous_series},
    $\Angle{x, \Rbf^{\Compo_*}} = 0$. Since, in this case
    $\lambda_x = 0$, the property holds. Let us now assume that $x$
    belongs to $\Bud_\Ccr(\Oca)^\Rfr$. Again by
    Lemma~\ref{lem:finitely_factorizing_sets}, the $\Rfr$-degree of $x$
    is well-defined. Therefore, we proceed by induction on
    $\deg_\Rfr(x)$. By
    Lemma~\ref{lem:composition_star_synchronous_series}, when $x$ is a
    colored unit $\Unit_a$, $a \in \Ccr$, one has
    $\Angle{x, \Rbf^{\Compo_*}}  = 1$. Since there is exactly one
    treelike expression which is a perfect tree for $\Unit_a$, namely
    the syntax tree consisting in one leaf of output and input color
    $a$, $\lambda_{\Unit_a} = 1$ so that the base case holds. Otherwise,
    again by Lemma~\ref{lem:composition_star_synchronous_series}, we
    have, by using induction hypothesis,
    \begin{equation} \label{equ:synchronous_series_1}
        \Angle{x, \Rbf^{\Compo_*}} =
        \sum_{\substack{
            y \in \Bud_\Ccr(\Oca) \\
            z_1, \dots, z_{|y|} \in \Rfr \\
            x = y \circ \Han{z_1, \dots, z_{|y|}}}}
        \lambda_y = \lambda_x.
    \end{equation}
    Notice that one can apply the induction hypothesis to
    state~\eqref{equ:synchronous_series_1} since one has
    $\deg_\Rfr(x) \geq 1 + \deg_\Rfr(y)$.
    \smallbreak

    Now, from~\eqref{equ:synchronous_series_1} and by using
    Lemma~\ref{lem:initial_terminal_composition}, we obtain that for all
    $x \in \Bud_\Ccr(\Oca)$ such that $\Out(x) \in I$ and
    $\In(x) \in T^+$, $\Angle{x, \Sync(\Bca)} = \lambda_x$. By denoting
    by $\Fbf$ the series of the right member
    of~\eqref{equ:synchronous_series}, we have
    $\Angle{x, \Fbf} = \lambda_x$ if $\Out(x) \in I$ and
    $\In(x) \in T^*$, and $\Angle{x, \Fbf} = 0$ otherwise. This shows
    that this expression is equal to~$\Sync(\Bca)$.
\end{proof}
\medbreak

Theorem~\ref{thm:synchronous_series} implies that for any
$x \in \Bud_\Ccr(\Oca)$, the coefficient of $\Angle{x, \Sync(\Bca)}$
is the number of $\Rfr$-treelike expressions for $x$ which are perfect
trees.
\medbreak

The following result establishes a link between the synchronous
generating series of $\Bca$ and its synchronous language.
\medbreak

\begin{Proposition} \label{prop:support_synchronous_series_language}
    Let $\Bca := (\Oca, \Ccr, \Rfr, I, T)$ be a bud generating system
    such that $\Oca$ is a locally finite operad and $\Rfr(1)$ finitely
    factorizes $\Bud_\Ccr(\Oca)$. Then, the support of the synchronous
    generating series of $\Bca$ is the synchronous language of~$\Bca$.
\end{Proposition}
\begin{proof}
    This is an immediate consequence of
    Theorem~\ref{thm:synchronous_series} and
    Lemma~\ref{lem:synchronous_language_treelike_expressions}.
\end{proof}
\medbreak

We say that $\Bca$ is \Def{synchronously unambiguous} if all
coefficients of $\Sync(\Bca)$ are equal to $0$ or to $1$. This property
is important to describe the coefficients of $\Colors(\Sync(\Bca))$
for the same reasons as the ones concerning the unambiguity property
exposed in Section~\ref{subsubsec:syntactic_generating_series}.
\medbreak

Let us now describe the coefficients of $\ColorTypes(\Sync(\Bca))$, the
series of colors types of the synchronous series of $\Bca$, in the
particular case when $\Bca$ is unambiguous. We shall give two
descriptions: a first one involving a system of functional equations of
series of $\K \Han{\Han{\Ybb_\Ccr}}$, and a second one involving a
recurrence relation on the coefficients of a series of
$\K \Han{\Han{\Xbb_\Ccr \sqcup \Ybb_\Ccr}}$.
\medbreak

\begin{Lemma} \label{lem:colt_sync_coefficients_description}
    Let $\Bca := (\Oca, \Ccr, \Rfr, I, T)$ be a synchronously
    unambiguous bud generating system such that $\Oca$ is a locally
    finite operad and $\Rfr(1)$ finitely factorizes $\Bud_\Ccr(\Oca)$.
    Then, for all colors $a \in I$ and all types
    $\alpha \in \Types_\Ccr$ such that $\Ccr^\alpha \in T^+$, the
    coefficients
    $\Angle{\VarX_a \Ybb_\Ccr^\alpha, \ColorTypes(\Sync(\Bca))}$ count
    the number of elements $x$ of $\SyncLang(\Bca)$ such that
    $(\Out(x), \Type(\In(x))) = (a, \alpha)$.
\end{Lemma}
\begin{proof}
    The proof is similar to the one of
    Lemma~\ref{lem:colt_synt_coefficients_description} and
    uses~\eqref{equ:definition_color_types} and
    Proposition~\ref{prop:support_synchronous_series_language}.
\end{proof}
\medbreak

\begin{Proposition} \label{prop:functional_equation_sync}
    Let $\Bca := (\Oca, \Ccr, \Rfr, I, T)$ be a synchronously
    unambiguous bud generating system such that $\Oca$ is a locally
    finite operad and $\Rfr(1)$ finitely factorizes $\Bud_\Ccr(\Oca)$.
    For all $a \in \Ccr$, let
    $\Fbf_a\Par{\VarY_{c_1}, \dots, \VarY_{c_k}}$ be the series of
    $\K \Han{\Han{\Ybb_\Ccr}}$ satisfying
    \begin{equation}
        \Fbf_a\Par{\VarY_{c_1}, \dots, \VarY_{c_k}} =
        \VarY_a +
        \Fbf_a\Par{\Gbf_{c_1}\Par{\VarY_{c_1}, \dots, \VarY_{c_k}},
        \dots,
        \Gbf_{c_k}\Par{\VarY_{c_1}, \dots, \VarY_{c_k}}}.
    \end{equation}
    Then, for any color $a \in I$ and any type $\alpha \in \Types_\Ccr$
    such that $\Ccr^\alpha \in T^+$, the coefficients
    $\Angle{\VarX_a \Ybb_\Ccr^\alpha, \ColorTypes(\Sync(\Bca))}$
    and $\Angle{\Ybb_\Ccr^\alpha, \Fbf_a}$ are equal.
\end{Proposition}
\begin{proof}
    The proof is similar to the one of
    Proposition~\ref{prop:functional_equation_synt} and uses
    Lemma~\ref{lem:initial_terminal_composition} and
    Proposition~\ref{prop:composition_star}.
\end{proof}
\medbreak

When $\Bca$ is a bud generating system satisfying the conditions of
Proposition~\ref{prop:functional_equation_sync}, the generating series
of the synchronous language of $\Bca$ satisfies
\begin{equation} \label{equ:generating_series_sync_functional_equation}
    \GenS_{\SyncLang(\Bca)} = \sum_{a \in I} \Fbf_a^{T},
\end{equation}
where $\Fbf_a^{T}$ is the specialization of the series
$\Fbf_a(\VarY_{c_1}, \dots, \VarY_{c_k})$ at $\VarY_b := t$ for all
$b \in T$ and at $\VarY_c := 0$ for all $c \in \Ccr \setminus T$.
Therefore, the resolution of the system of equations given by
Proposition~\ref{prop:functional_equation_sync} provides a way to
compute the coefficients of $\GenS_{\SyncLang(\Bca)}$. This resolution
can be made in most cases by iteration~\cite{BLL97,FS09}%
\footnote{See example of a computation of a series
$\GenS_{\SyncLang(\Bca)}$ by iteration in
Section~\ref{subsubsec:colt_sync_bbaltree}.}
. Moreover, when $\Gca$ is a synchronous grammar~\cite{Gir12} (see also
Section~\ref{subsubsec:synchronous_grammars} for a description of these
grammars) and when $\SG(\Gca) = \Bca$, the system of functional
equations provided by Proposition~\ref{prop:functional_equation_sync}
and~\eqref{equ:generating_series_sync_functional_equation} for
$\GenS_{\SyncLang(\Bca)}$ is the same as the one which can be extracted
from~$\Gca$.
\medbreak

\begin{Theorem} \label{thm:series_color_types_sync}
    Let $\Bca := (\Oca, \Ccr, \Rfr, I, T)$ be a synchronously
    unambiguous bud generating system such that $\Oca$ is a locally
    finite operad and $\Rfr(1)$ finitely factorizes $\Bud_\Ccr(\Oca)$.
    Let $\Fbf$ be the series of
    $\K \Han{\Han{\Xbb_\Ccr \sqcup \Ybb_\Ccr}}$ satisfying, for any
    $a \in \Ccr$ and any type $\alpha \in \Types_\Ccr$,
    \begin{equation} \label{equ:series_color_types_sync}
        \Angle{\VarX_a \Ybb_\Ccr^\alpha, \Fbf}
        =
        \delta_{\alpha, \Type(a)} +
        \sum_{\substack{
            \phi : \Ccr \times \Types_\Ccr \to \N \\
            \alpha = \phi^{(c_1)} \dots \phi^{(c_k)}
        }}
        \Par{
            \prod_{b \in \Ccr}
            \phi_b !
        }
        \Par{
            \prod_{\substack{
                b \in \Ccr \\
                \gamma \in \Types_\Ccr
            }}
            \MultOutIn_{b, \gamma}^{\phi(b, \gamma)}
        }
        \Angle{\VarX_a
        \prod_{b \in \Ccr} \VarY_b^{\sum \phi_b},
        \Fbf}.
    \end{equation}
    Then, for any color $a \in I$ and any type $\alpha \in \Types_\Ccr$
    such that $\Ccr^\alpha \in T^+$, the coefficients
    $\Angle{\VarX_a \Ybb_\Ccr^\alpha, \ColorTypes(\Sync(\Bca))}$ and
    $\Angle{\VarX_a \Ybb_\Ccr^\alpha, \Fbf}$ are equal.
\end{Theorem}
\begin{proof}
    The proof is similar to the one of
    Theorem~\ref{thm:series_color_types_synt} and uses
    Lemma~\ref{lem:initial_terminal_composition} and
    Proposition~\ref{prop:composition_star}.
\end{proof}
\medbreak

When $\Bca$ is a bud generating system satisfying the conditions of
Theorem~\ref{thm:series_color_types_sync} (which are the same as the
ones required by Proposition~\ref{prop:functional_equation_sync}), one
has for any $n \geq 1$,
\begin{equation} \label{equ:generating_series_sync_recurrence}
    \Angle{t^n, \GenS_{\SyncLang(\Bca)}}
    =
    \sum_{a \in I}
    \sum_{\substack{
        \alpha \in \Types_\Ccr \\
        \alpha_i = 0, c_i \in \Ccr \setminus T
    }}
    \Angle{\VarX_a \Ybb_\Ccr^\alpha, \Fbf}.
\end{equation}
Therefore, this provides an alternative and recursive way to compute the
coefficients of $\GenS_{\SyncLang(\Bca)}$, different from the one of
Proposition~\ref{prop:functional_equation_sync}%
\footnote{See examples of computations of series
$\GenS_{\SyncLang(\Bca)}$ in Sections~\ref{subsubsec:colt_sync_bbtree}
and~\ref{subsubsec:colt_sync_bbaltree}.}
.
\medbreak

\section{Examples} \label{sec:examples}
This final section is devoted to illustrate the notions and the results
contained in the previous ones. We first define here some monochrome
operads, then give examples of series on colored operads, and construct
some bud generating systems. We end this section by explaining how bud
generating systems can be used as tools for enumeration. For this
purpose, we use the syntactic and synchronous generating series of
several bud generating systems to compute the generating series of
various combinatorial objects.
\medbreak

\subsection{Monochrome operads and bud operads}%
\label{subsec:examples_monochrome_operads}
Let us start by defining three monochrome operads involving some
classical combinatorial objects: binary trees, some words of integers,
and Motzkin paths.
\medbreak

\subsubsection{The magmatic operad}\label{subsubsec:magmatic_operad}
A \Def{binary tree} is a planar rooted tree $\Tfr$ such that any
internal node of $\Tfr$ has two children. The \Def{magmatic operad}
$\Mag$ is the monochrome operad wherein $\Mag(n)$ is the set of all
binary trees with $n$ leaves. The partial composition
$\Sfr \circ_i \Tfr$ of two binary trees $\Sfr$ and $\Tfr$ is the binary
tree obtained by grafting the root of $\Tfr$ on the $i$-th leaf of
$\Sfr$. The only tree of $\Mag$ consisting in exactly one leaf is
denoted by $\TreeLeaf$ and is the unit of $\Mag$. Notice that $\Mag$ is
isomorphic to the operad $\Free(C)$ where $C$ is the monochrome graded
collection defined by $C := C(2) := \{\Att\}$.
\medbreak

For any set $\Ccr$ of colors, the bud operad $\Bud_\Ccr(\Mag)$ is the
$\Ccr$-graded colored collection of all binary trees $\Tfr$ where all
leaves (inputs) and the root (output) of $\Tfr$ are labeled on $\Ccr$.
For instance, in $\Bud_{\{1, 2, 3\}}(\Mag)$, one has
\begin{equation}
    \begin{tikzpicture}[xscale=.18,yscale=.15,Centering]
        \node[Leaf](0)at(0.00,-4.50){};
        \node[Leaf](2)at(2.00,-6.75){};
        \node[Leaf](4)at(4.00,-6.75){};
        \node[Leaf](6)at(6.00,-4.50){};
        \node[Leaf](8)at(8.00,-4.50){};
        \node[Node](1)at(1.00,-2.25){};
        \node[Node](3)at(3.00,-4.50){};
        \node[Node](5)at(5.00,0.00){};
        \node[Node](7)at(7.00,-2.25){};
        \draw[Edge](0)--(1);
        \draw[Edge](1)--(5);
        \draw[Edge](2)--(3);
        \draw[Edge](3)--(1);
        \draw[Edge](4)--(3);
        \draw[Edge](6)--(7);
        \draw[Edge](7)--(5);
        \draw[Edge](8)--(7);
        \node(r)at(5.00,2){};
        \draw[Edge](r)--(5);
        \node[LeafLabel,above of=r]{\begin{math}2\end{math}};
        \node[LeafLabel,below of=0]{\begin{math}2\end{math}};
        \node[LeafLabel,below of=2]{\begin{math}2\end{math}};
        \node[LeafLabel,below of=4]{\begin{math}1\end{math}};
        \node[LeafLabel,below of=6]{\begin{math}1\end{math}};
        \node[LeafLabel,below of=8]{\begin{math}3\end{math}};
    \end{tikzpicture}
    \enspace \circ_4 \enspace
    \begin{tikzpicture}[xscale=.18,yscale=.15,Centering]
        \node[Leaf](0)at(0.00,-1.75){};
        \node[Leaf](2)at(2.00,-5.25){};
        \node[Leaf](4)at(4.00,-5.25){};
        \node[Leaf](6)at(6.00,-3.50){};
        \node[Node,Mark4](1)at(1.00,0.00){};
        \node[Node,Mark4](3)at(3.00,-3.50){};
        \node[Node,Mark4](5)at(5.00,-1.75){};
        \draw[Edge](0)--(1);
        \draw[Edge](2)--(3);
        \draw[Edge](3)--(5);
        \draw[Edge](4)--(3);
        \draw[Edge](5)--(1);
        \draw[Edge](6)--(5);
        \node(r)at(1.00,2){};
        \draw[Edge](r)--(1);
        \node[LeafLabel,above of=r]{\begin{math}1\end{math}};
        \node[LeafLabel,below of=0]{\begin{math}1\end{math}};
        \node[LeafLabel,below of=2]{\begin{math}3\end{math}};
        \node[LeafLabel,below of=4]{\begin{math}3\end{math}};
        \node[LeafLabel,below of=6]{\begin{math}3\end{math}};
    \end{tikzpicture}
    \enspace = \enspace
    \begin{tikzpicture}[xscale=.17,yscale=.11,Centering]
        \node[Leaf](0)at(0.00,-5.00){};
        \node[Leaf](10)at(10.00,-12.50){};
        \node[Leaf](12)at(12.00,-10.00){};
        \node[Leaf](14)at(14.00,-5.00){};
        \node[Leaf](2)at(2.00,-7.50){};
        \node[Leaf](4)at(4.00,-7.50){};
        \node[Leaf](6)at(6.00,-7.50){};
        \node[Leaf](8)at(8.00,-12.50){};
        \node[Node](1)at(1.00,-2.50){};
        \node[Node,Mark4](11)at(11.00,-7.50){};
        \node[Node](13)at(13.00,-2.50){};
        \node[Node](3)at(3.00,-5.00){};
        \node[Node](5)at(5.00,0.00){};
        \node[Node,Mark4](7)at(7.00,-5.00){};
        \node[Node,Mark4](9)at(9.00,-10.00){};
        \draw[Edge](0)--(1);
        \draw[Edge](1)--(5);
        \draw[Edge](10)--(9);
        \draw[Edge](11)--(7);
        \draw[Edge](12)--(11);
        \draw[Edge](13)--(5);
        \draw[Edge](14)--(13);
        \draw[Edge](2)--(3);
        \draw[Edge](3)--(1);
        \draw[Edge](4)--(3);
        \draw[Edge](6)--(7);
        \draw[Edge](7)--(13);
        \draw[Edge](8)--(9);
        \draw[Edge](9)--(11);
        \node(r)at(5.00,2.5){};
        \draw[Edge](r)--(5);
        \node[LeafLabel,above of=r]{\begin{math}2\end{math}};
        \node[LeafLabel,below of=0]{\begin{math}2\end{math}};
        \node[LeafLabel,below of=2]{\begin{math}2\end{math}};
        \node[LeafLabel,below of=4]{\begin{math}1\end{math}};
        \node[LeafLabel,below of=6]{\begin{math}1\end{math}};
        \node[LeafLabel,below of=8]{\begin{math}3\end{math}};
        \node[LeafLabel,below of=10]{\begin{math}3\end{math}};
        \node[LeafLabel,below of=12]{\begin{math}3\end{math}};
        \node[LeafLabel,below of=14]{\begin{math}3\end{math}};
    \end{tikzpicture}\,.
\end{equation}
\medbreak

\subsubsection{The pluriassociative operad}
Let $\gamma$ be a nonnegative integer. The
\Def{$\gamma$-pluriassociative operad} $\Dias_\gamma$~\cite{Gir16a} is
the monochrome operad wherein $\Dias_\gamma(n)$ is the set of all words
of length $n$ on the alphabet $\{0\} \cup [\gamma]$ with exactly one
occurrence of $0$. The partial composition $u \circ_i v$ of two such
words $u$ and $v$ consists in replacing the $i$-th letter of $u$ by
$v'$, where $v'$ is the word obtained from $v$ by replacing all its
letters $a$ by the greatest integer in $\Bra{a, u_i}$. For instance, in
$\Dias_4$, one has
\begin{equation}
    \ColA{313 {\bf 3} 21} \circ_4 \ColD{4112}
    = \ColA{313} \ColD{4333} \ColA{21}.
\end{equation}
Observe that $\Dias_0$ is the operad $\As$ and that $\Dias_1$ is the
\Def{diassociative operad} introduced by Loday~\cite{Lod01}.
\medbreak

For any set $\Ccr$ of colors, the bud operad $\Bud_\Ccr(\Dias_\gamma)$
is the $\Ccr$-colored graded collection of all words $u$ of
$\Dias_\gamma$ where all letters of $u$ (inputs) and the whole word $u$
(output) are labeled on~$\Ccr$.
\medbreak

\subsubsection{The operad of Motzkin paths}
The \Def{operad of Motzkin paths} $\Motz$~\cite{Gir15}
is a monochrome operad where $\Motz(n)$ is the set of all Motzkin paths
consisting in $n - 1$ steps. A \Def{Motzkin path} of arity $n$ is a path
in $\N^2$ connecting the points $(0, 0)$ and $(n - 1, 0)$, and made of
steps $(1, 0)$, $(1, 1)$, and $(1, -1)$. If $\Afr$ is a Motzkin path,
the \Def{$i$-th point} of $\Afr$ is the point of $\Afr$ of abscissa
$i - 1$. The partial composition $\Afr \circ_i \Bfr$ of two Motzkin
paths $\Afr$ and $\Bfr$ consists in replacing the $i$-th point of $\Afr$
by $\Bfr$. For instance, in $\Motz$, one has
\begin{equation}
    \begin{tikzpicture}[scale=.25,Centering]
        \draw[Grid] (0,0) grid (7,3);
        \node[PathNode](0)at(0,0){};
        \node[PathNode](1)at(1,1){};
        \node[PathNode](2)at(2,1){};
        \node[PathNode,Mark6](3)at(3,2){};
        \node[PathNode](4)at(4,3){};
        \node[PathNode](5)at(5,2){};
        \node[PathNode](6)at(6,1){};
        \node[PathNode](7)at(7,0){};
        \draw[PathStep](0)--(1);
        \draw[PathStep](1)--(2);
        \draw[PathStep](2)--(3);
        \draw[PathStep](3)--(4);
        \draw[PathStep](4)--(5);
        \draw[PathStep](5)--(6);
        \draw[PathStep](6)--(7);
    \end{tikzpicture}
    \enspace \circ_4 \enspace
    \begin{tikzpicture}[scale=.25,Centering]
        \draw[Grid] (0,0) grid (6,2);
        \node[PathNode,Mark4](0)at(0,0){};
        \node[PathNode,Mark4](1)at(1,1){};
        \node[PathNode,Mark4](2)at(2,2){};
        \node[PathNode,Mark4](3)at(3,2){};
        \node[PathNode,Mark4](4)at(4,1){};
        \node[PathNode,Mark4](5)at(5,1){};
        \node[PathNode,Mark4](6)at(6,0){};
        \draw[PathStep](0)--(1);
        \draw[PathStep](1)--(2);
        \draw[PathStep](2)--(3);
        \draw[PathStep](3)--(4);
        \draw[PathStep](4)--(5);
        \draw[PathStep](5)--(6);
    \end{tikzpicture}
    \enspace = \enspace
    \begin{tikzpicture}[scale=.25,Centering]
        \draw[Grid] (0,0) grid (13,4);
        \node[PathNode](0)at(0,0){};
        \node[PathNode](1)at(1,1){};
        \node[PathNode](2)at(2,1){};
        \node[PathNode,Mark4](3)at(3,2){};
        \node[PathNode,Mark4](4)at(4,3){};
        \node[PathNode,Mark4](5)at(5,4){};
        \node[PathNode,Mark4](6)at(6,4){};
        \node[PathNode,Mark4](7)at(7,3){};
        \node[PathNode,Mark4](8)at(8,3){};
        \node[PathNode,Mark4](9)at(9,2){};
        \node[PathNode](10)at(10,3){};
        \node[PathNode](11)at(11,2){};
        \node[PathNode](12)at(12,1){};
        \node[PathNode](13)at(13,0){};
        \draw[PathStep](0)--(1);
        \draw[PathStep](1)--(2);
        \draw[PathStep](2)--(3);
        \draw[PathStep](3)--(4);
        \draw[PathStep](4)--(5);
        \draw[PathStep](5)--(6);
        \draw[PathStep](6)--(7);
        \draw[PathStep](7)--(8);
        \draw[PathStep](8)--(9);
        \draw[PathStep](9)--(10);
        \draw[PathStep](10)--(11);
        \draw[PathStep](11)--(12);
        \draw[PathStep](12)--(13);
    \end{tikzpicture}\,.
\end{equation}
\medbreak

For any set $\Ccr$ of colors, the bud operad $\Bud_\Ccr(\Motz)$ is the
$\Ccr$-colored graded collection of all Motzkin paths $\Afr$ where all
points of $\Afr$ (inputs) and the whole path $\Afr$ (output) are
labeled on~$\Ccr$.
\medbreak

\subsection{Series on colored operads}
Here, some examples of series on colored operads are constructed, as
well as examples of series of colors, series of color types, and pruned
series.
\medbreak

\subsubsection{Series of trees}%
\label{subsubsec:example_series_trees}
Let $\Cca$ be the free $\Ccr$-colored operad over $C$ where
$\Ccr := \{1, 2\}$ and $C$ is the $\Ccr$-graded collection defined by
$C := C(2) \sqcup C(3)$ with $C(2) := \{\Att\}$, $C(3) := \{\Btt\}$,
$\Out(\Att) := 1$, $\Out(\Btt) := 2$, $\In(\Att) := 21$, and
$\In(\Btt) := 121$. Let $\Fbf_\Att$ (resp. $\Fbf_\Btt$) be the series of
$\K \AAngle{\Cca}$ where for any syntax tree $\Tfr$ of $\Cca$,
$\Angle{\Tfr, \Fbf_\Att}$ (resp. $\Angle{\Tfr, \Fbf_\Btt}$) is the
number of internal nodes of $\Tfr$ labeled by $\Att$ (resp. $\Btt$).
The series $\Fbf_\Att$ and $\Fbf_\Btt$ are of the form
\begin{subequations}
\begin{equation}
    \Fbf_\Att =
    \begin{tikzpicture}[xscale=.3,yscale=.28,Centering]
        \node(0)at(0.00,-1.50){};
        \node(2)at(2.00,-1.50){};
        \node[NodeST](1)at(1.00,0.00){\begin{math}\Att\end{math}};
        \draw[Edge](0)--(1);
        \draw[Edge](2)--(1);
        \node[NodeST](r)at(1.00,1.5){};
        \draw[Edge](r)--(1);
        \node[LeafLabel,above of=r]{\begin{math}1\end{math}};
        \node[LeafLabel,below of=0]{\begin{math}2\end{math}};
        \node[LeafLabel,below of=2]{\begin{math}1\end{math}};
    \end{tikzpicture}
    \enspace + \enspace
    2\,
    \begin{tikzpicture}[xscale=.27,yscale=.25,Centering]
        \node(0)at(0.00,-1.67){};
        \node(2)at(2.00,-3.33){};
        \node(4)at(4.00,-3.33){};
        \node[NodeST](1)at(1.00,0.00){\begin{math}\Att\end{math}};
        \node(3)at(3.00,-1.67){\begin{math}\Att\end{math}};
        \draw[Edge](0)--(1);
        \draw[Edge](2)--(3);
        \draw[Edge](3)--(1);
        \draw[Edge](4)--(3);
        \node[NodeST](r)at(1.00,1.35){};
        \draw[Edge](r)--(1);
        \node[LeafLabel,above of=r]{\begin{math}1\end{math}};
        \node[LeafLabel,below of=0]{\begin{math}2\end{math}};
        \node[LeafLabel,below of=2]{\begin{math}2\end{math}};
        \node[LeafLabel,below of=4]{\begin{math}1\end{math}};
    \end{tikzpicture}
    \enspace + \enspace
    3\,
    \begin{tikzpicture}[xscale=.25,yscale=.23,Centering]
        \node(0)at(0.00,-1.75){};
        \node(2)at(2.00,-3.50){};
        \node(4)at(4.00,-5.25){};
        \node(6)at(6.00,-5.25){};
        \node[NodeST](1)at(1.00,0.00){\begin{math}\Att\end{math}};
        \node[NodeST](3)at(3.00,-1.75){\begin{math}\Att\end{math}};
        \node[NodeST](5)at(5.00,-3.50){\begin{math}\Att\end{math}};
        \draw[Edge](0)--(1);
        \draw[Edge](2)--(3);
        \draw[Edge](3)--(1);
        \draw[Edge](4)--(5);
        \draw[Edge](5)--(3);
        \draw[Edge](6)--(5);
        \node[NodeST](r)at(1.00,1.5){};
        \draw[Edge](r)--(1);
        \node[LeafLabel,above of=r]{\begin{math}1\end{math}};
        \node[LeafLabel,below of=0]{\begin{math}2\end{math}};
        \node[LeafLabel,below of=2]{\begin{math}2\end{math}};
        \node[LeafLabel,below of=4]{\begin{math}2\end{math}};
        \node[LeafLabel,below of=6]{\begin{math}1\end{math}};
    \end{tikzpicture}
    \enspace + \enspace
    \begin{tikzpicture}[xscale=.3,yscale=.24,Centering]
        \node(0)at(0.00,-4.00){};
        \node(2)at(2.00,-4.00){};
        \node(4)at(3.00,-2.00){};
        \node(5)at(4.00,-2.00){};
        \node[NodeST](1)at(1.00,-2.00){\begin{math}\Att\end{math}};
        \node[NodeST](3)at(3.00,0.00){\begin{math}\Btt\end{math}};
        \draw[Edge](0)--(1);
        \draw[Edge](1)--(3);
        \draw[Edge](2)--(1);
        \draw[Edge](4)--(3);
        \draw[Edge](5)--(3);
        \node[NodeST](r)at(3.00,1.6){};
        \draw[Edge](r)--(3);
        \node[LeafLabel,above of=r]{\begin{math}2\end{math}};
        \node[LeafLabel,below of=0]{\begin{math}2\end{math}};
        \node[LeafLabel,below of=2]{\begin{math}1\end{math}};
        \node[LeafLabel,below of=4]{\begin{math}2\end{math}};
        \node[LeafLabel,below of=5]{\begin{math}1\end{math}};
    \end{tikzpicture}
    \enspace + \enspace
    \begin{tikzpicture}[xscale=.3,yscale=.24,Centering]
        \node(0)at(0.00,-2.00){};
        \node(2)at(1.00,-2.00){};
        \node(3)at(2.00,-4.00){};
        \node(5)at(4.00,-4.00){};
        \node[NodeST](1)at(1.00,0.00){\begin{math}\Btt\end{math}};
        \node[NodeST](4)at(3.00,-2.00){\begin{math}\Att\end{math}};
        \draw[Edge](0)--(1);
        \draw[Edge](2)--(1);
        \draw[Edge](3)--(4);
        \draw[Edge](4)--(1);
        \draw[Edge](5)--(4);
        \node[NodeST](r)at(1.00,1.5){};
        \draw[Edge](r)--(1);
        \node[LeafLabel,above of=r]{\begin{math}2\end{math}};
        \node[LeafLabel,below of=0]{\begin{math}1\end{math}};
        \node[LeafLabel,below of=2]{\begin{math}2\end{math}};
        \node[LeafLabel,below of=3]{\begin{math}2\end{math}};
        \node[LeafLabel,below of=5]{\begin{math}1\end{math}};
    \end{tikzpicture}
    \enspace + \enspace
    \begin{tikzpicture}[xscale=.27,yscale=.24,Centering]
        \node(0)at(0.00,-4.00){};
        \node(2)at(1.00,-4.00){};
        \node(3)at(2.00,-4.00){};
        \node(5)at(4.00,-2.00){};
        \node[NodeST](1)at(1.00,-2.00){\begin{math}\Btt\end{math}};
        \node[NodeST](4)at(3.00,0.00){\begin{math}\Att\end{math}};
        \draw[Edge](0)--(1);
        \draw[Edge](1)--(4);
        \draw[Edge](2)--(1);
        \draw[Edge](3)--(1);
        \draw[Edge](5)--(4);
        \node[NodeST](r)at(3.00,1.50){};
        \draw[Edge](r)--(4);
        \node[LeafLabel,above of=r]{\begin{math}1\end{math}};
        \node[LeafLabel,below of=0]{\begin{math}1\end{math}};
        \node[LeafLabel,below of=2]{\begin{math}2\end{math}};
        \node[LeafLabel,below of=3]{\begin{math}1\end{math}};
        \node[LeafLabel,below of=5]{\begin{math}1\end{math}};
    \end{tikzpicture}
    \enspace + \cdots.
\end{equation}
\begin{equation}
    \Fbf_\Btt =
    \begin{tikzpicture}[xscale=.35,yscale=.27,Centering]
        \node(0)at(0.00,-2.00){};
        \node(2)at(1.00,-2.00){};
        \node(3)at(2.00,-2.00){};
        \node[NodeST](1)at(1.00,0.00){\begin{math}\Btt\end{math}};
        \draw[Edge](0)--(1);
        \draw[Edge](2)--(1);
        \draw[Edge](3)--(1);
        \node[NodeST](r)at(1.00,1.50){};
        \draw[Edge](r)--(1);
        \node[LeafLabel,above of=r]{\begin{math}2\end{math}};
        \node[LeafLabel,below of=0]{\begin{math}1\end{math}};
        \node[LeafLabel,below of=2]{\begin{math}2\end{math}};
        \node[LeafLabel,below of=3]{\begin{math}1\end{math}};
    \end{tikzpicture}
    \enspace + \enspace
    \begin{tikzpicture}[xscale=.3,yscale=.24,Centering]
        \node(0)at(0.00,-4.00){};
        \node(2)at(2.00,-4.00){};
        \node(4)at(3.00,-2.00){};
        \node(5)at(4.00,-2.00){};
        \node[NodeST](1)at(1.00,-2.00){\begin{math}\Att\end{math}};
        \node[NodeST](3)at(3.00,0.00){\begin{math}\Btt\end{math}};
        \draw[Edge](0)--(1);
        \draw[Edge](1)--(3);
        \draw[Edge](2)--(1);
        \draw[Edge](4)--(3);
        \draw[Edge](5)--(3);
        \node[NodeST](r)at(3.00,1.6){};
        \draw[Edge](r)--(3);
        \node[LeafLabel,above of=r]{\begin{math}2\end{math}};
        \node[LeafLabel,below of=0]{\begin{math}2\end{math}};
        \node[LeafLabel,below of=2]{\begin{math}1\end{math}};
        \node[LeafLabel,below of=4]{\begin{math}2\end{math}};
        \node[LeafLabel,below of=5]{\begin{math}1\end{math}};
    \end{tikzpicture}
    \enspace + \enspace
    \begin{tikzpicture}[xscale=.3,yscale=.24,Centering]
        \node(0)at(0.00,-2.00){};
        \node(2)at(1.00,-2.00){};
        \node(3)at(2.00,-4.00){};
        \node(5)at(4.00,-4.00){};
        \node[NodeST](1)at(1.00,0.00){\begin{math}\Btt\end{math}};
        \node[NodeST](4)at(3.00,-2.00){\begin{math}\Att\end{math}};
        \draw[Edge](0)--(1);
        \draw[Edge](2)--(1);
        \draw[Edge](3)--(4);
        \draw[Edge](4)--(1);
        \draw[Edge](5)--(4);
        \node[NodeST](r)at(1.00,1.6){};
        \draw[Edge](r)--(1);
        \node[LeafLabel,above of=r]{\begin{math}2\end{math}};
        \node[LeafLabel,below of=0]{\begin{math}1\end{math}};
        \node[LeafLabel,below of=2]{\begin{math}2\end{math}};
        \node[LeafLabel,below of=3]{\begin{math}2\end{math}};
        \node[LeafLabel,below of=5]{\begin{math}1\end{math}};
    \end{tikzpicture}
    \enspace + \enspace 2\,
    \begin{tikzpicture}[xscale=.35,yscale=.23,Centering]
        \node(0)at(0.00,-2.33){};
        \node(2)at(1.00,-4.67){};
        \node(4)at(2.00,-4.67){};
        \node(5)at(3.00,-4.67){};
        \node(6)at(4.00,-2.33){};
        \node[NodeST](1)at(2.00,0.00){\begin{math}\Btt\end{math}};
        \node[NodeST](3)at(2.00,-2.33){\begin{math}\Btt\end{math}};
        \draw[Edge](0)--(1);
        \draw[Edge](2)--(3);
        \draw[Edge](3)--(1);
        \draw[Edge](4)--(3);
        \draw[Edge](5)--(3);
        \draw[Edge](6)--(1);
        \node[NodeST](r)at(2.00,1.75){};
        \draw[Edge](r)--(1);
        \node[LeafLabel,above of=r]{\begin{math}2\end{math}};
        \node[LeafLabel,below of=0]{\begin{math}1\end{math}};
        \node[LeafLabel,below of=2]{\begin{math}1\end{math}};
        \node[LeafLabel,below of=4]{\begin{math}2\end{math}};
        \node[LeafLabel,below of=5]{\begin{math}1\end{math}};
        \node[LeafLabel,below of=6]{\begin{math}1\end{math}};
    \end{tikzpicture}
    \enspace + \cdots.
\end{equation}
\end{subequations}
The sum $\Fbf_\Att + \Fbf_\Btt$ is the series wherein the coefficient of
any syntax tree $\Tfr$ of $\Cca$ is its degree. Let also $\Fbf_{|_1}$
(resp. $\Fbf_{|_2}$) be the series of $\K \AAngle{\Cca}$ where for any
syntax tree $\Tfr$ of $\Cca$, $\Angle{\Tfr, \Fbf_{|_1}}$ (resp.
$\Angle{\Tfr, \Fbf_{|_2}}$) is the number of inputs colors $1$ (resp.
$2$) of $\Tfr$. The sum $\Fbf_{|_1} + \Fbf_{|_2}$ is the series wherein
the coefficient of any syntax tree $\Tfr$ of $\Cca$ is its arity.
Moreover, the series $\Fbf_\Att + \Fbf_\Btt + \Fbf_{|_1} + \Fbf_{|_2}$
is the series wherein the coefficient of any syntax tree $\Tfr$ of
$\Cca$ is its total number of nodes.
\medbreak

The series of colors of $\Fbf_\Att$ is of the form
\begin{equation}
    \Colors(\Fbf_\Att) = (1, 21) + 2\, (1, 221) + 3\, (1, 2221)
    + (2, 2121) + (2, 1221) + (1, 1211) + \cdots,
\end{equation}
and the series of color types of $\Fbf$ is of the form
\begin{equation}
    \ColorTypes(\Fbf_\Att) =
    \VarX_1\VarY_1\VarY_2 + 2\VarX_1\VarY_1\VarY_2^2
    + \VarX_1\VarY_1^3\VarY_2
    + 3\VarX_1\VarY_1\VarY_2^3
    + 2\VarX_2\VarY_1^2\VarY_2^2 + \cdots.
\end{equation}
\medbreak

\subsubsection{Series of Motzkin paths}%
\label{subsubsec:example_series_motzkin_paths}
Let $\Cca$ be the $\Ccr$-bud operad $\Bud_\Ccr(\Motz)$, where
$\Ccr := \{-1, 1\}$. Let $\Fbf$ be the series of $\K \AAngle{\Cca}$
defined for any Motzkin path $\Bfr$, input color $a \in \Ccr$, and word
of input colors $u \in \Ccr^{|\Bfr|}$ by
\begin{equation}
    \Angle{(a, \Bfr, u), \Fbf} :=
    \frac{1}{2^{|\Bfr| + 1}}
    \Par{\prod_{i \in [|\Bfr|]} q_{\Height_\Bfr(i)}^{u_i}}^a,
\end{equation}
where $\Height_\Bfr(i)$ is the ordinate of the $i$-th point of $\Bfr$.
One has for instance, where the notation $\bar 1$ stands for~$-1$,
\begin{small}
\begin{subequations}
\begin{equation}
    \Angle{\Par{1,
    \begin{tikzpicture}[scale=.24,Centering]
        \draw[PathGrid] (0,0) grid (6,2);
        \node[PathNode](0)at(0,0){};
        \node[PathNode](1)at(1,0){};
        \node[PathNode](2)at(2,1){};
        \node[PathNode](3)at(3,1){};
        \node[PathNode](4)at(4,2){};
        \node[PathNode](5)at(5,1){};
        \node[PathNode](6)at(6,0){};
        \draw[PathStep](0)--(1);
        \draw[PathStep](1)--(2);
        \draw[PathStep](2)--(3);
        \draw[PathStep](3)--(4);
        \draw[PathStep](4)--(5);
        \draw[PathStep](5)--(6);
    \end{tikzpicture}\,,
    1 \bar 1 1 1 1 \bar 1 1}, \Fbf} =
    \frac{1}{2^8}
    \Par{q_0 q_0^{\bar 1} q_1 q_1 q_2 q_1^{ \bar1} q_0}^1
    = \frac{q_0 q_1 q_2}{2^8},
\end{equation}
\begin{equation}
    \Angle{\Par{\bar 1,
    \begin{tikzpicture}[scale=.24,Centering]
        \draw[PathGrid] (0,0) grid (9,2);
        \node[PathNode](0)at(0,0){};
        \node[PathNode](1)at(1,1){};
        \node[PathNode](2)at(2,0){};
        \node[PathNode](3)at(3,1){};
        \node[PathNode](4)at(4,2){};
        \node[PathNode](5)at(5,2){};
        \node[PathNode](6)at(6,1){};
        \node[PathNode](7)at(7,1){};
        \node[PathNode](8)at(8,0){};
        \draw[PathStep](0)--(1);
        \draw[PathStep](1)--(2);
        \draw[PathStep](2)--(3);
        \draw[PathStep](3)--(4);
        \draw[PathStep](4)--(5);
        \draw[PathStep](5)--(6);
        \draw[PathStep](6)--(7);
        \draw[PathStep](7)--(8);
    \end{tikzpicture}\,,
    \bar 1 1 \bar 1 1 1 \bar 1 1 \bar 1 1 }, \Fbf} =
    \frac{1}{2^{10}} \Par{q_0^{\bar 1} q_1 q_0^{ \bar 1} q_1
    q_2 q_2^{\bar 1} q_1 q_1^{\bar 1} q_0}^{\bar 1}
    = \frac{q_0}{2^{10} q_1^2}.
\end{equation}
\end{subequations}
\end{small}
\medbreak

Moreover, the coefficients of the pruned series of $\Fbf$ satisfy, by
definition of $\Prune$ and $\Fbf$,
\begin{equation}
    \Angle{\Bfr, \Prune(\Fbf)} =
    \sum_{\substack{
        a \in \Bra{\bar 1, 1} \\
        u \in \Bra{\bar 1, 1}^{|\Bfr|}
    }}
    \Angle{(a, \Bfr, u), \Fbf}
    =
    \frac{1}{2^{|\Bfr| + 1}}
    \sum_{u \in \Bra{\bar 1, 1}^{|\Bfr|}}
    \Par{
    \Par{\prod_{i \in [|\Bfr|]} q_{\Height_\Bfr(i)}^{u_i}}
    +
    \Par{\prod_{i \in [|\Bfr|]} q_{\Height_\Bfr(i)}^{-u_i}}
    }.
\end{equation}
These coefficients seem to factorize nicely. For instance,
\begin{multicols}{2}\small
\begin{subequations}
\begin{equation}
    \Angle{
    \begin{tikzpicture}[scale=.25,Centering]
        \node[PathNode](0)at(0,0){};
    \end{tikzpicture}\,,
    \Prune(\Fbf)} =
    \frac{1 + q_0^2}{2 q_0},
\end{equation}
\begin{equation}
    \Angle{
    \begin{tikzpicture}[scale=.25,Centering]
        \draw[PathGrid] (0,0) grid (1,0);
        \node[PathNode](0)at(0,0){};
        \node[PathNode](1)at(1,0){};
        \draw[PathStep](0)--(1);
    \end{tikzpicture}\,,
    \Prune(\Fbf)} =
    \frac{{\Par{1 + q_0^2}}^2}{4 q_0^2},
\end{equation}
\begin{equation}
    \Angle{
    \begin{tikzpicture}[scale=.25,Centering]
        \draw[PathGrid] (0,0) grid (2,0);
        \node[PathNode](0)at(0,0){};
        \node[PathNode](1)at(1,0){};
        \node[PathNode](2)at(2,0){};
        \draw[PathStep](0)--(1);
        \draw[PathStep](1)--(2);
    \end{tikzpicture}\,,
    \Prune(\Fbf)} =
    \frac{{\Par{1 + q_0^2}}^3}{8 q_0^3},
\end{equation}
\begin{equation}
    \Angle{
    \begin{tikzpicture}[scale=.25,Centering]
        \draw[PathGrid] (0,0) grid (2,1);
        \node[PathNode](0)at(0,0){};
        \node[PathNode](1)at(1,1){};
        \node[PathNode](2)at(2,0){};
        \draw[PathStep](0)--(1);
        \draw[PathStep](1)--(2);
    \end{tikzpicture}\,,
    \Prune(\Fbf)}
    = \frac{{(1 + q_0^2)}^2 \Par{1 + q_1^2}}{8 q_0^2 q_1},
\end{equation}
\begin{equation}
    \Angle{
    \begin{tikzpicture}[scale=.25,Centering]
        \draw[PathGrid] (0,0) grid (3,0);
        \node[PathNode](0)at(0,0){};
        \node[PathNode](1)at(1,0){};
        \node[PathNode](2)at(2,0){};
        \node[PathNode](3)at(3,0){};
        \draw[PathStep](0)--(1);
        \draw[PathStep](1)--(2);
        \draw[PathStep](2)--(3);
    \end{tikzpicture}\,,
    \Prune(\Fbf)} =
    \frac{{\Par{1 + q_0^2}}^4}{16 q_0^4},
\end{equation}
\begin{equation}
    \Angle{
    \begin{tikzpicture}[scale=.25,Centering]
        \draw[PathGrid] (0,0) grid (3,1);
        \node[PathNode](0)at(0,0){};
        \node[PathNode](1)at(1,0){};
        \node[PathNode](2)at(2,1){};
        \node[PathNode](3)at(3,0){};
        \draw[PathStep](0)--(1);
        \draw[PathStep](1)--(2);
        \draw[PathStep](2)--(3);
    \end{tikzpicture}\,,
    \Prune(\Fbf)} =
    \frac{{\Par{1 + q_0^2}}^3 \Par{1 + q_1^2}}{16 q_0^3 q_1},
\end{equation}
\begin{equation}
    \Angle{
    \begin{tikzpicture}[scale=.25,Centering]
        \draw[PathGrid] (0,0) grid (3,1);
        \node[PathNode](0)at(0,0){};
        \node[PathNode](1)at(1,1){};
        \node[PathNode](2)at(2,0){};
        \node[PathNode](3)at(3,0){};
        \draw[PathStep](0)--(1);
        \draw[PathStep](1)--(2);
        \draw[PathStep](2)--(3);
    \end{tikzpicture}\,,
    \Prune(\Fbf)} =
    \frac{{\Par{1 + q_0^2}}^3 \Par{1 + q_1^2}}{16 q_0^3 q_1},
\end{equation}
\begin{equation}
    \Angle{
    \begin{tikzpicture}[scale=.25,Centering]
        \draw[PathGrid] (0,0) grid (3,1);
        \node[PathNode](0)at(0,0){};
        \node[PathNode](1)at(1,1){};
        \node[PathNode](2)at(2,1){};
        \node[PathNode](3)at(3,0){};
        \draw[PathStep](0)--(1);
        \draw[PathStep](1)--(2);
        \draw[PathStep](2)--(3);
    \end{tikzpicture}\,,
    \Prune(\Fbf)} =
    \frac{{\Par{1 + q_0^2}}^2 {\Par{1 + q_1^2}}^2}
        {16 q_0^2 q_1^2},
\end{equation}
\begin{equation}
    \Angle{
    \begin{tikzpicture}[scale=.25,Centering]
        \draw[PathGrid] (0,0) grid (4,0);
        \node[PathNode](0)at(0,0){};
        \node[PathNode](1)at(1,0){};
        \node[PathNode](2)at(2,0){};
        \node[PathNode](3)at(3,0){};
        \node[PathNode](4)at(4,0){};
        \draw[PathStep](0)--(1);
        \draw[PathStep](1)--(2);
        \draw[PathStep](2)--(3);
        \draw[PathStep](3)--(4);
    \end{tikzpicture}\,,
    \Prune(\Fbf)} =
    \frac{{\Par{1 + q_0^2}}^5}
        {32 q_0^5},
\end{equation}

\begin{equation}
    \Angle{
    \begin{tikzpicture}[scale=.25,Centering]
        \draw[PathGrid] (0,0) grid (4,1);
        \node[PathNode](0)at(0,0){};
        \node[PathNode](1)at(1,0){};
        \node[PathNode](2)at(2,0){};
        \node[PathNode](3)at(3,1){};
        \node[PathNode](4)at(4,0){};
        \draw[PathStep](0)--(1);
        \draw[PathStep](1)--(2);
        \draw[PathStep](2)--(3);
        \draw[PathStep](3)--(4);
    \end{tikzpicture}\,,
    \Prune(\Fbf)} =
    \frac{{\Par{1 + q_0^2}}^4 {\Par{1 + q_1^2}}}
        {32 q_0^4 q_1},
\end{equation}
\begin{equation}
    \Angle{
    \begin{tikzpicture}[scale=.25,Centering]
        \draw[PathGrid] (0,0) grid (4,1);
        \node[PathNode](0)at(0,0){};
        \node[PathNode](1)at(1,0){};
        \node[PathNode](2)at(2,1){};
        \node[PathNode](3)at(3,0){};
        \node[PathNode](4)at(4,0){};
        \draw[PathStep](0)--(1);
        \draw[PathStep](1)--(2);
        \draw[PathStep](2)--(3);
        \draw[PathStep](3)--(4);
    \end{tikzpicture}\,,
    \Prune(\Fbf)} =
    \frac{{\Par{1 + q_0^2}}^4 {\Par{1 + q_1^2}}}
        {32 q_0^4 q_1},
\end{equation}
\begin{equation}
    \Angle{
    \begin{tikzpicture}[scale=.25,Centering]
        \draw[PathGrid] (0,0) grid (4,1);
        \node[PathNode](0)at(0,0){};
        \node[PathNode](1)at(1,0){};
        \node[PathNode](2)at(2,1){};
        \node[PathNode](3)at(3,1){};
        \node[PathNode](4)at(4,0){};
        \draw[PathStep](0)--(1);
        \draw[PathStep](1)--(2);
        \draw[PathStep](2)--(3);
        \draw[PathStep](3)--(4);
    \end{tikzpicture}\,,
    \Prune(\Fbf)} =
    \frac{{\Par{1 + q_0^2}}^3 {\Par{1 + q_1^2}}^2}
        {32 q_0^3 q_1^2},
\end{equation}
\begin{equation}
    \Angle{
    \begin{tikzpicture}[scale=.25,Centering]
        \draw[PathGrid] (0,0) grid (4,1);
        \node[PathNode](0)at(0,0){};
        \node[PathNode](1)at(1,1){};
        \node[PathNode](2)at(2,0){};
        \node[PathNode](3)at(3,0){};
        \node[PathNode](4)at(4,0){};
        \draw[PathStep](0)--(1);
        \draw[PathStep](1)--(2);
        \draw[PathStep](2)--(3);
        \draw[PathStep](3)--(4);
    \end{tikzpicture}\,,
    \Prune(\Fbf)} =
    \frac{{\Par{1 + q_0^2}}^4 {\Par{1 + q_1^2}}}
        {32 q_0^4 q_1},
\end{equation}
\begin{equation}
    \Angle{
    \begin{tikzpicture}[scale=.25,Centering]
        \draw[PathGrid] (0,0) grid (4,1);
        \node[PathNode](0)at(0,0){};
        \node[PathNode](1)at(1,1){};
        \node[PathNode](2)at(2,0){};
        \node[PathNode](3)at(3,1){};
        \node[PathNode](4)at(4,0){};
        \draw[PathStep](0)--(1);
        \draw[PathStep](1)--(2);
        \draw[PathStep](2)--(3);
        \draw[PathStep](3)--(4);
    \end{tikzpicture}\,,
    \Prune(\Fbf)} =
    \frac{{\Par{1 + q_0^2}}^3 {\Par{1 + q_1^2}}^2}
        {32 q_0^3 q_1^2},
\end{equation}
\begin{equation}
    \Angle{
    \begin{tikzpicture}[scale=.25,Centering]
        \draw[PathGrid] (0,0) grid (4,1);
        \node[PathNode](0)at(0,0){};
        \node[PathNode](1)at(1,1){};
        \node[PathNode](2)at(2,1){};
        \node[PathNode](3)at(3,0){};
        \node[PathNode](4)at(4,0){};
        \draw[PathStep](0)--(1);
        \draw[PathStep](1)--(2);
        \draw[PathStep](2)--(3);
        \draw[PathStep](3)--(4);
    \end{tikzpicture}\,,
    \Prune(\Fbf)} =
    \frac{{\Par{1 + q_0^2}}^3 {\Par{1 + q_1^2}}^2}
        {32 q_0^3 q_1^2},
\end{equation}
\begin{equation}
    \Angle{
    \begin{tikzpicture}[scale=.25,Centering]
        \draw[PathGrid] (0,0) grid (4,1);
        \node[PathNode](0)at(0,0){};
        \node[PathNode](1)at(1,1){};
        \node[PathNode](2)at(2,1){};
        \node[PathNode](3)at(3,1){};
        \node[PathNode](4)at(4,0){};
        \draw[PathStep](0)--(1);
        \draw[PathStep](1)--(2);
        \draw[PathStep](2)--(3);
        \draw[PathStep](3)--(4);
    \end{tikzpicture}\,,
    \Prune(\Fbf)} =
    \frac{{\Par{1 + q_0^2}}^2 {\Par{1 + q_1^2}}^3}
        {32 q_0^2 q_1^3},
\end{equation}
\begin{equation}
    \Angle{
    \begin{tikzpicture}[scale=.25,Centering]
        \draw[PathGrid] (0,0) grid (4,2);
        \node[PathNode](0)at(0,0){};
        \node[PathNode](1)at(1,1){};
        \node[PathNode](2)at(2,2){};
        \node[PathNode](3)at(3,1){};
        \node[PathNode](4)at(4,0){};
        \draw[PathStep](0)--(1);
        \draw[PathStep](1)--(2);
        \draw[PathStep](2)--(3);
        \draw[PathStep](3)--(4);
    \end{tikzpicture}\,,
    \Prune(\Fbf)} =
    \frac{{\Par{1 + q_0^2}}^2 {\Par{1 + q_1^2}}^2
            \Par{1 + q_2^2}}
        {32 q_0^2 q_1^2 q_2}.
\end{equation}
\end{subequations}
\end{multicols}
\noindent Observe that the specializations at $q_0 := 1$, $q_1 := 1$,
and $q_2 := 1$ of all these coefficients are equal to~$1$.
\medbreak

\subsection{Bud generating systems}%
\label{subsec:examples_bud_generating_systems}
We rely on the monochrome operads defined in
Section~\ref{subsec:examples_monochrome_operads} to construct several
bud generating systems. We review some properties of these, leaving
the proofs to the reader.
\medbreak

\subsubsection{Monochrome bud generating systems from $\Dias_\gamma$}%
\label{subsubsec:bud_generating_system_dias_gamma}
Let $\gamma$ be a nonnegative integer and consider the monochrome bud
generating system $\BDias{\gamma} := \Par{\Dias_\gamma, \Rfr_\gamma}$
where
\begin{equation}
    \Rfr_\gamma := \{0a, a0 : a \in [\gamma]\}.
\end{equation}
The derivation graph of $\BDias{1}$ is depicted by
Figure~\ref{fig:derivation_graph_bdias_1} and the one of $\BDias{2}$, by
Figure~\ref{fig:derivation_graph_bdias_2}.
\begin{figure}[ht]
    \centering
    \begin{tikzpicture}[xscale=.9,yscale=.8,Centering,font=\small]
        \node(0)at(0,0){\begin{math}0\end{math}};
        \node(01)at(-1,-1){\begin{math}01\end{math}};
        \node(10)at(1,-1){\begin{math}10\end{math}};
        \node(011)at(-2,-2){\begin{math}011\end{math}};
        \node(101)at(0,-2){\begin{math}101\end{math}};
        \node(110)at(2,-2){\begin{math}110\end{math}};
        \node(0111)at(-3,-3){\begin{math}0111\end{math}};
        \node(1011)at(-1,-3){\begin{math}1011\end{math}};
        \node(1101)at(1,-3){\begin{math}1101\end{math}};
        \node(1110)at(3,-3){\begin{math}1110\end{math}};
        \node(01111)at(-4,-4){\begin{math}01111\end{math}};
        \node(10111)at(-2,-4){\begin{math}10111\end{math}};
        \node(11011)at(0,-4){\begin{math}11011\end{math}};
        \node(11101)at(2,-4){\begin{math}11101\end{math}};
        \node(11110)at(4,-4){\begin{math}11110\end{math}};
        \draw[EdgeDeriv](0)
            edge[]node[anchor=south,font=\scriptsize,color=ColBlack]
            {}(01);
        \draw[EdgeDeriv](0)
            edge[]node[anchor=south,font=\scriptsize,color=ColBlack]
            {}(10);
        \draw[EdgeDeriv](01)
            edge[]node[anchor=south,font=\scriptsize,color=ColBlack]
            {\begin{math}3\end{math}}(011);
        \draw[EdgeDeriv](01)
            edge[]node[anchor=south,font=\scriptsize,color=ColBlack]
            {}(101);
        \draw[EdgeDeriv](10)
            edge[]node[anchor=south,font=\scriptsize,color=ColBlack]
            {}(101);
        \draw[EdgeDeriv](10)
            edge[]node[anchor=south,font=\scriptsize,color=ColBlack]
            {\begin{math}3\end{math}}(110);
        \draw[EdgeDeriv](011)
            edge[]node[anchor=south,font=\scriptsize,color=ColBlack]
            {\begin{math}5\end{math}}(0111);
        \draw[EdgeDeriv](011)
            edge[]node[anchor=south,font=\scriptsize,color=ColBlack]
            {}(1011);
        \draw[EdgeDeriv](101)
            edge[]node[anchor=south,font=\scriptsize,color=ColBlack]
            {\begin{math}3\end{math}}(1011);
        \draw[EdgeDeriv](101)
            edge[]node[anchor=south,font=\scriptsize,color=ColBlack]
            {\begin{math}3\end{math}}(1101);
        \draw[EdgeDeriv](110)
            edge[]node[anchor=south,font=\scriptsize,color=ColBlack]
            {}(1101);
        \draw[EdgeDeriv](110)
            edge[]node[anchor=south,font=\scriptsize,color=ColBlack]
            {\begin{math}5\end{math}}(1110);
        \draw[EdgeDeriv](0111)
            edge[]node[anchor=south,font=\scriptsize,color=ColBlack]
            {\begin{math}7\end{math}}(01111);
        \draw[EdgeDeriv](0111)
            edge[]node[anchor=south,font=\scriptsize,color=ColBlack]
            {}(10111);
        \draw[EdgeDeriv](1011)
            edge[]node[anchor=south,font=\scriptsize,color=ColBlack]
            {\begin{math}5\end{math}}(10111);
        \draw[EdgeDeriv](1011)
            edge[]node[anchor=south,font=\scriptsize,color=ColBlack]
            {\begin{math}3\end{math}}(11011);
        \draw[EdgeDeriv](1101)
            edge[]node[anchor=south,font=\scriptsize,color=ColBlack]
            {\begin{math}3\end{math}}(11011);
        \draw[EdgeDeriv](1101)
            edge[]node[anchor=south,font=\scriptsize,color=ColBlack]
            {\begin{math}5\end{math}}(11101);
        \draw[EdgeDeriv](1110)
            edge[]node[anchor=south,font=\scriptsize,color=ColBlack]
            {}(11101);
        \draw[EdgeDeriv](1110)
            edge[]node[anchor=south,font=\scriptsize,color=ColBlack]
            {\begin{math}7\end{math}}(11110);
        \draw[EdgeDeriv,dotted,shorten >=4mm](01111)--(-4,-5.25);
        \draw[EdgeDeriv,dotted,shorten >=4mm](10111)--(-2,-5.25);
        \draw[EdgeDeriv,dotted,shorten >=4mm](11011)--(0,-5.25);
        \draw[EdgeDeriv,dotted,shorten >=4mm](11101)--(2,-5.25);
        \draw[EdgeDeriv,dotted,shorten >=4mm](11110)--(4,-5.25);
    \end{tikzpicture}
    \vspace{-1.5em}
    \caption{\footnotesize The derivation graph of $\BDias{1}$.}
    \label{fig:derivation_graph_bdias_1}
\end{figure}
\begin{figure}[ht]
    \centering
    \begin{tikzpicture}[xscale=.42,yscale=1.5,Centering,font=\small]
        \node(0)at(0,0){\begin{math}0\end{math}};
        \node(02)at(-3,-1){\begin{math}02\end{math}};
        \node(01)at(-9,-1){\begin{math}01\end{math}};
        \node(10)at(9,-1){\begin{math}10\end{math}};
        \node(20)at(3,-1){\begin{math}20\end{math}};
        \node(022)at(-3,-2){\begin{math}022\end{math}};
        \node(021)at(-15,-2){\begin{math}021\end{math}};
        \node(012)at(-9,-2){\begin{math}012\end{math}};
        \node(011)at(-12,-2){\begin{math}011\end{math}};
        \node(202)at(0,-1.85){\begin{math}202\end{math}};
        \node(201)at(-6,-2){\begin{math}201\end{math}};
        \node(102)at(6,-2){\begin{math}102\end{math}};
        \node(101)at(0,-2.15){\begin{math}101\end{math}};
        \node(220)at(3,-2){\begin{math}220\end{math}};
        \node(120)at(15,-2){\begin{math}120\end{math}};
        \node(210)at(9,-2){\begin{math}210\end{math}};
        \node(110)at(12,-2){\begin{math}110\end{math}};
        \draw[EdgeDeriv](0)
            edge[]node[anchor=south,font=\scriptsize,color=ColBlack]
            {}(01);
        \draw[EdgeDeriv](0)
            edge[]node[anchor=south,font=\scriptsize,color=ColBlack]
            {}(02);
        \draw[EdgeDeriv](0)
            edge[]node[anchor=south,font=\scriptsize,color=ColBlack]
            {}(20);
        \draw[EdgeDeriv](0)
            edge[]node[anchor=south,font=\scriptsize,color=ColBlack]
            {}(10);
        \draw[EdgeDeriv](01)
            edge[]node[anchor=south,font=\scriptsize,color=ColBlack]
            {\begin{math}3\end{math}}(011);
        \draw[EdgeDeriv](01)
            edge[]node[anchor=south,font=\scriptsize,color=ColBlack]
            {\begin{math}2\end{math}}(021);
        \draw[EdgeDeriv](01)
            edge[]node[anchor=west,font=\scriptsize,color=ColBlack]
            {}(012);
        \draw[EdgeDeriv](01)
            edge[]node[anchor=south,font=\scriptsize,color=ColBlack]
            {}(101);
        \draw[EdgeDeriv](01)
            edge[]node[anchor=south,font=\scriptsize,color=ColBlack]
            {}(201);
        \draw[EdgeDeriv](02)
            edge[]node[anchor=south,font=\scriptsize,color=ColBlack]
            {}(012);
        \draw[EdgeDeriv](02)
            edge[]node[anchor=east,font=\scriptsize,color=ColBlack]
            {\begin{math}5\end{math}}(022);
        \draw[EdgeDeriv](02)
            edge[]node[anchor=south,font=\scriptsize,color=ColBlack]
            {}(102);
        \draw[EdgeDeriv](02)
            edge[]node[anchor=east,font=\scriptsize,color=ColBlack]
            {}(202);
        \draw[EdgeDeriv](10)
            edge[]node[anchor=south,font=\scriptsize,color=ColBlack]
            {}(101);
        \draw[EdgeDeriv](10)
            edge[]node[anchor=south,font=\scriptsize,color=ColBlack]
            {}(102);
        \draw[EdgeDeriv](10)
            edge[]node[anchor=south,font=\scriptsize,color=ColBlack]
            {\begin{math}2\end{math}}(120);
        \draw[EdgeDeriv](10)
            edge[]node[anchor=south,font=\scriptsize,color=ColBlack]
            {\begin{math}3\end{math}}(110);
        \draw[EdgeDeriv](10)
            edge[]node[anchor=west,font=\scriptsize,color=ColBlack]
            {}(210);
        \draw[EdgeDeriv](20)
            edge[]node[anchor=south,font=\scriptsize,color=ColBlack]
            {}(201);
        \draw[EdgeDeriv](20)
            edge[]node[anchor=west,font=\scriptsize,color=ColBlack]
            {}(202);
        \draw[EdgeDeriv](20)
            edge[]node[anchor=south,font=\scriptsize,color=ColBlack]
            {}(210);
        \draw[EdgeDeriv](20)
            edge[]node[anchor=west,font=\scriptsize,color=ColBlack]
            {\begin{math}5\end{math}}(220);
        \draw[EdgeDeriv,dotted,shorten >=4mm](021)--(-15,-2.75);
        \draw[EdgeDeriv,dotted,shorten >=4mm](011)--(-12,-2.75);
        \draw[EdgeDeriv,dotted,shorten >=4mm](012)--(-9,-2.75);
        \draw[EdgeDeriv,dotted,shorten >=4mm](201)--(-6,-2.75);
        \draw[EdgeDeriv,dotted,shorten >=4mm](022)--(-3,-2.75);
        \draw[EdgeDeriv,dotted,shorten >=4mm](101)--(0,-2.75);
        \draw[EdgeDeriv,dotted,shorten >=4mm](220)--(3,-2.75);
        \draw[EdgeDeriv,dotted,shorten >=4mm](102)--(6,-2.75);
        \draw[EdgeDeriv,dotted,shorten >=4mm](210)--(9,-2.75);
        \draw[EdgeDeriv,dotted,shorten >=4mm](110)--(12,-2.75);
        \draw[EdgeDeriv,dotted,shorten >=4mm](120)--(15,-2.75);
    \end{tikzpicture}
    \vspace{-1.5em}
    \caption{\footnotesize The derivation graph of $\BDias{2}$.}
    \label{fig:derivation_graph_bdias_2}
\end{figure}
\medbreak

\begin{Proposition} \label{prop:properties_BDias}
    For any $\gamma \geq 0$, the monochrome bud generating system
    $\BDias{\gamma}$ satisfies the following properties.
    \begin{enumerate}[label = ({\it {\roman*})}]
        \item \label{item:properties_BDias_faithful}
        It is faithful.
        \item \label{item:properties_BDias_language}
        The set $\Lang\Par{\BDias{\gamma}}$ is equal to the underlying
        monochrome graded collection of $\Dias_\gamma$.
        \item \label{item:properties_BDias_finitely_factorizing}
        The set of rules $\Rfr_\gamma(1)$ factorizes
        finitely~$\Dias_\gamma$.
    \end{enumerate}
\end{Proposition}
\medbreak

Property~\ref{item:properties_BDias_faithful} of
Proposition~\ref{prop:properties_BDias} is a consequence of the fact
that $\BDias{\gamma}$ is monochrome and
Property~\ref{item:properties_BDias_language} is implied by the fact
that $\Rfr_\gamma$ is a generating set of $\Dias_\gamma$~\cite{Gir15}.
Moreover, observe that since the word $\gamma 0 \gamma$ of
$\Dias_\gamma(3)$ admits exactly the two $\Rfr_\gamma$-treelike
expressions
\begin{math}
    0 \gamma \circ_1 \gamma 0
\end{math}
and
\begin{math}
    \gamma 0 \circ_2 0 \gamma,
\end{math}
by Theorem~\ref{thm:syntactic_series},
$\Angle{\gamma 0 \gamma, \Synt(\BDias{\gamma})} = 2$. Hence,
$\BDias{\gamma}$ is not unambiguous.
\medbreak

\subsubsection{A bud generating system for Motzkin paths}%
\label{subsubsec:example_bmotz}
Consider the bud generating system
$\BMotz := (\Motz, \{1, 2\}, \Rfr, \{1\}, \{1, 2\})$ where
\begin{equation}
    \Rfr := \Bra{\Par{1, \MotzHoriz, 22},
    \Par{1, \MotzPeak, 111}}.
\end{equation}
Figure~\ref{fig:derivations_bmotz} shows a sequence of derivations in
$\BMotz$ and Figure~\ref{fig:derivation_graph_bmotz} shows the
derivation graph of~$\BMotz$.
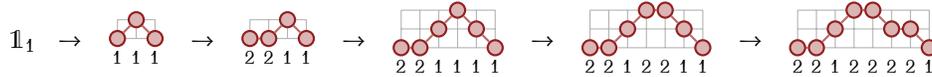
\begin{figure}[ht]
    \centering
    \begin{equation*}
        \Unit_1
        \enspace \Deriv \enspace
        \begin{tikzpicture}[scale=.25,Centering]
            \draw[PathGrid](0,0) grid (2,1);
            \node[PathNode](1)at(0,0){};
            \node[PathNode](2)at(1,1){};
            \node[PathNode](3)at(2,0){};
            \draw[PathStep](1)--(2);
            \draw[PathStep](2)--(3);
            \node[font=\scriptsize]at(0,-1){\begin{math}1\end{math}};
            \node[font=\scriptsize]at(1,-1){\begin{math}1\end{math}};
            \node[font=\scriptsize]at(2,-1){\begin{math}1\end{math}};
        \end{tikzpicture}
        \enspace \Deriv \enspace
        \begin{tikzpicture}[scale=.25,Centering]
            \draw[PathGrid](0,0) grid (3,1);
            \node[PathNode](1)at(0,0){};
            \node[PathNode](2)at(1,0){};
            \node[PathNode](3)at(2,1){};
            \node[PathNode](4)at(3,0){};
            \draw[PathStep](1)--(2);
            \draw[PathStep](2)--(3);
            \draw[PathStep](3)--(4);
            \node[font=\scriptsize]at(0,-1){\begin{math}2\end{math}};
            \node[font=\scriptsize]at(1,-1){\begin{math}2\end{math}};
            \node[font=\scriptsize]at(2,-1){\begin{math}1\end{math}};
            \node[font=\scriptsize]at(3,-1){\begin{math}1\end{math}};
        \end{tikzpicture}
        \enspace \Deriv \enspace
        \begin{tikzpicture}[scale=.25,Centering]
            \draw[PathGrid](0,0) grid (5,2);
            \node[PathNode](1)at(0,0){};
            \node[PathNode](2)at(1,0){};
            \node[PathNode](3)at(2,1){};
            \node[PathNode](4)at(3,2){};
            \node[PathNode](5)at(4,1){};
            \node[PathNode](6)at(5,0){};
            \draw[PathStep](1)--(2);
            \draw[PathStep](2)--(3);
            \draw[PathStep](3)--(4);
            \draw[PathStep](4)--(5);
            \draw[PathStep](5)--(6);
            \node[font=\scriptsize]at(0,-1){\begin{math}2\end{math}};
            \node[font=\scriptsize]at(1,-1){\begin{math}2\end{math}};
            \node[font=\scriptsize]at(2,-1){\begin{math}1\end{math}};
            \node[font=\scriptsize]at(3,-1){\begin{math}1\end{math}};
            \node[font=\scriptsize]at(4,-1){\begin{math}1\end{math}};
            \node[font=\scriptsize]at(5,-1){\begin{math}1\end{math}};
        \end{tikzpicture}
        \enspace \Deriv \enspace
        \begin{tikzpicture}[scale=.25,Centering]
            \draw[PathGrid](0,0) grid (6,2);
            \node[PathNode](1)at(0,0){};
            \node[PathNode](2)at(1,0){};
            \node[PathNode](3)at(2,1){};
            \node[PathNode](4)at(3,2){};
            \node[PathNode](5)at(4,2){};
            \node[PathNode](6)at(5,1){};
            \node[PathNode](7)at(6,0){};
            \draw[PathStep](1)--(2);
            \draw[PathStep](2)--(3);
            \draw[PathStep](3)--(4);
            \draw[PathStep](4)--(5);
            \draw[PathStep](5)--(6);
            \draw[PathStep](6)--(7);
            \node[font=\scriptsize]at(0,-1){\begin{math}2\end{math}};
            \node[font=\scriptsize]at(1,-1){\begin{math}2\end{math}};
            \node[font=\scriptsize]at(2,-1){\begin{math}1\end{math}};
            \node[font=\scriptsize]at(3,-1){\begin{math}2\end{math}};
            \node[font=\scriptsize]at(4,-1){\begin{math}2\end{math}};
            \node[font=\scriptsize]at(5,-1){\begin{math}1\end{math}};
            \node[font=\scriptsize]at(6,-1){\begin{math}1\end{math}};
        \end{tikzpicture}
        \enspace \Deriv \enspace
        \begin{tikzpicture}[scale=.25,Centering]
            \draw[PathGrid](0,0) grid (7,2);
            \node[PathNode](1)at(0,0){};
            \node[PathNode](2)at(1,0){};
            \node[PathNode](3)at(2,1){};
            \node[PathNode](4)at(3,2){};
            \node[PathNode](5)at(4,2){};
            \node[PathNode](6)at(5,1){};
            \node[PathNode](7)at(6,1){};
            \node[PathNode](8)at(7,0){};
            \draw[PathStep](1)--(2);
            \draw[PathStep](2)--(3);
            \draw[PathStep](3)--(4);
            \draw[PathStep](4)--(5);
            \draw[PathStep](5)--(6);
            \draw[PathStep](6)--(7);
            \draw[PathStep](7)--(8);
            \node[font=\scriptsize]at(0,-1){\begin{math}2\end{math}};
            \node[font=\scriptsize]at(1,-1){\begin{math}2\end{math}};
            \node[font=\scriptsize]at(2,-1){\begin{math}1\end{math}};
            \node[font=\scriptsize]at(3,-1){\begin{math}2\end{math}};
            \node[font=\scriptsize]at(4,-1){\begin{math}2\end{math}};
            \node[font=\scriptsize]at(5,-1){\begin{math}2\end{math}};
            \node[font=\scriptsize]at(6,-1){\begin{math}2\end{math}};
            \node[font=\scriptsize]at(7,-1){\begin{math}1\end{math}};
        \end{tikzpicture}
    \end{equation*}
    \vspace{-1.5em}
    \caption{\footnotesize A sequence of derivations in $\BMotz$. The
    input colors of the elements of $\Bud_{\{1, 2\}}(\Motz)$ are
    depicted below the paths. The output color of all these elements
    is~$1$.}
    \label{fig:derivations_bmotz}
\end{figure}
\begin{figure}[ht]
    \centering
    \begin{equation*}
        \begin{tikzpicture}[yscale=1.5,xscale=.85,,Centering,
        font=\scriptsize]
            \node(Unit)at(0,0){\begin{math}\Unit_1\end{math}};
            \node(00)at(-2,-1){\begin{math}
            \begin{tikzpicture}[scale=.25]
                \draw[PathGrid](0,0) grid (1,0);
                \node[PathNode](1)at(0,0){};
                \node[PathNode](2)at(1,0){};
                \draw[PathStep](1)--(2);
                \node[font=\scriptsize]at(0,-1)
                    {\begin{math}2\end{math}};
                \node[font=\scriptsize]at(1,-1)
                    {\begin{math}2\end{math}};
            \end{tikzpicture}\end{math}};
            \node(010)at(2,-1){\begin{math}
            \begin{tikzpicture}[scale=.25]
                \draw[PathGrid](0,0) grid (2,1);
                \node[PathNode](1)at(0,0){};
                \node[PathNode](2)at(1,1){};
                \node[PathNode](3)at(2,0){};
                \draw[PathStep](1)--(2);
                \draw[PathStep](2)--(3);
                \node[font=\scriptsize]at(0,-1)
                    {\begin{math}1\end{math}};
                \node[font=\scriptsize]at(1,-1)
                    {\begin{math}1\end{math}};
                \node[font=\scriptsize]at(2,-1)
                    {\begin{math}1\end{math}};
            \end{tikzpicture}\end{math}};
            \node(0010)at(-4,-2){\begin{math}
            \begin{tikzpicture}[scale=.25]
                \draw[PathGrid](0,0) grid (3,1);
                \node[PathNode](1)at(0,0){};
                \node[PathNode](2)at(1,0){};
                \node[PathNode](3)at(2,1){};
                \node[PathNode](4)at(3,0){};
                \draw[PathStep](1)--(2);
                \draw[PathStep](2)--(3);
                \draw[PathStep](3)--(4);
                \node[font=\scriptsize]at(0,-1)
                    {\begin{math}2\end{math}};
                \node[font=\scriptsize]at(1,-1)
                    {\begin{math}2\end{math}};
                \node[font=\scriptsize]at(2,-1)
                    {\begin{math}1\end{math}};
                \node[font=\scriptsize]at(3,-1)
                    {\begin{math}1\end{math}};
            \end{tikzpicture}\end{math}};
            \node(0100)at(-1,-2){\begin{math}
            \begin{tikzpicture}[scale=.25]
                \draw[PathGrid](0,0) grid (3,1);
                \node[PathNode](1)at(0,0){};
                \node[PathNode](2)at(1,1){};
                \node[PathNode](3)at(2,0){};
                \node[PathNode](4)at(3,0){};
                \draw[PathStep](1)--(2);
                \draw[PathStep](2)--(3);
                \draw[PathStep](3)--(4);
                \node[font=\scriptsize]at(0,-1)
                    {\begin{math}1\end{math}};
                \node[font=\scriptsize]at(1,-1)
                    {\begin{math}1\end{math}};
                \node[font=\scriptsize]at(2,-1)
                    {\begin{math}2\end{math}};
                \node[font=\scriptsize]at(3,-1)
                    {\begin{math}2\end{math}};
            \end{tikzpicture}\end{math}};
            \node(0110)at(2,-2){\begin{math}
            \begin{tikzpicture}[scale=.25]
                \draw[PathGrid](0,0) grid (3,1);
                \node[PathNode](1)at(0,0){};
                \node[PathNode](2)at(1,1){};
                \node[PathNode](3)at(2,1){};
                \node[PathNode](4)at(3,0){};
                \draw[PathStep](1)--(2);
                \draw[PathStep](2)--(3);
                \draw[PathStep](3)--(4);
                \node[font=\scriptsize]at(0,-1)
                    {\begin{math}1\end{math}};
                \node[font=\scriptsize]at(1,-1)
                    {\begin{math}2\end{math}};
                \node[font=\scriptsize]at(2,-1)
                    {\begin{math}2\end{math}};
                \node[font=\scriptsize]at(3,-1)
                    {\begin{math}1\end{math}};
            \end{tikzpicture}\end{math}};
            \node(01010)at(5,-2){\begin{math}
            \begin{tikzpicture}[scale=.25]
                \draw[PathGrid](0,0) grid (3,1);
                \node[PathNode](1)at(0,0){};
                \node[PathNode](2)at(1,1){};
                \node[PathNode](3)at(2,0){};
                \node[PathNode](4)at(3,1){};
                \node[PathNode](5)at(4,0){};
                \draw[PathStep](1)--(2);
                \draw[PathStep](2)--(3);
                \draw[PathStep](3)--(4);
                \draw[PathStep](4)--(5);
                \node[font=\scriptsize]at(0,-1)
                    {\begin{math}1\end{math}};
                \node[font=\scriptsize]at(1,-1)
                    {\begin{math}1\end{math}};
                \node[font=\scriptsize]at(2,-1)
                    {\begin{math}1\end{math}};
                \node[font=\scriptsize]at(3,-1)
                    {\begin{math}1\end{math}};
                \node[font=\scriptsize]at(4,-1)
                    {\begin{math}1\end{math}};
            \end{tikzpicture}\end{math}};
            \node(01210)at(8,-2){\begin{math}
            \begin{tikzpicture}[scale=.25]
                \draw[PathGrid](0,0) grid (4,2);
                \node[PathNode](1)at(0,0){};
                \node[PathNode](2)at(1,1){};
                \node[PathNode](3)at(2,2){};
                \node[PathNode](4)at(3,1){};
                \node[PathNode](5)at(4,0){};
                \draw[PathStep](1)--(2);
                \draw[PathStep](2)--(3);
                \draw[PathStep](3)--(4);
                \draw[PathStep](4)--(5);
                \node[font=\scriptsize]at(0,-1)
                    {\begin{math}1\end{math}};
                \node[font=\scriptsize]at(1,-1)
                    {\begin{math}1\end{math}};
                \node[font=\scriptsize]at(2,-1)
                    {\begin{math}1\end{math}};
                \node[font=\scriptsize]at(3,-1)
                    {\begin{math}1\end{math}};
                \node[font=\scriptsize]at(4,-1)
                    {\begin{math}1\end{math}};
            \end{tikzpicture}\end{math}};
            \node(00100)at(-2.5,-3){\begin{math}
            \begin{tikzpicture}[scale=.25]
                \draw[PathGrid](0,0) grid (4,1);
                \node[PathNode](1)at(0,0){};
                \node[PathNode](2)at(1,0){};
                \node[PathNode](3)at(2,1){};
                \node[PathNode](4)at(3,0){};
                \node[PathNode](5)at(4,0){};
                \draw[PathStep](1)--(2);
                \draw[PathStep](2)--(3);
                \draw[PathStep](3)--(4);
                \draw[PathStep](4)--(5);
                \node[font=\scriptsize]at(0,-1)
                    {\begin{math}2\end{math}};
                \node[font=\scriptsize]at(1,-1)
                    {\begin{math}2\end{math}};
                \node[font=\scriptsize]at(2,-1)
                    {\begin{math}1\end{math}};
                \node[font=\scriptsize]at(3,-1)
                    {\begin{math}2\end{math}};
                \node[font=\scriptsize]at(4,-1)
                    {\begin{math}2\end{math}};
            \end{tikzpicture}\end{math}};
            \draw[EdgeDeriv](Unit)--(00);
            \draw[EdgeDeriv](Unit)--(010);
            \draw[EdgeDeriv](010)--(0010);
            \draw[EdgeDeriv](010)--(0110);
            \draw[EdgeDeriv](010)--(0100);
            \draw[EdgeDeriv](010)
                edge[]
                node[anchor=north,font=\scriptsize,color=ColBlack]
                {\begin{math}2\end{math}}(01010);
            \draw[EdgeDeriv](010)--(01210);
            \draw[EdgeDeriv](0010)--(00100);
            \draw[EdgeDeriv](0100)--(00100);
            \draw[EdgeDeriv,dotted,shorten >=6mm](0010)--(-4,-3);
            \draw[EdgeDeriv,dotted,shorten >=6mm](0110)--(2,-3);
            \draw[EdgeDeriv,dotted,shorten >=6mm](0100)--(-1,-3);
            \draw[EdgeDeriv,dotted,shorten >=6mm](01010)--(5,-3);
            \draw[EdgeDeriv,dotted,shorten >=6mm](01210)--(8,-3);
            \draw[EdgeDeriv,dotted,shorten >=6mm](00100)--(-2.5,-4);
        \end{tikzpicture}
    \end{equation*}
    \vspace{-1.5em}
    \caption{\footnotesize The derivation graph of $\BMotz$. The input
    colors of the elements of $\Bud_{\{1, 2\}}(\Motz)$ are depicted
    below the paths. The output color of all these elements is~$1$.}
    \label{fig:derivation_graph_bmotz}
\end{figure}
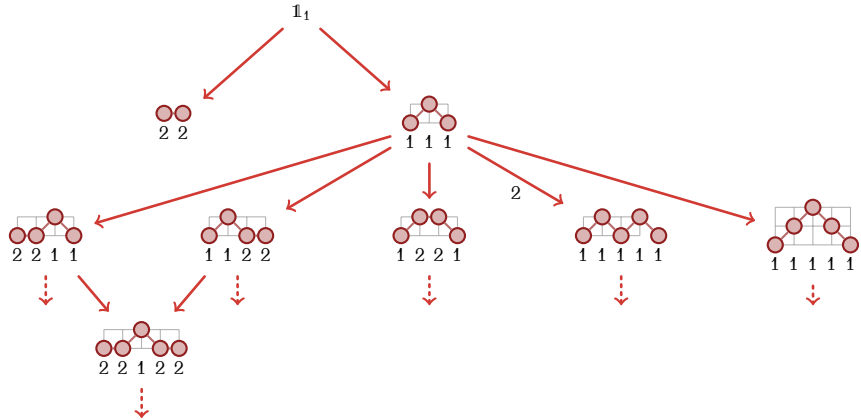
\medbreak

Let $L_{\BMotz}$ be the set of Motzkin paths with no consecutive
horizontal steps.
\medbreak

\begin{Proposition} \label{prop:properties_BMotz}
    The bud generating system $\BMotz$ satisfies the following
    properties.
    \begin{enumerate}[label = ({\it {\roman*})}]
        \item \label{item:properties_BMotz_faithful}
        It is faithful.
        \item \label{item:properties_BMotz_language}
        The restriction of the pruning map $\Prune$ on the domain
        $\Lang\Par{\BMotz}$
        is a bijection between $\Lang\Par{\BMotz}$ and $L_{\BMotz}$.
        \item \label{item:properties_BASchr_finitely_factorizing}
        The set of rules $\Rfr(1)$ finitely factorizes
        $\Bud_{\{1, 2\}}(\Motz)$.
    \end{enumerate}
\end{Proposition}
\medbreak

Properties~\ref{item:properties_BMotz_faithful}
and~\ref{item:properties_BMotz_language} of
Proposition~\ref{prop:properties_BMotz} together say that the sequence
enumerating the elements of $\Lang(\BMotz)$ with respect to their arity
is the one enumerating the Motzkin paths with no consecutive horizontal
steps. This sequence is Sequence~\OEIS{A104545} of~\cite{Slo}, starting
by
\begin{equation}
    1, 1, 1, 3, 5, 11, 25, 55, 129, 303, 721, 1743, 4241, 10415,
    25761, 64095.
\end{equation}
Moreover, since the Motzkin path
\begin{math}
    \begin{tikzpicture}[scale=.2,Centering]
        \draw[PathGrid](0,0) grid (3,1);
        \node[PathNode](1)at(0,0){};
        \node[PathNode](2)at(1,1){};
        \node[PathNode](3)at(2,0){};
        \node[PathNode](4)at(3,1){};
        \node[PathNode](5)at(4,0){};
        \draw[PathStep](1)--(2);
        \draw[PathStep](2)--(3);
        \draw[PathStep](3)--(4);
        \draw[PathStep](4)--(5);
    \end{tikzpicture}
\end{math}
of $\Motz(5)$ admits exactly the two $\Rfr$-treelike expressions
\begin{math}
    \MotzPeak \circ_1 \MotzPeak
\end{math}
and
\begin{math}
    \MotzPeak \circ_3 \MotzPeak,
\end{math}
by Theorem~\ref{thm:syntactic_series},
\begin{math}
    \Angle{\Par{1,
    \begin{tikzpicture}[scale=.23,Centering]
        \draw[PathGrid](0,0) grid (3,1);
        \node[PathNode](1)at(0,0){};
        \node[PathNode](2)at(1,1){};
        \node[PathNode](3)at(2,0){};
        \node[PathNode](4)at(3,1){};
        \node[PathNode](5)at(4,0){};
        \draw[PathStep](1)--(2);
        \draw[PathStep](2)--(3);
        \draw[PathStep](3)--(4);
        \draw[PathStep](4)--(5);
    \end{tikzpicture},
    11111},
    \Synt(\BMotz)}
    = 2.
\end{math}
Hence $\BMotz$ is not unambiguous.
\medbreak

\subsubsection{A bud generating system for unary-binary trees}%
\label{subsubsec:example_bubtree}
Let $C$ be the monochrome graded collection defined by
$C := C(1) \sqcup C(2)$ where $C(1) := \{\Att, \Btt\}$ and
$C(2) := \{\Ctt\}$. Let
$\BUBTree := (\Free(C), \{1, 2\}, \Rfr, \{1\}, \{2\})$ be the bud
generating system where
\begin{equation}
    \Rfr := \Bra{
        \Par{1, \FreeCorollaOne{\Att}, 2},
        \Par{1, \FreeCorollaOne{\Btt}, 2},
        \Par{2, \FreeCorollaTwo{\Ctt}, 11}
    }.
\end{equation}
Figure~\ref{fig:derivations_unary_binary_trees} shows a sequence of
derivations in $\BUBTree$.
\begin{figure}[ht]
    \centering
    \begin{equation*}
        \Unit_1
        \enspace \Deriv \enspace
        \begin{tikzpicture}[xscale=.25,yscale=.45,Centering]
            \node(1)at(0.00,-1.00){};
            \node[NodeST](0)at(0.00,0.00){\begin{math}\Btt\end{math}};
            \draw[Edge](1)--(0);
            \node(r)at(0.00,1){};
            \draw[Edge](r)--(0);
            \node[LeafLabel,below of=1]{\begin{math}2\end{math}};
        \end{tikzpicture}
        \enspace \Deriv \enspace
        \begin{tikzpicture}[xscale=.22,yscale=.35,Centering]
            \node(1)at(0.00,-2.67){};
            \node(3)at(2.00,-2.67){};
            \node[NodeST](0)at(1.00,0.00)
                {\begin{math}\Btt\end{math}};
            \node[NodeST](2)at(1.00,-1.33)
                {\begin{math}\Ctt\end{math}};
            \draw[Edge](1)--(2);
            \draw[Edge](2)--(0);
            \draw[Edge](3)--(2);
            \node(r)at(1.00,1.25){};
            \draw[Edge](r)--(0);
            \node[LeafLabel,below of=1]
                {\begin{math}1\end{math}};
            \node[LeafLabel,below of=3]
                {\begin{math}1\end{math}};
        \end{tikzpicture}
        \enspace \Deriv \enspace
        \begin{tikzpicture}[xscale=.26,yscale=.35,Centering]
            \node(1)at(0.00,-2.50){};
            \node(4)at(2.00,-3.75){};
            \node[NodeST](0)at(1.00,0.00)
                {\begin{math}\Btt\end{math}};
            \node[NodeST](2)at(1.00,-1.25)
                {\begin{math}\Ctt\end{math}};
            \node[NodeST](3)at(2.00,-2.50)
                {\begin{math}\Att\end{math}};
            \draw[Edge](1)--(2);
            \draw[Edge](2)--(0);
            \draw[Edge](3)--(2);
            \draw[Edge](4)--(3);
            \node(r)at(1.00,01.25){};
            \draw[Edge](r)--(0);
            \node[LeafLabel,below of=1]
                {\begin{math}1\end{math}};
            \node[LeafLabel,below of=4]
                {\begin{math}2\end{math}};
        \end{tikzpicture}
        \enspace \Deriv \enspace
        \begin{tikzpicture}[xscale=.24,yscale=.3,Centering]
            \node(1)at(0.00,-2.80){};
            \node(4)at(2.00,-5.60){};
            \node(6)at(4.00,-5.60){};
            \node[NodeST](0)at(2.00,0.00)
                {\begin{math}\Btt\end{math}};
            \node[NodeST](2)at(1.00,-1.40)
                {\begin{math}\Ctt\end{math}};
            \node[NodeST](3)at(3.00,-2.80)
                {\begin{math}\Att\end{math}};
            \node[NodeST](5)at(3.00,-4.20)
                {\begin{math}\Ctt\end{math}};
            \draw[Edge](1)--(2);
            \draw[Edge](2)--(0);
            \draw[Edge](3)--(2);
            \draw[Edge](4)--(5);
            \draw[Edge](5)--(3);
            \draw[Edge](6)--(5);
            \node(r)at(2.00,1.25){};
            \draw[Edge](r)--(0);
            \node[LeafLabel,below of=1]
                {\begin{math}1\end{math}};
            \node[LeafLabel,below of=4]
                {\begin{math}1\end{math}};
            \node[LeafLabel,below of=6]
                {\begin{math}1\end{math}};
        \end{tikzpicture}
        \enspace \Deriv \enspace
        \begin{tikzpicture}[xscale=.28,yscale=.3,Centering]
            \node(2)at(0.00,-4.80){};
            \node(5)at(2.00,-6.40){};
            \node(7)at(4.00,-6.40){};
            \node[NodeST](0)at(2.00,0.00)
                {\begin{math}\Btt\end{math}};
            \node[NodeST](1)at(0.00,-3.20)
                {\begin{math}\Btt\end{math}};
            \node[NodeST](3)at(1.00,-1.60)
                {\begin{math}\Ctt\end{math}};
            \node[NodeST](4)at(3.00,-3.20)
                {\begin{math}\Att\end{math}};
            \node[NodeST](6)at(3.00,-4.80)
                {\begin{math}\Ctt\end{math}};
            \draw[Edge](1)--(3);
            \draw[Edge](2)--(1);
            \draw[Edge](3)--(0);
            \draw[Edge](4)--(3);
            \draw[Edge](5)--(6);
            \draw[Edge](6)--(4);
            \draw[Edge](7)--(6);
            \node(r)at(2.00,1.5){};
            \draw[Edge](r)--(0);
            \node[LeafLabel,below of=2]
                {\begin{math}2\end{math}};
            \node[LeafLabel,below of=5]
                {\begin{math}1\end{math}};
            \node[LeafLabel,below of=7]
                {\begin{math}1\end{math}};
        \end{tikzpicture}
        \enspace \Deriv \enspace
        \begin{tikzpicture}[xscale=.28,yscale=.29,Centering]
            \node(2)at(0.00,-4.50){};
            \node(6)at(2.00,-7.50){};
            \node(8)at(4.00,-6.00){};
            \node[NodeST](0)at(2.00,0.00)
                {\begin{math}\Btt\end{math}};
            \node[NodeST](1)at(0.00,-3.00)
                {\begin{math}\Btt\end{math}};
            \node[NodeST](3)at(1.00,-1.50)
                {\begin{math}\Ctt\end{math}};
            \node[NodeST](4)at(3.00,-3.00)
                {\begin{math}\Att\end{math}};
            \node[NodeST](5)at(2.00,-6.00)
                {\begin{math}\Att\end{math}};
            \node[NodeST](7)at(3.00,-4.50)
                {\begin{math}\Ctt\end{math}};
            \draw[Edge](1)--(3);
            \draw[Edge](2)--(1);
            \draw[Edge](3)--(0);
            \draw[Edge](4)--(3);
            \draw[Edge](5)--(7);
            \draw[Edge](6)--(5);
            \draw[Edge](7)--(4);
            \draw[Edge](8)--(7);
            \node(r)at(2.00,1.5){};
            \draw[Edge](r)--(0);
            \node[LeafLabel,below of=2]
                {\begin{math}2\end{math}};
            \node[LeafLabel,below of=6]
                {\begin{math}2\end{math}};
            \node[LeafLabel,below of=8]
                {\begin{math}1\end{math}};
        \end{tikzpicture}
        \enspace \Deriv \enspace
        \begin{tikzpicture}[xscale=.32,yscale=.25,Centering]
            \node(2)at(0.00,-5.00){};
            \node(6)at(2.00,-8.33){};
            \node(9)at(4.00,-8.33){};
            \node[NodeST](0)at(2.00,0.00)
                {\begin{math}\Btt\end{math}};
            \node[NodeST](1)at(0.00,-3.33)
                {\begin{math}\Btt\end{math}};
            \node[NodeST](3)at(1.00,-1.67)
                {\begin{math}\Ctt\end{math}};
            \node[NodeST](4)at(3.00,-3.33)
                {\begin{math}\Att\end{math}};
            \node[NodeST](5)at(2.00,-6.67)
                {\begin{math}\Att\end{math}};
            \node[NodeST](7)at(3.00,-5.00)
                {\begin{math}\Ctt\end{math}};
            \node[NodeST](8)at(4.00,-6.67)
                {\begin{math}\Btt\end{math}};
            \draw[Edge](1)--(3);
            \draw[Edge](2)--(1);
            \draw[Edge](3)--(0);
            \draw[Edge](4)--(3);
            \draw[Edge](5)--(7);
            \draw[Edge](6)--(5);
            \draw[Edge](7)--(4);
            \draw[Edge](8)--(7);
            \draw[Edge](9)--(8);
            \node(r)at(2.00,1.5){};
            \draw[Edge](r)--(0);
            \node[LeafLabel,below of=2]
                {\begin{math}2\end{math}};
            \node[LeafLabel,below of=6]
                {\begin{math}2\end{math}};
            \node[LeafLabel,below of=9]
                {\begin{math}2\end{math}};
        \end{tikzpicture}
    \end{equation*}
    \vspace{-1.5em}
    \caption{\footnotesize A sequence of derivations in $\BUBTree$. The
    input colors of the elements of $\Bud_{\{1, 2\}}(\Free(C))$ are
    depicted below the leaves. The output color of all these elements is
    $1$. Since all input colors of the last tree are $2$, this tree is
    in~$\Lang(\BUBTree)$.}
    \label{fig:derivations_unary_binary_trees}
\end{figure}
\medbreak

A \Def{unary-binary tree} is a planar rooted tree $\Tfr$ such that all
internal nodes of $\Tfr$ are of arities $1$ or $2$, all nodes of $\Tfr$
of arity $1$ have a child which is an internal node of arity $2$ or is a
leaf, and all nodes of $\Tfr$ of arity $2$ have two children which are
internal nodes of arity $1$ or are leaves.
\medbreak

Let $L_{\BUBTree}$ be the set of unary-binary trees with a root of arity
$1$, all parents of the leaves are of arity $1$, and unary nodes are
labeled by $\Att$ or~$\Btt$.
\medbreak

\begin{Proposition} \label{prop:properties_BUBTree}
    The bud generating system $\BUBTree$ satisfies the following
    properties.
    \begin{enumerate}[label = ({\it {\roman*})}]
        \item \label{item:properties_BUBTree_faithful}
        It is faithful.
        \item \label{item:properties_BUBTree_unambiguous}
        It is unambiguous.
        \item \label{item:properties_BUBTree_language}
        The restriction of the pruning map $\Prune$ on the domain
        $\Lang\Par{\BUBTree}$ is a bijection between
        $\Lang\Par{\BUBTree}$ and $L_{\BUBTree}$.
        \item \label{item:properties_BUBTree_finitely_factorizing}
        The set of rules $\Rfr(1)$ finitely factorizes
        $\Bud_{\{1, 2\}}(\Free(C))$.
    \end{enumerate}
\end{Proposition}
\medbreak

\subsubsection{A bud generating system for $B$-perfect trees}%
\label{subsubsec:example_bbtree}
Let $B$ be a finite set of positive integers and $C_B$ be the monochrome
graded collection defined by
\begin{math}
    C_B := \bigsqcup_{n \in B} C_B(n)
    := \bigsqcup_{n \in B} \Bra{\Att_n}.
\end{math}
We consider the monochrome bud generating system
$\BBTree{B} := \Par{\Free\Par{C_B}, \Rfr_B}$ where $\Rfr_B$ is the set
of all corollas of $\Free\Par{C_B}(1)$.
Figure~\ref{fig:derivation_graph_bbtree} shows the synchronous
derivation graph of $\BBTree{\{2, 3\}}$.
\begin{figure}[ht]
    \centering
    \begin{equation*}
        \scalebox{.88}{
        \begin{tikzpicture}[yscale=1.1,xscale=.7,font=\scriptsize]
            \node(Unit)at(2.625,.45){\begin{math}\Unit\end{math}};
            \node(200)at(-3.25,-.5){\begin{math}
            \begin{tikzpicture}[xscale=.25,yscale=.23,Centering]
                \node(0)at(0.00,-1.50){};
                \node(2)at(2.00,-1.50){};
                \node[NodeST](1)at(1.00,0.00)
                    {\begin{math}\Att_2\end{math}};
                \draw[Edge](0)--(1);
                \draw[Edge](2)--(1);
                \node(r)at(1.00,1.5){};
                \draw[Edge](r)--(1);
            \end{tikzpicture}\end{math}};
            \node(3000)at(8.5,-.5){\begin{math}
            \begin{tikzpicture}[xscale=.25,yscale=.23,Centering]
                \node(0)at(0.00,-1.50){};
                \node(2)at(1.00,-1.50){};
                \node(3)at(2.00,-1.50){};
                \node[NodeST](1)at(1.00,0.00)
                    {\begin{math}\Att_3\end{math}};
                \draw[Edge](0)--(1);
                \draw[Edge](2)--(1);
                \draw[Edge](3)--(1);
                \node(r)at(1.00,1.5){};
                \draw[Edge](r)--(1);
            \end{tikzpicture}\end{math}};
            \node(2200200)at(-5.5,-2){\begin{math}
            \begin{tikzpicture}[xscale=.13,yscale=.18,Centering]
                \node(0)at(0.00,-4.67){};
                \node(2)at(2.00,-4.67){};
                \node(4)at(4.00,-4.67){};
                \node(6)at(6.00,-4.67){};
                \node[NodeST](1)at(1.00,-2.33)
                    {\begin{math}\Att_2\end{math}};
                \node[NodeST](3)at(3.00,0.00)
                    {\begin{math}\Att_2\end{math}};
                \node[NodeST](5)at(5.00,-2.33)
                    {\begin{math}\Att_2\end{math}};
                \draw[Edge](0)--(1);
                \draw[Edge](1)--(3);
                \draw[Edge](2)--(1);
                \draw[Edge](4)--(5);
                \draw[Edge](5)--(3);
                \draw[Edge](6)--(5);
                \node(r)at(3.00,2){};
                \draw[Edge](r)--(3);
            \end{tikzpicture}\end{math}};
            \node(22003000)at(-4,-2){\begin{math}
            \begin{tikzpicture}[xscale=.13,yscale=.18,Centering]
                \node(0)at(0.00,-4.67){};
                \node(2)at(2.00,-4.67){};
                \node(4)at(4.00,-4.67){};
                \node(6)at(5.00,-4.67){};
                \node(7)at(6.00,-4.67){};
                \node[NodeST](1)at(1.00,-2.33)
                    {\begin{math}\Att_2\end{math}};
                \node[NodeST](3)at(3.00,0.00)
                    {\begin{math}\Att_2\end{math}};
                \node[NodeST](5)at(5.00,-2.33)
                    {\begin{math}\Att_3\end{math}};
                \draw[Edge](0)--(1);
                \draw[Edge](1)--(3);
                \draw[Edge](2)--(1);
                \draw[Edge](4)--(5);
                \draw[Edge](5)--(3);
                \draw[Edge](6)--(5);
                \draw[Edge](7)--(5);
                \node(r)at(3.00,2.00){};
                \draw[Edge](r)--(3);
            \end{tikzpicture}\end{math}};
            \node(23000200)at(-2.5,-2){\begin{math}
            \begin{tikzpicture}[xscale=.13,yscale=.18,Centering]
                \node(0)at(0.00,-4.67){};
                \node(2)at(1.00,-4.67){};
                \node(3)at(2.00,-4.67){};
                \node(5)at(4.00,-4.67){};
                \node(7)at(6.00,-4.67){};
                \node[NodeST](1)at(1.00,-2.33)
                    {\begin{math}\Att_3\end{math}};
                \node[NodeST](4)at(3.00,0.00)
                    {\begin{math}\Att_2\end{math}};
                \node[NodeST](6)at(5.00,-2.33)
                    {\begin{math}\Att_2\end{math}};
                \draw[Edge](0)--(1);
                \draw[Edge](1)--(4);
                \draw[Edge](2)--(1);
                \draw[Edge](3)--(1);
                \draw[Edge](5)--(6);
                \draw[Edge](6)--(4);
                \draw[Edge](7)--(6);
                \node(r)at(3.00,2.00){};
                \draw[Edge](r)--(4);
            \end{tikzpicture}\end{math}};
            \node(230003000)at(-1,-2){\begin{math}
            \begin{tikzpicture}[xscale=.13,yscale=.18,Centering]
                \node(0)at(0.00,-4.67){};
                \node(2)at(1.00,-4.67){};
                \node(3)at(2.00,-4.67){};
                \node(5)at(4.00,-4.67){};
                \node(7)at(5.00,-4.67){};
                \node(8)at(6.00,-4.67){};
                \node[NodeST](1)at(1.00,-2.33)
                    {\begin{math}\Att_3\end{math}};
                \node[NodeST](4)at(3.00,0.00)
                    {\begin{math}\Att_2\end{math}};
                \node[NodeST](6)at(5.00,-2.33)
                    {\begin{math}\Att_3\end{math}};
                \draw[Edge](0)--(1);
                \draw[Edge](1)--(4);
                \draw[Edge](2)--(1);
                \draw[Edge](3)--(1);
                \draw[Edge](5)--(6);
                \draw[Edge](6)--(4);
                \draw[Edge](7)--(6);
                \draw[Edge](8)--(6);
                \node(r)at(3.00,2){};
                \draw[Edge](r)--(4);
            \end{tikzpicture}\end{math}};
            \node(3200200200)at(1.5,-2){\begin{math}
            \begin{tikzpicture}[xscale=.15,yscale=.15,Centering]
                \node(0)at(0.00,-6){};
                \node(2)at(2.00,-6){};
                \node(4)at(3.00,-6){};
                \node(6)at(5.00,-6){};
                \node(7)at(6.00,-6){};
                \node(9)at(8.00,-6){};
                \node[NodeST](1)at(1.00,-3)
                    {\begin{math}\Att_2\end{math}};
                \node[NodeST](3)at(4.00,0.00)
                    {\begin{math}\Att_3\end{math}};
                \node[NodeST](5)at(4.00,-3)
                    {\begin{math}\Att_2\end{math}};
                \node[NodeST](8)at(7.00,-3)
                    {\begin{math}\Att_2\end{math}};
                \draw[Edge](0)--(1);
                \draw[Edge](1)--(3);
                \draw[Edge](2)--(1);
                \draw[Edge](4)--(5);
                \draw[Edge](5)--(3);
                \draw[Edge](6)--(5);
                \draw[Edge](7)--(8);
                \draw[Edge](8)--(3);
                \draw[Edge](9)--(8);
                \node(r)at(4.00,2.25){};
                \draw[Edge](r)--(3);
            \end{tikzpicture}\end{math}};
            \node(32002003000)at(3.5,-2){\begin{math}
            \begin{tikzpicture}[xscale=.15,yscale=.15,Centering]
                \node(0)at(0.00,-6){};
                \node(10)at(8.00,-6){};
                \node(2)at(2.00,-6){};
                \node(4)at(3.00,-6){};
                \node(6)at(5.00,-6){};
                \node(7)at(6.00,-6){};
                \node(9)at(7.00,-6){};
                \node[NodeST](1)at(1.00,-3)
                    {\begin{math}\Att_2\end{math}};
                \node[NodeST](3)at(4.00,0.00)
                    {\begin{math}\Att_3\end{math}};
                \node[NodeST](5)at(4.00,-3)
                    {\begin{math}\Att_2\end{math}};
                \node[NodeST](8)at(7.00,-3)
                    {\begin{math}\Att_3\end{math}};
                \draw[Edge](0)--(1);
                \draw[Edge](1)--(3);
                \draw[Edge](10)--(8);
                \draw[Edge](2)--(1);
                \draw[Edge](4)--(5);
                \draw[Edge](5)--(3);
                \draw[Edge](6)--(5);
                \draw[Edge](7)--(8);
                \draw[Edge](8)--(3);
                \draw[Edge](9)--(8);
                \node(r)at(4.00,2.25){};
                \draw[Edge](r)--(3);
            \end{tikzpicture}\end{math}};
            \node(32003000200)at(5.5,-2){\begin{math}
            \begin{tikzpicture}[xscale=.15,yscale=.15,Centering]
                \node(0)at(0.00,-6){};
                \node(10)at(8.00,-6){};
                \node(2)at(2.00,-6){};
                \node(4)at(3.00,-6){};
                \node(6)at(4.00,-6){};
                \node(7)at(5.00,-6){};
                \node(8)at(6.00,-6){};
                \node[NodeST](1)at(1.00,-3)
                    {\begin{math}\Att_2\end{math}};
                \node[NodeST](3)at(4.00,0.00)
                    {\begin{math}\Att_3\end{math}};
                \node[NodeST](5)at(4.00,-3)
                    {\begin{math}\Att_3\end{math}};
                \node[NodeST](9)at(7.00,-3)
                    {\begin{math}\Att_2\end{math}};
                \draw[Edge](0)--(1);
                \draw[Edge](1)--(3);
                \draw[Edge](10)--(9);
                \draw[Edge](2)--(1);
                \draw[Edge](4)--(5);
                \draw[Edge](5)--(3);
                \draw[Edge](6)--(5);
                \draw[Edge](7)--(5);
                \draw[Edge](8)--(9);
                \draw[Edge](9)--(3);
                \node(r)at(4.00,2.25){};
                \draw[Edge](r)--(3);
            \end{tikzpicture}\end{math}};
            \node(320030003000)at(7.5,-2){\begin{math}
            \begin{tikzpicture}[xscale=.15,yscale=.15,Centering]
                \node(0)at(0.00,-6){};
                \node(10)at(7.00,-6){};
                \node(11)at(8.00,-6){};
                \node(2)at(2.00,-6){};
                \node(4)at(3.00,-6){};
                \node(6)at(4.00,-6){};
                \node(7)at(5.00,-6){};
                \node(8)at(6.00,-6){};
                \node[NodeST](1)at(1.00,-3)
                    {\begin{math}\Att_2\end{math}};
                \node[NodeST](3)at(4.00,0.00)
                    {\begin{math}\Att_3\end{math}};
                \node[NodeST](5)at(4.00,-3)
                    {\begin{math}\Att_3\end{math}};
                \node[NodeST](9)at(7.00,-3)
                    {\begin{math}\Att_3\end{math}};
                \draw[Edge](0)--(1);
                \draw[Edge](1)--(3);
                \draw[Edge](10)--(9);
                \draw[Edge](11)--(9);
                \draw[Edge](2)--(1);
                \draw[Edge](4)--(5);
                \draw[Edge](5)--(3);
                \draw[Edge](6)--(5);
                \draw[Edge](7)--(5);
                \draw[Edge](8)--(9);
                \draw[Edge](9)--(3);
                \node(r)at(4.00,2.25){};
                \draw[Edge](r)--(3);
            \end{tikzpicture}\end{math}};
            \node(33000200200)at(9.5,-2){\begin{math}
            \begin{tikzpicture}[xscale=.15,yscale=.15,Centering]
                \node(0)at(0.00,-6){};
                \node(10)at(8.00,-6){};
                \node(2)at(1.00,-6){};
                \node(3)at(2.00,-6){};
                \node(5)at(3.00,-6){};
                \node(7)at(5.00,-6){};
                \node(8)at(6.00,-6){};
                \node[NodeST](1)at(1.00,-3)
                    {\begin{math}\Att_3\end{math}};
                \node[NodeST](4)at(4.00,0.00)
                    {\begin{math}\Att_3\end{math}};
                \node[NodeST](6)at(4.00,-3)
                    {\begin{math}\Att_2\end{math}};
                \node[NodeST](9)at(7.00,-3)
                    {\begin{math}\Att_2\end{math}};
                \draw[Edge](0)--(1);
                \draw[Edge](1)--(4);
                \draw[Edge](10)--(9);
                \draw[Edge](2)--(1);
                \draw[Edge](3)--(1);
                \draw[Edge](5)--(6);
                \draw[Edge](6)--(4);
                \draw[Edge](7)--(6);
                \draw[Edge](8)--(9);
                \draw[Edge](9)--(4);
                \node(r)at(4.00,2.25){};
                \draw[Edge](r)--(4);
            \end{tikzpicture}\end{math}};
            \node(330002003000)at(11.5,-2){\begin{math}
            \begin{tikzpicture}[xscale=.15,yscale=.15,Centering]
                \node(0)at(0.00,-6){};
                \node(10)at(7.00,-6){};
                \node(11)at(8.00,-6){};
                \node(2)at(1.00,-6){};
                \node(3)at(2.00,-6){};
                \node(5)at(3.00,-6){};
                \node(7)at(5.00,-6){};
                \node(8)at(6.00,-6){};
                \node[NodeST](1)at(1.00,-3.00)
                    {\begin{math}\Att_3\end{math}};
                \node[NodeST](4)at(4.00,0.00)
                    {\begin{math}\Att_3\end{math}};
                \node[NodeST](6)at(4.00,-3.00)
                    {\begin{math}\Att_2\end{math}};
                \node[NodeST](9)at(7.00,-3.00)
                    {\begin{math}\Att_3\end{math}};
                \draw[Edge](0)--(1);
                \draw[Edge](1)--(4);
                \draw[Edge](10)--(9);
                \draw[Edge](11)--(9);
                \draw[Edge](2)--(1);
                \draw[Edge](3)--(1);
                \draw[Edge](5)--(6);
                \draw[Edge](6)--(4);
                \draw[Edge](7)--(6);
                \draw[Edge](8)--(9);
                \draw[Edge](9)--(4);
                \node(r)at(4.00,2.25){};
                \draw[Edge](r)--(4);
            \end{tikzpicture}\end{math}};
            \node(330003000200)at(13.5,-2){\begin{math}
            \begin{tikzpicture}[xscale=.15,yscale=.15,Centering]
                \node(0)at(0.00,-6){};
                \node(11)at(8.00,-6){};
                \node(2)at(1.00,-6){};
                \node(3)at(2.00,-6){};
                \node(5)at(3.00,-6){};
                \node(7)at(4.00,-6){};
                \node(8)at(5.00,-6){};
                \node(9)at(6.00,-6){};
                \node[NodeST](1)at(1.00,-3.00)
                    {\begin{math}\Att_3\end{math}};
                \node[NodeST](10)at(7.00,-3.00)
                    {\begin{math}\Att_2\end{math}};
                \node[NodeST](4)at(4.00,0.00)
                    {\begin{math}\Att_3\end{math}};
                \node[NodeST](6)at(4.00,-3.00)
                    {\begin{math}\Att_3\end{math}};
                \draw[Edge](0)--(1);
                \draw[Edge](1)--(4);
                \draw[Edge](10)--(4);
                \draw[Edge](11)--(10);
                \draw[Edge](2)--(1);
                \draw[Edge](3)--(1);
                \draw[Edge](5)--(6);
                \draw[Edge](6)--(4);
                \draw[Edge](7)--(6);
                \draw[Edge](8)--(6);
                \draw[Edge](9)--(10);
                \node(r)at(4.00,2.25){};
                \draw[Edge](r)--(4);
            \end{tikzpicture}\end{math}};
            \node(3300030003000)at(15.5,-2){\begin{math}
            \begin{tikzpicture}[xscale=.15,yscale=.15,Centering]
                \node(0)at(0.00,-6){};
                \node(11)at(7.00,-6){};
                \node(12)at(8.00,-6){};
                \node(2)at(1.00,-6){};
                \node(3)at(2.00,-6){};
                \node(5)at(3.00,-6){};
                \node(7)at(4.00,-6){};
                \node(8)at(5.00,-6){};
                \node(9)at(6.00,-6){};
                \node[NodeST](1)at(1.00,-3)
                    {\begin{math}\Att_3\end{math}};
                \node[NodeST](10)at(7.00,-3)
                    {\begin{math}\Att_3\end{math}};
                \node[NodeST](4)at(4.00,0.00)
                    {\begin{math}\Att_3\end{math}};
                \node[NodeST](6)at(4.00,-3)
                    {\begin{math}\Att_3\end{math}};
                \draw[Edge](0)--(1);
                \draw[Edge](1)--(4);
                \draw[Edge](10)--(4);
                \draw[Edge](11)--(10);
                \draw[Edge](12)--(10);
                \draw[Edge](2)--(1);
                \draw[Edge](3)--(1);
                \draw[Edge](5)--(6);
                \draw[Edge](6)--(4);
                \draw[Edge](7)--(6);
                \draw[Edge](8)--(6);
                \draw[Edge](9)--(10);
                \node(r)at(4.00,2.25){};
                \draw[Edge](r)--(4);
            \end{tikzpicture}\end{math}};
            \draw[EdgeDeriv](Unit)
                edge[bend right=7]
                node[anchor=south,font=\scriptsize,color=ColBlack]
                {}(200);
            \draw[EdgeDeriv](Unit)
                edge[bend left=7]
                node[anchor=south,font=\scriptsize,color=ColBlack]
                {}(3000);
            \draw[EdgeDeriv](200)
                edge[bend right=20]
                node[anchor=east,font=\scriptsize,color=ColBlack]
                {}(2200200);
            \draw[EdgeDeriv](200)
                edge[]
                node[anchor=east,font=\scriptsize,color=ColBlack]
                {}(22003000);
            \draw[EdgeDeriv](200)
                edge[]
                node[anchor=west,font=\scriptsize,color=ColBlack]
                {}(23000200);
            \draw[EdgeDeriv](200)
                edge[bend left=20]
                node[anchor=west,font=\scriptsize,color=ColBlack]
                {}(230003000);
            \draw[EdgeDeriv](3000)
                edge[bend right=13]
                node[anchor=south,font=\scriptsize,color=ColBlack]
                {}(3200200200);
            \draw[EdgeDeriv](3000)
                edge[bend right=8]
                node[anchor=south,font=\scriptsize,color=ColBlack]
                {}(32002003000);
            \draw[EdgeDeriv](3000)
                edge[bend right=5]
                node[anchor=south,font=\scriptsize,color=ColBlack]
                {}(32003000200);
            \draw[EdgeDeriv](3000)
                edge[]
                node[anchor=east,font=\scriptsize,color=ColBlack]
                {}(320030003000);
            \draw[EdgeDeriv](3000)
                edge[]
                node[anchor=west,font=\scriptsize,color=ColBlack]
                {}(33000200200);
            \draw[EdgeDeriv](3000)
                edge[bend left=5]
                node[anchor=south,font=\scriptsize,color=ColBlack]
                {}(330002003000);
            \draw[EdgeDeriv](3000)
                edge[bend left=8]
                node[anchor=south,font=\scriptsize,color=ColBlack]
                {}(330003000200);
            \draw[EdgeDeriv](3000)
                edge[bend left=13]
                node[anchor=south,font=\scriptsize,color=ColBlack]
                {}(3300030003000);
            \draw[EdgeDeriv,dotted,shorten >=6mm](2200200)--(-5.5,-3.5);
            \draw[EdgeDeriv,dotted,shorten >=6mm](22003000)--(-4,-3.5);
            \draw[EdgeDeriv,dotted,shorten >=6mm](23000200)--(-2.5,-3.5);
            \draw[EdgeDeriv,dotted,shorten >=6mm](230003000)--(-1,-3.5);
            \draw[EdgeDeriv,dotted,shorten >=6mm]
                (3200200200)--(1.5,-3.5);
            \draw[EdgeDeriv,dotted,shorten >=6mm]
                (32002003000)--(3.5,-3.5);
            \draw[EdgeDeriv,dotted,shorten >=6mm]
                (32003000200)--(5.5,-3.5);
            \draw[EdgeDeriv,dotted,shorten >=6mm]
                (320030003000)--(7.5,-3.5);
            \draw[EdgeDeriv,dotted,shorten >=6mm]
                (33000200200)--(9.5,-3.5);
            \draw[EdgeDeriv,dotted,shorten >=6mm]
                (330002003000)--(11.5,-3.5);
            \draw[EdgeDeriv,dotted,shorten >=6mm]
                (330003000200)--(13.5,-3.5);
            \draw[EdgeDeriv,dotted,shorten >=6mm]
                (3300030003000)--(15.5,-3.5);
        \end{tikzpicture}}
    \end{equation*}
    \vspace{-1.5em}
    \caption{\footnotesize The synchronous derivation graph of
    $\BBTree{\{2, 3\}}$.}
    \label{fig:derivation_graph_bbtree}
\end{figure}
\medbreak

A \Def{$B$-perfect tree} is a planar rooted tree $\Tfr$ such that all
internal nodes of $\Tfr$ have an arity in $B$ and all paths connecting
the root of $\Tfr$ to its leaves have the same length. These trees and
their generating series have been studied for the particular case
$B := \{2, 3\}$~\cite{MPRS79,CLRS09} and appear as data structures in
computer science (see~\cite{Odl82,Knu98,FS09}).
\medbreak

\begin{Proposition} \label{prop:properties_BBTree}
    For any finite set $B$ of positive integers, the bud generating
    system $\BBTree{B}$ satisfies the following properties.
    \begin{enumerate}[label = ({\it {\roman*})}]
        \item \label{item:properties_BBTree_faithful}
        It is synchronously faithful.
        \item \label{item:properties_BBTree_unambiguous}
        It is synchronously unambiguous.
        \item \label{item:properties_BBTree_language}
        The synchronous language $\SyncLang\Par{\BBTree{B}}$ of
        $\BBTree{B}$ is the set of all $B$-perfect trees.
        \item \label{item:properties_BBTree_finitely_factorizing}
        The set of rules $\Rfr_B(1)$ finitely factorizes
        $\Free\Par{C_B}$.
        \item \label{item:properties_BBTree_finitely_locally_finite}
        When $1 \notin B$, the generating series
        $\GenS_{\SyncLang(\BBTree{B})}$ of the synchronous language
        of $\BBTree{B}$ is well-defined.
    \end{enumerate}
\end{Proposition}
\medbreak

Property~\ref{item:properties_BBTree_finitely_locally_finite} of
Proposition~\ref{prop:properties_BBTree} is a consequence of the fact
that when $1 \notin B$, $\Free\Par{C_B}$ is locally finite and therefore
there are only finitely many elements in $\SyncLang\Par{\BBTree{B}}$ of
a given arity. By Property~\ref{item:properties_BBTree_language} of
Proposition~\ref{prop:properties_BBTree}, the sequences enumerating the
elements of $\SyncLang(\BBTree{B})$ with respect to their arity are, for
instance, Sequence~\OEIS{A014535} of~\cite{Slo} for $B = \{2, 3\}$ which
starts by
\begin{equation}
    1, 1, 1, 1, 2, 2, 3, 4, 5, 8, 14, 23, 32, 43, 63, 97, 149, 224, 332,
    489,
\end{equation}
and Sequence~\OEIS{A037026} of~\cite{Slo} for $B = \{2, 3, 4\}$ which
starts by
\begin{equation}
    1, 1, 1, 2, 2, 4, 5, 9, 15, 28, 45, 73, 116, 199, 345 601, 1021,
    1738, 2987, 5244.
\end{equation}
\medbreak

\subsubsection{A bud generating system for balanced binary trees}%
\label{subsubsec:example_bbaltree}
Consider the bud generating system
$\BBalTree := (\Mag, \{1, 2\}, \Rfr, \{1\}, \{1\})$ where
\begin{equation}
    \Rfr := \Bra{
    \Par{1, \CorollaTwo{}, 11},
    \Par{1, \CorollaTwo{}, 12},
    \Par{1, \CorollaTwo{}, 21},
    \Par{2, \TreeLeaf, 1}}.
\end{equation}
Figure~\ref{fig:derivations_bal_tree} shows a sequence of synchronous
derivations in $\BBalTree$.
\begin{figure}[ht]
    \centering
    \begin{equation*}
        \begin{tikzpicture}[xscale=.13,yscale=.12,Centering]
            \node[Leaf](0)at(0.00,-1.50){};
            \draw[Edge](0)--(0,0);
            \node[LeafLabel,below of=0]{\begin{math}1\end{math}};
        \end{tikzpicture}
        \enspace \SyncDeriv \enspace
        \begin{tikzpicture}[xscale=.13,yscale=.15,Centering]
            \node[Leaf](0)at(0.00,-1.50){};
            \node[Leaf](2)at(2.00,-1.50){};
            \node[Node](1)at(1.00,0.00){};
            \draw[Edge](0)--(1);
            \draw[Edge](2)--(1);
            \node(r)at(1.00,2){};
            \draw[Edge](r)--(1);
            \node[LeafLabel,below of=0]{\begin{math}1\end{math}};
            \node[LeafLabel,below of=2]{\begin{math}2\end{math}};
        \end{tikzpicture}
        \enspace \SyncDeriv \enspace
        \begin{tikzpicture}[scale=.17,Centering]
            \node[Leaf](0)at(0.00,-3.33){};
            \node[Leaf](2)at(2.00,-3.33){};
            \node[Leaf](4)at(4.00,-1.67){};
            \node[Node](1)at(1.00,-1.67){};
            \node[Node](3)at(3.00,0.00){};
            \draw[Edge](0)--(1);
            \draw[Edge](1)--(3);
            \draw[Edge](2)--(1);
            \draw[Edge](4)--(3);
            \node(r)at(3.00,1.75){};
            \draw[Edge](r)--(3);
            \node[LeafLabel,below of=0]{\begin{math}2\end{math}};
            \node[LeafLabel,below of=2]{\begin{math}1\end{math}};
            \node[LeafLabel,below of=4]{\begin{math}1\end{math}};
        \end{tikzpicture}
        \enspace \SyncDeriv \enspace
        \begin{tikzpicture}[scale=.13,Centering]
            \node[Leaf](0)at(0.00,-4.50){};
            \node[Leaf](2)at(2.00,-6.75){};
            \node[Leaf](4)at(4.00,-6.75){};
            \node[Leaf](6)at(6.00,-4.50){};
            \node[Leaf](8)at(8.00,-4.50){};
            \node[Node](1)at(1.00,-2.25){};
            \node[Node](3)at(3.00,-4.50){};
            \node[Node](5)at(5.00,0.00){};
            \node[Node](7)at(7.00,-2.25){};
            \draw[Edge](0)--(1);
            \draw[Edge](1)--(5);
            \draw[Edge](2)--(3);
            \draw[Edge](3)--(1);
            \draw[Edge](4)--(3);
            \draw[Edge](6)--(7);
            \draw[Edge](7)--(5);
            \draw[Edge](8)--(7);
            \node(r)at(5.00,2.5){};
            \draw[Edge](r)--(5);
            \node[LeafLabel,below of=0]{\begin{math}1\end{math}};
            \node[LeafLabel,below of=2]{\begin{math}1\end{math}};
            \node[LeafLabel,below of=4]{\begin{math}1\end{math}};
            \node[LeafLabel,below of=6]{\begin{math}2\end{math}};
            \node[LeafLabel,below of=8]{\begin{math}1\end{math}};
        \end{tikzpicture}
        \enspace \SyncDeriv \enspace
        \begin{tikzpicture}[xscale=.12,yscale=.09,Centering]
            \node[Leaf](0)at(0.00,-10.20){};
            \node[Leaf](10)at(10.00,-13.60){};
            \node[Leaf](12)at(12.00,-6.80){};
            \node[Leaf](14)at(14.00,-10.20){};
            \node[Leaf](16)at(16.00,-10.20){};
            \node[Leaf](2)at(2.00,-10.20){};
            \node[Leaf](4)at(4.00,-13.60){};
            \node[Leaf](6)at(6.00,-13.60){};
            \node[Leaf](8)at(8.00,-13.60){};
            \node[Node](1)at(1.00,-6.80){};
            \node[Node](11)at(11.00,0.00){};
            \node[Node](13)at(13.00,-3.40){};
            \node[Node](15)at(15.00,-6.80){};
            \node[Node](3)at(3.00,-3.40){};
            \node[Node](5)at(5.00,-10.20){};
            \node[Node](7)at(7.00,-6.80){};
            \node[Node](9)at(9.00,-10.20){};
            \draw[Edge](0)--(1);
            \draw[Edge](1)--(3);
            \draw[Edge](10)--(9);
            \draw[Edge](12)--(13);
            \draw[Edge](13)--(11);
            \draw[Edge](14)--(15);
            \draw[Edge](15)--(13);
            \draw[Edge](16)--(15);
            \draw[Edge](2)--(1);
            \draw[Edge](3)--(11);
            \draw[Edge](4)--(5);
            \draw[Edge](5)--(7);
            \draw[Edge](6)--(5);
            \draw[Edge](7)--(3);
            \draw[Edge](8)--(9);
            \draw[Edge](9)--(7);
            \node(r)at(11.00,3.5){};
            \draw[Edge](r)--(11);
            \node[LeafLabel,below of=0]{\begin{math}1\end{math}};
            \node[LeafLabel,below of=2]{\begin{math}2\end{math}};
            \node[LeafLabel,below of=4]{\begin{math}1\end{math}};
            \node[LeafLabel,below of=6]{\begin{math}2\end{math}};
            \node[LeafLabel,below of=8]{\begin{math}1\end{math}};
            \node[LeafLabel,below of=10]{\begin{math}1\end{math}};
            \node[LeafLabel,below of=12]{\begin{math}1\end{math}};
            \node[LeafLabel,below of=14]{\begin{math}2\end{math}};
            \node[LeafLabel,below of=16]{\begin{math}1\end{math}};
        \end{tikzpicture}
        \enspace \SyncDeriv \enspace
        \begin{tikzpicture}[xscale=.11,yscale=.06,Centering]
            \node[Leaf](0)at(0.00,-19.33){};
            \node[Leaf](10)at(10.00,-19.33){};
            \node[Leaf](12)at(12.00,-24.17){};
            \node[Leaf](14)at(14.00,-24.17){};
            \node[Leaf](16)at(16.00,-24.17){};
            \node[Leaf](18)at(18.00,-24.17){};
            \node[Leaf](2)at(2.00,-19.33){};
            \node[Leaf](20)at(20.00,-14.50){};
            \node[Leaf](22)at(22.00,-14.50){};
            \node[Leaf](24)at(24.00,-14.50){};
            \node[Leaf](26)at(26.00,-19.33){};
            \node[Leaf](28)at(28.00,-19.33){};
            \node[Leaf](4)at(4.00,-14.50){};
            \node[Leaf](6)at(6.00,-24.17){};
            \node[Leaf](8)at(8.00,-24.17){};
            \node[Node](1)at(1.00,-14.50){};
            \node[Node](11)at(11.00,-9.67){};
            \node[Node](13)at(13.00,-19.33){};
            \node[Node](15)at(15.00,-14.50){};
            \node[Node](17)at(17.00,-19.33){};
            \node[Node](19)at(19.00,0.00){};
            \node[Node](21)at(21.00,-9.67){};
            \node[Node](23)at(23.00,-4.83){};
            \node[Node](25)at(25.00,-9.67){};
            \node[Node](27)at(27.00,-14.50){};
            \node[Node](3)at(3.00,-9.67){};
            \node[Node](5)at(5.00,-4.83){};
            \node[Node](7)at(7.00,-19.33){};
            \node[Node](9)at(9.00,-14.50){};
            \draw[Edge](0)--(1);
            \draw[Edge](1)--(3);
            \draw[Edge](10)--(9);
            \draw[Edge](11)--(5);
            \draw[Edge](12)--(13);
            \draw[Edge](13)--(15);
            \draw[Edge](14)--(13);
            \draw[Edge](15)--(11);
            \draw[Edge](16)--(17);
            \draw[Edge](17)--(15);
            \draw[Edge](18)--(17);
            \draw[Edge](2)--(1);
            \draw[Edge](20)--(21);
            \draw[Edge](21)--(23);
            \draw[Edge](22)--(21);
            \draw[Edge](23)--(19);
            \draw[Edge](24)--(25);
            \draw[Edge](25)--(23);
            \draw[Edge](26)--(27);
            \draw[Edge](27)--(25);
            \draw[Edge](28)--(27);
            \draw[Edge](3)--(5);
            \draw[Edge](4)--(3);
            \draw[Edge](5)--(19);
            \draw[Edge](6)--(7);
            \draw[Edge](7)--(9);
            \draw[Edge](8)--(7);
            \draw[Edge](9)--(11);
            \node(r)at(19.00,4.75){};
            \draw[Edge](r)--(19);
            \node[LeafLabel,below of=0]{\begin{math}1\end{math}};
            \node[LeafLabel,below of=2]{\begin{math}1\end{math}};
            \node[LeafLabel,below of=4]{\begin{math}1\end{math}};
            \node[LeafLabel,below of=6]{\begin{math}1\end{math}};
            \node[LeafLabel,below of=8]{\begin{math}1\end{math}};
            \node[LeafLabel,below of=10]{\begin{math}1\end{math}};
            \node[LeafLabel,below of=12]{\begin{math}1\end{math}};
            \node[LeafLabel,below of=14]{\begin{math}1\end{math}};
            \node[LeafLabel,below of=16]{\begin{math}1\end{math}};
            \node[LeafLabel,below of=18]{\begin{math}1\end{math}};
            \node[LeafLabel,below of=20]{\begin{math}1\end{math}};
            \node[LeafLabel,below of=22]{\begin{math}1\end{math}};
            \node[LeafLabel,below of=24]{\begin{math}1\end{math}};
            \node[LeafLabel,below of=26]{\begin{math}1\end{math}};
            \node[LeafLabel,below of=28]{\begin{math}1\end{math}};
        \end{tikzpicture}
    \end{equation*}
    \vspace{-1.5em}
    \caption{\footnotesize A sequence of synchronous derivations in
    $\BBalTree$. The input colors of the elements of
    $\Bud_{\{1, 2\}}(\Mag)$ are depicted below the leaves. The output
    color of all these elements is $1$. Since all input colors of the
    last tree are $1$, this tree is in~$\SyncLang(\BBalTree)$.}
    \label{fig:derivations_bal_tree}
\end{figure}
\medbreak

The \Def{height} of a binary tree $\Tfr$ is the height of $\Tfr$ seen
as a monochrome syntax tree. A \Def{balanced binary tree}~\cite{AVL62}
is a binary tree $\Tfr$ wherein, for any internal node $x$ of $\Tfr$,
the difference between the height of the left subtree and the height of
the right subtree of $x$ is $-1$, $0$, or~$1$.
\medbreak

\begin{Proposition} \label{prop:properties_BBalTree}
    The bud generating system $\BBalTree$ satisfies the following
    properties.
    \begin{enumerate}[label = ({\it {\roman*})}]
        \item \label{item:properties_BBalTree_faithful}
        It is synchronously faithful.
        \item \label{item:properties_BBalTree_unambiguous}
        It is synchronously unambiguous.
        \item \label{item:properties_BBalTree_language}
        The restriction of the pruning map $\Prune$ on the domain
        $\SyncLang\Par{\BBalTree}$
        is a bijection between $\SyncLang\Par{\BBalTree}$ and the set
        of balanced binary trees.
        \item \label{item:properties_BBalTree_finitely_factorizing}
        The set of rules $\Rfr(1)$ finitely factorizes
        $\Bud_{\{1, 2\}}(\Mag)$.
    \end{enumerate}
\end{Proposition}
\medbreak

Properties~\ref{item:properties_BBalTree_unambiguous}
and~\ref{item:properties_BBalTree_language} of
Proposition~\ref{prop:properties_BBalTree} are based upon
combinatorial properties of a synchronous grammar $\Gca$ of balanced
binary trees defined in~\cite{Gir12} and satisfying
$\SG(\Gca) = \BBalTree$ (see
Section~\ref{subsubsec:synchronous_grammars} and
Proposition~\ref{prop:emulation_synchronous_grammars}). Besides,
Properties~\ref{item:properties_BBalTree_faithful}
and~\ref{item:properties_BBalTree_language} of
Proposition~\ref{prop:properties_BBalTree} together imply that the
sequence enumerating the elements of $\SyncLang(\BBalTree)$ with respect
to their arity is the one enumerating the balanced binary trees. This
sequence in Sequence~\OEIS{A006265} of~\cite{Slo}, starting by
\begin{equation} \label{equ:first_terms_balanced_binary_trees}
    1, 1, 2, 1, 4, 6, 4, 17, 32, 44, 60, 70,
    184, 476, 872, 1553, 2720, 4288, 6312, 9004.
\end{equation}
\medbreak

\subsection{Series of bud generating systems}
We now consider the bud generating systems constructed in
Section~\ref{subsec:examples_bud_generating_systems} to give some
examples of hook generating series. We also put into practice what we
have exposed in Sections~\ref{subsubsec:syntactic_generating_series}
and~\ref{subsubsec:synchronous_generating_series} to compute the
generating series of languages or synchronous languages of bud
generating systems by using syntactic generating series and synchronous
generating series.
\medbreak

\subsubsection{Hook coefficients for binary trees}%
\label{subsubsec:example_hook_coeff_binary_trees}
Let us consider the hook bud generating system $\HookSystem_{\Mag, G}$
where
\begin{math}
    G := \Bra{\scalebox{.8}{\raisebox{.3em}{\CorollaTwo{}}}}.
\end{math}
This bud generating system leads to the definition of a statistic on
binary trees, provided by the coefficients of the hook generating series
$\Hook\Par{\HookSystem_{\Mag, G}}$ which begins by
\begin{multline}
    \Hook\Par{\HookSystem_{\Mag, G}} =
    \scalebox{.8}{\begin{tikzpicture}[scale=.5,Centering]
        \node[Leaf](0)at(0.00,0.00){};
        \node(r)at(0.00,0.75){};
        \draw[Edge](r)--(0);
    \end{tikzpicture}}
    \enspace + \enspace
    \scalebox{.8}{\begin{tikzpicture}[scale=.12,Centering]
        \node[Leaf](0)at(0.00,-1.50){};
        \node[Leaf](2)at(2.00,-1.50){};
        \node[Node](1)at(1.00,0.00){};
        \draw[Edge](0)--(1);
        \draw[Edge](2)--(1);
        \node(r)at(1.00,2.5){};
        \draw[Edge](r)--(1);
    \end{tikzpicture}}
    \enspace + \enspace
    \scalebox{.8}{\begin{tikzpicture}[scale=.12,Centering]
        \node[Leaf](0)at(0.00,-3.33){};
        \node[Leaf](2)at(2.00,-3.33){};
        \node[Leaf](4)at(4.00,-1.67){};
        \node[Node](1)at(1.00,-1.67){};
        \node[Node](3)at(3.00,0.00){};
        \draw[Edge](0)--(1);
        \draw[Edge](1)--(3);
        \draw[Edge](2)--(1);
        \draw[Edge](4)--(3);
        \node(r)at(3.00,2.5){};
        \draw[Edge](r)--(3);
    \end{tikzpicture}}
    \enspace + \enspace
    \scalebox{.8}{\begin{tikzpicture}[scale=.12,Centering]
        \node[Leaf](0)at(0.00,-1.67){};
        \node[Leaf](2)at(2.00,-3.33){};
        \node[Leaf](4)at(4.00,-3.33){};
        \node[Node](1)at(1.00,0.00){};
        \node[Node](3)at(3.00,-1.67){};
        \draw[Edge](0)--(1);
        \draw[Edge](2)--(3);
        \draw[Edge](3)--(1);
        \draw[Edge](4)--(3);
        \node(r)at(1.00,2.5){};
        \draw[Edge](r)--(1);
    \end{tikzpicture}}
    \enspace + \enspace
    \scalebox{.8}{\begin{tikzpicture}[scale=.12,Centering]
        \node[Leaf](0)at(0.00,-5.25){};
        \node[Leaf](2)at(2.00,-5.25){};
        \node[Leaf](4)at(4.00,-3.50){};
        \node[Leaf](6)at(6.00,-1.75){};
        \node[Node](1)at(1.00,-3.50){};
        \node[Node](3)at(3.00,-1.75){};
        \node[Node](5)at(5.00,0.00){};
        \draw[Edge](0)--(1);
        \draw[Edge](1)--(3);
        \draw[Edge](2)--(1);
        \draw[Edge](3)--(5);
        \draw[Edge](4)--(3);
        \draw[Edge](6)--(5);
        \node(r)at(5.00,2.5){};
        \draw[Edge](r)--(5);
    \end{tikzpicture}}
    \enspace + \enspace
    2\,
    \scalebox{.8}{\begin{tikzpicture}[scale=.12,Centering]
        \node[Leaf](0)at(0.00,-4.67){};
        \node[Leaf](2)at(2.00,-4.67){};
        \node[Leaf](4)at(4.00,-4.67){};
        \node[Leaf](6)at(6.00,-4.67){};
        \node[Node](1)at(1.00,-2.33){};
        \node[Node](3)at(3.00,0.00){};
        \node[Node](5)at(5.00,-2.33){};
        \draw[Edge](0)--(1);
        \draw[Edge](1)--(3);
        \draw[Edge](2)--(1);
        \draw[Edge](4)--(5);
        \draw[Edge](5)--(3);
        \draw[Edge](6)--(5);
        \node(r)at(3.00,2.5){};
        \draw[Edge](r)--(3);
    \end{tikzpicture}}
    \enspace + \enspace
    \scalebox{.8}{\begin{tikzpicture}[scale=.12,Centering]
        \node[Leaf](0)at(0.00,-3.50){};
        \node[Leaf](2)at(2.00,-5.25){};
        \node[Leaf](4)at(4.00,-5.25){};
        \node[Leaf](6)at(6.00,-1.75){};
        \node[Node](1)at(1.00,-1.75){};
        \node[Node](3)at(3.00,-3.50){};
        \node[Node](5)at(5.00,0.00){};
        \draw[Edge](0)--(1);
        \draw[Edge](1)--(5);
        \draw[Edge](2)--(3);
        \draw[Edge](3)--(1);
        \draw[Edge](4)--(3);
        \draw[Edge](6)--(5);
        \node(r)at(5.00,2.5){};
        \draw[Edge](r)--(5);
    \end{tikzpicture}}
    \enspace + \enspace
    \scalebox{.8}{\begin{tikzpicture}[scale=.12,Centering]
        \node[Leaf](0)at(0.00,-1.75){};
        \node[Leaf](2)at(2.00,-5.25){};
        \node[Leaf](4)at(4.00,-5.25){};
        \node[Leaf](6)at(6.00,-3.50){};
        \node[Node](1)at(1.00,0.00){};
        \node[Node](3)at(3.00,-3.50){};
        \node[Node](5)at(5.00,-1.75){};
        \draw[Edge](0)--(1);
        \draw[Edge](2)--(3);
        \draw[Edge](3)--(5);
        \draw[Edge](4)--(3);
        \draw[Edge](5)--(1);
        \draw[Edge](6)--(5);
        \node(r)at(1.00,2.5){};
        \draw[Edge](r)--(1);
    \end{tikzpicture}} \\
    + \enspace
    \scalebox{.8}{\begin{tikzpicture}[scale=.12,Centering]
        \node[Leaf](0)at(0.00,-7.20){};
        \node[Leaf](2)at(2.00,-7.20){};
        \node[Leaf](4)at(4.00,-5.40){};
        \node[Leaf](6)at(6.00,-3.60){};
        \node[Leaf](8)at(8.00,-1.80){};
        \node[Node](1)at(1.00,-5.40){};
        \node[Node](3)at(3.00,-3.60){};
        \node[Node](5)at(5.00,-1.80){};
        \node[Node](7)at(7.00,0.00){};
        \draw[Edge](0)--(1);
        \draw[Edge](1)--(3);
        \draw[Edge](2)--(1);
        \draw[Edge](3)--(5);
        \draw[Edge](4)--(3);
        \draw[Edge](5)--(7);
        \draw[Edge](6)--(5);
        \draw[Edge](8)--(7);
        \node(r)at(7.00,2.5){};
        \draw[Edge](r)--(7);
    \end{tikzpicture}}
    \enspace + \enspace
    3\,
    \scalebox{.8}{\begin{tikzpicture}[scale=.12,Centering]
        \node[Leaf](0)at(0.00,-6.75){};
        \node[Leaf](2)at(2.00,-6.75){};
        \node[Leaf](4)at(4.00,-4.50){};
        \node[Leaf](6)at(6.00,-4.50){};
        \node[Leaf](8)at(8.00,-4.50){};
        \node[Node](1)at(1.00,-4.50){};
        \node[Node](3)at(3.00,-2.25){};
        \node[Node](5)at(5.00,0.00){};
        \node[Node](7)at(7.00,-2.25){};
        \draw[Edge](0)--(1);
        \draw[Edge](1)--(3);
        \draw[Edge](2)--(1);
        \draw[Edge](3)--(5);
        \draw[Edge](4)--(3);
        \draw[Edge](6)--(7);
        \draw[Edge](7)--(5);
        \draw[Edge](8)--(7);
        \node(r)at(5.00,2.5){};
        \draw[Edge](r)--(5);
    \end{tikzpicture}}
    \enspace + \enspace
    2\,
    \scalebox{.8}{\begin{tikzpicture}[scale=.12,Centering]
        \node[Leaf](0)at(0.00,-6.75){};
        \node[Leaf](2)at(2.00,-6.75){};
        \node[Leaf](4)at(4.00,-6.75){};
        \node[Leaf](6)at(6.00,-6.75){};
        \node[Leaf](8)at(8.00,-2.25){};
        \node[Node](1)at(1.00,-4.50){};
        \node[Node](3)at(3.00,-2.25){};
        \node[Node](5)at(5.00,-4.50){};
        \node[Node](7)at(7.00,0.00){};
        \draw[Edge](0)--(1);
        \draw[Edge](1)--(3);
        \draw[Edge](2)--(1);
        \draw[Edge](3)--(7);
        \draw[Edge](4)--(5);
        \draw[Edge](5)--(3);
        \draw[Edge](6)--(5);
        \draw[Edge](8)--(7);
        \node(r)at(7.00,2.5){};
        \draw[Edge](r)--(7);
    \end{tikzpicture}}
    \enspace + \enspace
    3\,
    \scalebox{.8}{\begin{tikzpicture}[scale=.12,Centering]
        \node[Leaf](0)at(0.00,-4.50){};
        \node[Leaf](2)at(2.00,-4.50){};
        \node[Leaf](4)at(4.00,-6.75){};
        \node[Leaf](6)at(6.00,-6.75){};
        \node[Leaf](8)at(8.00,-4.50){};
        \node[Node](1)at(1.00,-2.25){};
        \node[Node](3)at(3.00,0.00){};
        \node[Node](5)at(5.00,-4.50){};
        \node[Node](7)at(7.00,-2.25){};
        \draw[Edge](0)--(1);
        \draw[Edge](1)--(3);
        \draw[Edge](2)--(1);
        \draw[Edge](4)--(5);
        \draw[Edge](5)--(7);
        \draw[Edge](6)--(5);
        \draw[Edge](7)--(3);
        \draw[Edge](8)--(7);
        \node(r)at(3.00,2.5){};
        \draw[Edge](r)--(3);
    \end{tikzpicture}}
    \enspace + \enspace
    3\,
    \scalebox{.8}{\begin{tikzpicture}[scale=.12,Centering]
        \node[Leaf](0)at(0.00,-4.50){};
        \node[Leaf](2)at(2.00,-4.50){};
        \node[Leaf](4)at(4.00,-4.50){};
        \node[Leaf](6)at(6.00,-6.75){};
        \node[Leaf](8)at(8.00,-6.75){};
        \node[Node](1)at(1.00,-2.25){};
        \node[Node](3)at(3.00,0.00){};
        \node[Node](5)at(5.00,-2.25){};
        \node[Node](7)at(7.00,-4.50){};
        \draw[Edge](0)--(1);
        \draw[Edge](1)--(3);
        \draw[Edge](2)--(1);
        \draw[Edge](4)--(5);
        \draw[Edge](5)--(3);
        \draw[Edge](6)--(7);
        \draw[Edge](7)--(5);
        \draw[Edge](8)--(7);
        \node(r)at(3.00,2.5){};
        \draw[Edge](r)--(3);
    \end{tikzpicture}}
    \enspace + \enspace
    \scalebox{.8}{\begin{tikzpicture}[scale=.12,Centering]
        \node[Leaf](0)at(0.00,-5.40){};
        \node[Leaf](2)at(2.00,-7.20){};
        \node[Leaf](4)at(4.00,-7.20){};
        \node[Leaf](6)at(6.00,-3.60){};
        \node[Leaf](8)at(8.00,-1.80){};
        \node[Node](1)at(1.00,-3.60){};
        \node[Node](3)at(3.00,-5.40){};
        \node[Node](5)at(5.00,-1.80){};
        \node[Node](7)at(7.00,0.00){};
        \draw[Edge](0)--(1);
        \draw[Edge](1)--(5);
        \draw[Edge](2)--(3);
        \draw[Edge](3)--(1);
        \draw[Edge](4)--(3);
        \draw[Edge](5)--(7);
        \draw[Edge](6)--(5);
        \draw[Edge](8)--(7);
        \node(r)at(7.00,2.5){};
        \draw[Edge](r)--(7);
    \end{tikzpicture}}
    \enspace + \enspace
    3\,
    \scalebox{.8}{\begin{tikzpicture}[scale=.12,Centering]
        \node[Leaf](0)at(0.00,-4.50){};
        \node[Leaf](2)at(2.00,-6.75){};
        \node[Leaf](4)at(4.00,-6.75){};
        \node[Leaf](6)at(6.00,-4.50){};
        \node[Leaf](8)at(8.00,-4.50){};
        \node[Node](1)at(1.00,-2.25){};
        \node[Node](3)at(3.00,-4.50){};
        \node[Node](5)at(5.00,0.00){};
        \node[Node](7)at(7.00,-2.25){};
        \draw[Edge](0)--(1);
        \draw[Edge](1)--(5);
        \draw[Edge](2)--(3);
        \draw[Edge](3)--(1);
        \draw[Edge](4)--(3);
        \draw[Edge](6)--(7);
        \draw[Edge](7)--(5);
        \draw[Edge](8)--(7);
        \node(r)at(5.00,2.5){};
        \draw[Edge](r)--(5);
    \end{tikzpicture}} \\
    + \enspace
    \scalebox{.8}{\begin{tikzpicture}[scale=.12,Centering]
        \node[Leaf](0)at(0.00,-3.60){};
        \node[Leaf](2)at(2.00,-7.20){};
        \node[Leaf](4)at(4.00,-7.20){};
        \node[Leaf](6)at(6.00,-5.40){};
        \node[Leaf](8)at(8.00,-1.80){};
        \node[Node](1)at(1.00,-1.80){};
        \node[Node](3)at(3.00,-5.40){};
        \node[Node](5)at(5.00,-3.60){};
        \node[Node](7)at(7.00,0.00){};
        \draw[Edge](0)--(1);
        \draw[Edge](1)--(7);
        \draw[Edge](2)--(3);
        \draw[Edge](3)--(5);
        \draw[Edge](4)--(3);
        \draw[Edge](5)--(1);
        \draw[Edge](6)--(5);
        \draw[Edge](8)--(7);
        \node(r)at(7.00,2.5){};
        \draw[Edge](r)--(7);
    \end{tikzpicture}}
    \enspace + \enspace
    \scalebox{.8}{\begin{tikzpicture}[scale=.12,Centering]
        \node[Leaf](0)at(0.00,-3.60){};
        \node[Leaf](2)at(2.00,-5.40){};
        \node[Leaf](4)at(4.00,-7.20){};
        \node[Leaf](6)at(6.00,-7.20){};
        \node[Leaf](8)at(8.00,-1.80){};
        \node[Node](1)at(1.00,-1.80){};
        \node[Node](3)at(3.00,-3.60){};
        \node[Node](5)at(5.00,-5.40){};
        \node[Node](7)at(7.00,0.00){};
        \draw[Edge](0)--(1);
        \draw[Edge](1)--(7);
        \draw[Edge](2)--(3);
        \draw[Edge](3)--(1);
        \draw[Edge](4)--(5);
        \draw[Edge](5)--(3);
        \draw[Edge](6)--(5);
        \draw[Edge](8)--(7);
        \node(r)at(7.00,2.5){};
        \draw[Edge](r)--(7);
    \end{tikzpicture}}
    \enspace + \enspace
    \scalebox{.8}{\begin{tikzpicture}[scale=.12,Centering]
        \node[Leaf](0)at(0.00,-1.80){};
        \node[Leaf](2)at(2.00,-7.20){};
        \node[Leaf](4)at(4.00,-7.20){};
        \node[Leaf](6)at(6.00,-5.40){};
        \node[Leaf](8)at(8.00,-3.60){};
        \node[Node](1)at(1.00,0.00){};
        \node[Node](3)at(3.00,-5.40){};
        \node[Node](5)at(5.00,-3.60){};
        \node[Node](7)at(7.00,-1.80){};
        \draw[Edge](0)--(1);
        \draw[Edge](2)--(3);
        \draw[Edge](3)--(5);
        \draw[Edge](4)--(3);
        \draw[Edge](5)--(7);
        \draw[Edge](6)--(5);
        \draw[Edge](7)--(1);
        \draw[Edge](8)--(7);
        \node(r)at(1.00,2.5){};
        \draw[Edge](r)--(1);
    \end{tikzpicture}}
    \enspace + \enspace
    2\,
    \scalebox{.8}{\begin{tikzpicture}[scale=.12,Centering]
        \node[Leaf](0)at(0.00,-2.25){};
        \node[Leaf](2)at(2.00,-6.75){};
        \node[Leaf](4)at(4.00,-6.75){};
        \node[Leaf](6)at(6.00,-6.75){};
        \node[Leaf](8)at(8.00,-6.75){};
        \node[Node](1)at(1.00,0.00){};
        \node[Node](3)at(3.00,-4.50){};
        \node[Node](5)at(5.00,-2.25){};
        \node[Node](7)at(7.00,-4.50){};
        \draw[Edge](0)--(1);
        \draw[Edge](2)--(3);
        \draw[Edge](3)--(5);
        \draw[Edge](4)--(3);
        \draw[Edge](5)--(1);
        \draw[Edge](6)--(7);
        \draw[Edge](7)--(5);
        \draw[Edge](8)--(7);
        \node(r)at(1.00,2.5){};
        \draw[Edge](r)--(1);
    \end{tikzpicture}}
    \enspace + \enspace
    \scalebox{.8}{\begin{tikzpicture}[scale=.12,Centering]
        \node[Leaf](0)at(0.00,-1.80){};
        \node[Leaf](2)at(2.00,-5.40){};
        \node[Leaf](4)at(4.00,-7.20){};
        \node[Leaf](6)at(6.00,-7.20){};
        \node[Leaf](8)at(8.00,-3.60){};
        \node[Node](1)at(1.00,0.00){};
        \node[Node](3)at(3.00,-3.60){};
        \node[Node](5)at(5.00,-5.40){};
        \node[Node](7)at(7.00,-1.80){};
        \draw[Edge](0)--(1);
        \draw[Edge](2)--(3);
        \draw[Edge](3)--(7);
        \draw[Edge](4)--(5);
        \draw[Edge](5)--(3);
        \draw[Edge](6)--(5);
        \draw[Edge](7)--(1);
        \draw[Edge](8)--(7);
        \node(r)at(1.00,2.5){};
        \draw[Edge](r)--(1);
    \end{tikzpicture}}
    \enspace + \enspace
    \scalebox{.8}{\begin{tikzpicture}[scale=.12,Centering]
        \node[Leaf](0)at(0.00,-1.80){};
        \node[Leaf](2)at(2.00,-3.60){};
        \node[Leaf](4)at(4.00,-7.20){};
        \node[Leaf](6)at(6.00,-7.20){};
        \node[Leaf](8)at(8.00,-5.40){};
        \node[Node](1)at(1.00,0.00){};
        \node[Node](3)at(3.00,-1.80){};
        \node[Node](5)at(5.00,-5.40){};
        \node[Node](7)at(7.00,-3.60){};
        \draw[Edge](0)--(1);
        \draw[Edge](2)--(3);
        \draw[Edge](3)--(1);
        \draw[Edge](4)--(5);
        \draw[Edge](5)--(7);
        \draw[Edge](6)--(5);
        \draw[Edge](7)--(3);
        \draw[Edge](8)--(7);
        \node(r)at(1.00,2.5){};
        \draw[Edge](r)--(1);
    \end{tikzpicture}}
    \enspace + \enspace
    \scalebox{.8}{\begin{tikzpicture}[scale=.12,Centering]
        \node[Leaf](0)at(0.00,-1.80){};
        \node[Leaf](2)at(2.00,-3.60){};
        \node[Leaf](4)at(4.00,-5.40){};
        \node[Leaf](6)at(6.00,-7.20){};
        \node[Leaf](8)at(8.00,-7.20){};
        \node[Node](1)at(1.00,0.00){};
        \node[Node](3)at(3.00,-1.80){};
        \node[Node](5)at(5.00,-3.60){};
        \node[Node](7)at(7.00,-5.40){};
        \draw[Edge](0)--(1);
        \draw[Edge](2)--(3);
        \draw[Edge](3)--(1);
        \draw[Edge](4)--(5);
        \draw[Edge](5)--(3);
        \draw[Edge](6)--(7);
        \draw[Edge](7)--(5);
        \draw[Edge](8)--(7);
        \node(r)at(1.00,2.5){};
        \draw[Edge](r)--(1);
    \end{tikzpicture}}
    \enspace + \cdots.
\end{multline}
Theorem~\ref{thm:hook_series} implies that for any binary tree $\Tfr$,
the coefficient $\Angle{\Tfr, \Hook\Par{\HookSystem_{\Mag, G}}}$ can be
obtained by the usual hook-length formula of binary trees. This explains
the name of hook generating series for $\Hook(\Bca)$, when $\Bca$ is a
bud generating system. Alternatively, the coefficient
$\Angle{\Tfr, \Hook\Par{\HookSystem_{\Mag, G}}}$ is the cardinal of
the sylvester class~\cite{HNT05} of permutations encoded by~$\Tfr$.
\medbreak

\subsubsection{Hook coefficients for words of $\Dias_\gamma$}%
\label{subsubsec:example_hook_coeff_dias}
Let us consider the monochrome bud generating system $\BDias{\gamma}$
and its set of rules $\Rfr_\gamma$ introduced in
Section~\ref{subsubsec:bud_generating_system_dias_gamma}. Since, by
Proposition~\ref{prop:properties_BDias}, $\Rfr_\gamma$ generates
$\Dias_\gamma$, $\BDias{\gamma}$ is a
hook bud generating system $\HookSystem_{\Dias_\gamma, \Rfr_\gamma}$
(see Section~\ref{subsubsec:hook_generating_series}). This leads to
the definition of a statistic on the words of $\Dias_\gamma$, provided
by the coefficients of the hook generating series
$\Hook\Par{\HookSystem_{\Dias_\gamma, \Rfr_\gamma}}$ of
$\HookSystem_{\Dias_\gamma, \Rfr_\gamma}$ which begins, when
$\gamma = 1$, by
\begin{small}
\begin{multline}
    \Hook\Par{\HookSystem_{\Dias_1, \Rfr_1}} =
    (0)
    \enspace + \enspace (01)
    \enspace + \enspace (10)
    \enspace + \enspace 3\, (011)
    \enspace + \enspace 2\, (101)
    \enspace + \enspace 3\, (110)
    \enspace + \enspace 15\, (0111)
    \enspace + \enspace 9\, (1011) \\
    + \enspace 9\, (1101)
    \enspace + \enspace 15\, (1110)
    \enspace + \enspace 105\, (01111)
    \enspace + \enspace 60\, (10111)
    \enspace + \enspace 54\, (11011)
    \enspace + \enspace 60\, (11101)
    \enspace + \enspace 105\, (11110) \\
    + \enspace 945\, (011111)
    \enspace + \enspace 525\, (101111)
    \enspace + \enspace 45 \, (110111)
    \enspace + \enspace 450\, (111011)
    \enspace + \enspace 525\, (111101)
    \enspace + \enspace 945\, (111110)
    \enspace + \cdots.
\end{multline}
\end{small}
\medbreak

Let us set, for all $0 \leq a \leq n - 1$,
\begin{math}
    h_{n, a} :=
    \Angle{1^a 0 1^{n - a - 1},
    \Hook\Par{\HookSystem_{\Dias_1, \Rfr_1}}}.
\end{math}
By Lemmas~\ref{lem:initial_terminal_composition}
and~\ref{lem:pre_lie_star_hook_series}, the $h_{n, a}$ satisfy the
recurrence
\begin{equation}
    h_{n, a} =
    \begin{cases}
        1 & \mbox{if } n = 1, \\
        (2a - 1) h_{n - 1, a - 1} & \mbox{if } a = n - 1, \\
        (2n - 2a - 3) h_{n - 1, a} & \mbox{if } a = 0, \\
        (2a - 1) h_{n - 1, a - 1} + (2n - 2a - 3) h_{n - 1, a}
            & \mbox{otherwise}.
    \end{cases}
\end{equation}
The numbers $h_{n, a}$ form a triangle beginning by
\begin{equation} \renewcommand{\arraystretch}{1.1}
    \begin{tabular}{llllllll}
        1 \\
        1 & 1 \\
        3 & 2 & 3 \\
        15 & 9 & 9 & 15 \\
        105 & 60 & 54 & 60 & 105 \\
        945 & 525 & 450 & 450 & 525 & 945 \\
        10395 & 5670 & 4725 & 4500 & 4725 & 5670 & 10395 \\
        135135 & 72765 & 59535 & 55125 & 55125 & 59535 & 72765 & 135135
    \end{tabular}.
\end{equation}
These numbers form Sequence~\OEIS{A059366} of~\cite{Slo}.
\medbreak

\subsubsection{Hook coefficients for Motkzin paths}%
\label{subsubsec:example_hook_coeff_motz}
It is proven in~\cite{Gir15} that
$G := \Bra{\MotzHoriz, \MotzPeak}$ is a generating set of
$\Motz$. Hence, $\HookSystem_{\Motz, G}$ is a hook generating system.
This leads to the definition of a statistic on Motzkin paths, provided
by the coefficients of the hook generating series
$\Hook\Par{\HookSystem_{\Motz, G}}$ of $\HookSystem_{\Motz, G}$
which begins by
\begin{multline}
    \Hook\Par{\HookSystem_{\Motz, G}} =
    \begin{tikzpicture}[scale=.25,Centering]
        \draw[PathGrid](0,0) grid (0,0);
        \node[PathNode](1)at(0,0){};
    \end{tikzpicture}
    \enspace + \enspace
    \begin{tikzpicture}[scale=.25,Centering]
        \draw[PathGrid](0,0) grid (1,0);
        \node[PathNode](1)at(0,0){};
        \node[PathNode](2)at(1,0){};
        \draw[PathStep](1)--(2);
    \end{tikzpicture}
    \enspace + \enspace
    2\,
    \begin{tikzpicture}[scale=.25,Centering]
        \draw[PathGrid](0,0) grid (2,0);
        \node[PathNode](1)at(0,0){};
        \node[PathNode](2)at(1,0){};
        \node[PathNode](3)at(2,0){};
        \draw[PathStep](1)--(2);
        \draw[PathStep](2)--(3);
    \end{tikzpicture}
    \enspace + \enspace
    \begin{tikzpicture}[scale=.25,Centering]
        \draw[PathGrid](0,0) grid (2,1);
        \node[PathNode](1)at(0,0){};
        \node[PathNode](2)at(1,1){};
        \node[PathNode](3)at(2,0){};
        \draw[PathStep](1)--(2);
        \draw[PathStep](2)--(3);
    \end{tikzpicture}
    \enspace + \enspace
    6\,
    \begin{tikzpicture}[scale=.25,Centering]
        \draw[PathGrid](0,0) grid (3,0);
        \node[PathNode](1)at(0,0){};
        \node[PathNode](2)at(1,0){};
        \node[PathNode](3)at(2,0){};
        \node[PathNode](4)at(3,0){};
        \draw[PathStep](1)--(2);
        \draw[PathStep](2)--(3);
        \draw[PathStep](3)--(4);
    \end{tikzpicture}
    \enspace + \enspace
    2\,
    \begin{tikzpicture}[scale=.25,Centering]
        \draw[PathGrid](0,0) grid (3,1);
        \node[PathNode](1)at(0,0){};
        \node[PathNode](2)at(1,0){};
        \node[PathNode](3)at(2,1){};
        \node[PathNode](4)at(3,0){};
        \draw[PathStep](1)--(2);
        \draw[PathStep](2)--(3);
        \draw[PathStep](3)--(4);
    \end{tikzpicture}
    \enspace + \enspace
    2\,
    \begin{tikzpicture}[scale=.25,Centering]
        \draw[PathGrid](0,0) grid (3,1);
        \node[PathNode](1)at(0,0){};
        \node[PathNode](2)at(1,1){};
        \node[PathNode](3)at(2,0){};
        \node[PathNode](4)at(3,0){};
        \draw[PathStep](1)--(2);
        \draw[PathStep](2)--(3);
        \draw[PathStep](3)--(4);
    \end{tikzpicture}
    \\
    + \enspace
    \begin{tikzpicture}[scale=.25,Centering]
        \draw[PathGrid](0,0) grid (3,1);
        \node[PathNode](1)at(0,0){};
        \node[PathNode](2)at(1,1){};
        \node[PathNode](3)at(2,1){};
        \node[PathNode](4)at(3,0){};
        \draw[PathStep](1)--(2);
        \draw[PathStep](2)--(3);
        \draw[PathStep](3)--(4);
    \end{tikzpicture}
    \enspace + \enspace
    24\,
    \begin{tikzpicture}[scale=.25,Centering]
        \draw[PathGrid](0,0) grid (4,0);
        \node[PathNode](1)at(0,0){};
        \node[PathNode](2)at(1,0){};
        \node[PathNode](3)at(2,0){};
        \node[PathNode](4)at(3,0){};
        \node[PathNode](5)at(4,0){};
        \draw[PathStep](1)--(2);
        \draw[PathStep](2)--(3);
        \draw[PathStep](3)--(4);
        \draw[PathStep](4)--(5);
    \end{tikzpicture}
    \enspace + \enspace
    6\,
    \begin{tikzpicture}[scale=.25,Centering]
        \draw[PathGrid](0,0) grid (4,1);
        \node[PathNode](1)at(0,0){};
        \node[PathNode](2)at(1,0){};
        \node[PathNode](3)at(2,0){};
        \node[PathNode](4)at(3,1){};
        \node[PathNode](5)at(4,0){};
        \draw[PathStep](1)--(2);
        \draw[PathStep](2)--(3);
        \draw[PathStep](3)--(4);
        \draw[PathStep](4)--(5);
    \end{tikzpicture}
    \enspace + \enspace
    6\,
    \begin{tikzpicture}[scale=.25,Centering]
        \draw[PathGrid](0,0) grid (4,1);
        \node[PathNode](1)at(0,0){};
        \node[PathNode](2)at(1,0){};
        \node[PathNode](3)at(2,1){};
        \node[PathNode](4)at(3,0){};
        \node[PathNode](5)at(4,0){};
        \draw[PathStep](1)--(2);
        \draw[PathStep](2)--(3);
        \draw[PathStep](3)--(4);
        \draw[PathStep](4)--(5);
    \end{tikzpicture}
    \enspace + \enspace
    3\,
    \begin{tikzpicture}[scale=.25,Centering]
        \draw[PathGrid](0,0) grid (4,1);
        \node[PathNode](1)at(0,0){};
        \node[PathNode](2)at(1,0){};
        \node[PathNode](3)at(2,1){};
        \node[PathNode](4)at(3,1){};
        \node[PathNode](5)at(4,0){};
        \draw[PathStep](1)--(2);
        \draw[PathStep](2)--(3);
        \draw[PathStep](3)--(4);
        \draw[PathStep](4)--(5);
    \end{tikzpicture}
    \enspace + \enspace
    6\,
    \begin{tikzpicture}[scale=.25,Centering]
        \draw[PathGrid](0,0) grid (4,1);
        \node[PathNode](1)at(0,0){};
        \node[PathNode](2)at(1,1){};
        \node[PathNode](3)at(2,0){};
        \node[PathNode](4)at(3,0){};
        \node[PathNode](5)at(4,0){};
        \draw[PathStep](1)--(2);
        \draw[PathStep](2)--(3);
        \draw[PathStep](3)--(4);
        \draw[PathStep](4)--(5);
    \end{tikzpicture}
    \\
    + \enspace
    2\,
    \begin{tikzpicture}[scale=.25,Centering]
        \draw[PathGrid](0,0) grid (4,1);
        \node[PathNode](1)at(0,0){};
        \node[PathNode](2)at(1,1){};
        \node[PathNode](3)at(2,0){};
        \node[PathNode](4)at(3,1){};
        \node[PathNode](5)at(4,0){};
        \draw[PathStep](1)--(2);
        \draw[PathStep](2)--(3);
        \draw[PathStep](3)--(4);
        \draw[PathStep](4)--(5);
    \end{tikzpicture}
    \enspace + \enspace
    3\,
    \begin{tikzpicture}[scale=.25,Centering]
        \draw[PathGrid](0,0) grid (4,1);
        \node[PathNode](1)at(0,0){};
        \node[PathNode](2)at(1,1){};
        \node[PathNode](3)at(2,1){};
        \node[PathNode](4)at(3,0){};
        \node[PathNode](5)at(4,0){};
        \draw[PathStep](1)--(2);
        \draw[PathStep](2)--(3);
        \draw[PathStep](3)--(4);
        \draw[PathStep](4)--(5);
    \end{tikzpicture}
    \enspace + \enspace
    2\,
    \begin{tikzpicture}[scale=.25,Centering]
        \draw[PathGrid](0,0) grid (4,1);
        \node[PathNode](1)at(0,0){};
        \node[PathNode](2)at(1,1){};
        \node[PathNode](3)at(2,1){};
        \node[PathNode](4)at(3,1){};
        \node[PathNode](5)at(4,0){};
        \draw[PathStep](1)--(2);
        \draw[PathStep](2)--(3);
        \draw[PathStep](3)--(4);
        \draw[PathStep](4)--(5);
    \end{tikzpicture}
    \enspace + \enspace
    \begin{tikzpicture}[scale=.25,Centering]
        \draw[PathGrid](0,0) grid (4,2);
        \node[PathNode](1)at(0,0){};
        \node[PathNode](2)at(1,1){};
        \node[PathNode](3)at(2,2){};
        \node[PathNode](4)at(3,1){};
        \node[PathNode](5)at(4,0){};
        \draw[PathStep](1)--(2);
        \draw[PathStep](2)--(3);
        \draw[PathStep](3)--(4);
        \draw[PathStep](4)--(5);
    \end{tikzpicture}
    \enspace + \cdots.
\end{multline}
\medbreak

\subsubsection{Generating series of some unary-binary trees}%
\label{subsubsec:colt_synt_bubtree}
Let us consider the bud generating system $\BUBTree$ introduced in
Section~\ref{subsubsec:example_bubtree}. We have, for all
$a \in \{1, 2\}$ and $\alpha \in \Types_{\{1, 2\}}$,
\begin{equation}
    \MultOutIn_{a, \alpha} =
    \begin{cases}
        2 & \mbox{if } (a, \alpha) = (1, 01), \\
        1 & \mbox{if } (a, \alpha) = (2, 20), \\
        0 & \mbox{otherwise},
    \end{cases}
\end{equation}
and
\begin{subequations}
\begin{multicols}{2}
    \begin{equation}
        \Gbf_1\Par{\VarY_1, \VarY_2} = 2 \VarY_2,
    \end{equation}

    \begin{equation}
        \Gbf_2\Par{\VarY_1, \VarY_2} = \VarY_1^2.
    \end{equation}
\end{multicols}
\end{subequations}
\medbreak

Since, by Proposition~\ref{prop:properties_BUBTree}, $\BUBTree$
satisfies the conditions of
Proposition~\ref{prop:functional_equation_synt}, by this last
proposition
and~\eqref{equ:generating_series_synt_functional_equation}, the
generating series $\GenS_{\Lang(\BUBTree)}$ of $\Lang(\BUBTree)$
satisfies $\GenS_{\Lang(\BUBTree)} = \Fbf_1(0, t)$ where
\begin{subequations}
\begin{multicols}{2}
    \begin{equation}
        \Fbf_1\Par{\VarY_1, \VarY_2}
        = \VarY_1 + 2 \Fbf_2\Par{\VarY_1, \VarY_2},
    \end{equation}

    \begin{equation}
        \Fbf_2\Par{\VarY_1, \VarY_2}
        = \VarY_2 + \Fbf_1\Par{\VarY_1, \VarY_2}^2.
    \end{equation}
\end{multicols}
\end{subequations}
We obtain that $\Fbf_1(\VarY_1, \VarY_2)$ satisfies the functional
equation
\begin{equation}
    \VarY_1 + 2 \VarY_2 - \Fbf_1\Par{\VarY_1, \VarY_2}
    + 2 \Fbf_1\Par{\VarY_1, \VarY_2}^2 = 0.
\end{equation}
Hence, $\GenS_{\Lang\Par{\BUBTree}}$ satisfies
\begin{equation}
    2t - \GenS_{\Lang\Par{\BUBTree}}
    + 2 \GenS_{\Lang\Par{\BUBTree}}^2 = 0,
\end{equation}
showing that the elements of $\Lang\Par{\BUBTree}$ are enumerated, arity
by arity, by Sequence \OEIS{A052707} of~\cite{Slo} starting by
\begin{equation}
    2, 8, 64, 640, 7168, 86016, 1081344, 14057472, 187432960,
    2549088256.
\end{equation}
\medbreak

Besides, since by Proposition~\ref{prop:properties_BUBTree}, $\BUBTree$
satisfies the conditions of Theorem~\ref{thm:series_color_types_synt},
by this last theorem and~\eqref{equ:generating_series_synt_recurrence},
$\GenS_{\Lang(\BUBTree)}$ satisfies, for any $n \geq 1$,
\begin{equation}
    \Angle{t^n, \GenS_{\Lang\Par{\BUBTree}}} =
    \Angle{\VarX_1 \VarY_2^n, \Fbf},
\end{equation}
where $\Fbf$ is the series satisfying, for any $a \in \Ccr$ and any type
$\alpha \in \Types_{\{1, 2\}}$, the recursive formula
\begin{multline}
    \Angle{\VarX_a \VarY^{\alpha_1} \VarY^{\alpha_2}, \Fbf}
    = \delta_{\alpha, \Type(a)}
    + \delta_{a, 1}
    2 \Angle{\VarX_2 \VarY_1^{\alpha_1} \VarY_2^{\alpha_2}, \Fbf} \\
    + \delta_{a, 2}
    \sum_{\substack{
        d_1, d_2, d_3, d_4 \in \N \\
        \alpha_1 = d_1 + d_2 \\
        \alpha_2 = d_3 + d_4 \\
        (d_1, d_3) \ne (0, 0) \ne (d_2, d_4)
    }}
    \Angle{\VarX_1 \VarY_1^{d_1} \VarY_2^{d_3}, \Fbf}
    \Angle{\VarX_2 \VarY_1^{d_2} \VarY_2^{d_4}, \Fbf}.
\end{multline}
\medbreak

\subsubsection{Generating series of $B$-perfect trees}%
\label{subsubsec:colt_sync_bbtree}
Let us consider the monochrome bud generating system $\BBTree{B}$ and
its set of rules $\Rfr_B$ introduced in
Section~\ref{subsubsec:example_bbtree}. By
Proposition~\ref{prop:properties_BBTree}, the generating series
$\GenS_{\SyncLang(\BBTree{B})}$ is well-defined when $1 \notin B$. For
this reason, in all this section we restrict ourselves to the case where
all elements of $B$ are greater than or equal to~$2$. To maintain here
homogeneous notations with the rest of the text, we consider that the
set of colors $\Ccr$ of $\BBTree{B}$ is the singleton $\{1\}$. We have,
for all $\alpha \in \Types_{\{1\}}$,
\begin{equation}
    \MultOutIn_{1, \alpha} =
    \begin{cases}
        1 & \mbox{if } (a, \alpha) = (1, b) \mbox{ with } b \in B, \\
        0 & \mbox{otherwise},
    \end{cases}
\end{equation}
and
\begin{equation}
    \Gbf_1(\VarY_1) = \sum_{b \in B} \VarY_1^b.
\end{equation}
\medbreak

Since by Proposition~\ref{prop:properties_BBTree}, $\BBTree{B}$
satisfies the conditions of
Proposition~\ref{prop:functional_equation_sync}, by this last
proposition and~\eqref{equ:generating_series_sync_functional_equation},
the generating series $\GenS_{\SyncLang(\BBTree{B})}$ of
$\SyncLang\Par{\BBTree{B}}$ satisfies
$\GenS_{\SyncLang\Par{\BBTree{B}}} = \Fbf_1(t)$ where
\begin{equation}
    \Fbf_1\Par{\VarY_1} =
    \VarY_1 + \Fbf_1\Par{\sum_{b \in B} \VarY_1^b}.
\end{equation}
This functional equation for the generating series of $B$-perfect trees,
in the case where $B = \{2, 3\}$, is the one obtained
in~\cite{Odl82,FS09,Gir12} by different methods.
\medbreak

Besides, since by Proposition~\ref{prop:properties_BBTree}, $\BBTree{B}$
satisfies the conditions of Theorem~\ref{thm:series_color_types_sync},
by this last theorem and~\eqref{equ:generating_series_sync_recurrence},
$\GenS_{\SyncLang\Par{\BBTree{B}}}$ satisfies, for any $n \geq 1$, the
recursive formula
\begin{equation} \label{equ:recursive_enumeration_b_perfect_trees}
    \Angle{t^n, \GenS_{\SyncLang(\BBTree{B})}} =
    \delta_{n, 1} +
    \sum_{\substack{
        d_b \in \N, b \in B \\
        n = \sum_{b \in B} b\, d_b
    }}
    \lbag d_b : b \in B \rbag !
    \Angle{t^{\sum_{b \in B} d_b}, \GenS_{\SyncLang(\BBTree{B})}}
\end{equation}
For instance, for $B := \{2, 3\}$, one has
\begin{equation}
    \Angle{t^n, \GenS_{\SyncLang(\BBTree{\{2, 3\}})}} =
    \delta_{n, 1}
    + \sum_{\substack{
        d_2, d_3 \geq 0 \\
        n = 2 d_2 + 3 d_3
    }}
    \binom{d_2 + d_3}{d_2}
    \Angle{t^{d_2 + d_3}, \GenS_{\SyncLang(\BBTree{\{2, 3\}})}},
\end{equation}
which is a recursive formula to enumerate the $\{2, 3\}$-perfect
trees known from~\cite{MPRS79}, and for $B := \{2, 3, 4\}$,
\begin{equation}
    \Angle{t^n, \GenS_{\SyncLang(\BBTree{\{2, 3, 4\}})}} =
    \delta_{n, 1}
    + \sum_{\substack{
        d_2, d_3, d_4 \geq 0 \\
        n = 2 d_2 + 3 d_3 + 4 d_4
    }}
    \lbag d_2, d_3, d_4 \rbag !
    \Angle{t^{d_2 + d_3 + d_4}, \GenS_{\SyncLang(\BBTree{\{2, 3, 4\}})}}.
\end{equation}
\medbreak

Moreover, it is possible to refine the enumeration of $B$-perfect trees
to take into account of the number of internal nodes with a given arity
in the trees. For this, we consider the series $\Sbf_q$ satisfying the
recurrence
\begin{equation}
    \Angle{t^n, \Sbf_q} =
    \delta_{n, 1} +
    \sum_{\substack{
        d_b \in \N, b \geq 2 \\
        n = \sum_{b \geq 2} b\, d_b
    }}
    \lbag d_b : b \geq 2 \rbag !
    \Par{\prod_{b \geq 2} q_b^{d_b}}
    \Angle{t^{\sum_{b \geq 2} d_b}, \Sbf_q}.
\end{equation}
The coefficient of $\Par{\prod_{b \geq 2} q_b^{d_b}}t^n$ in
$\Sbf_q$ is the number of $\N \setminus \{0, 1\}$-perfect trees with $n$
leaves and with $d_b$ internal nodes of arity $b$ for all $b \geq 2$.
The specialization of $\Sbf_q$ at $q_b := 0$ for all $b \notin B$ and
$q_b := t$ for all $b \in B$ is equal to the
series~$\GenS_{\SyncLang(\BBTree{B})}$.
\medbreak

First coefficients of $\Sbf_q$ are
\begin{subequations} \small
\begin{multicols}{2}
\begin{equation}
    \Angle{t, \Sbf_q} = 1,
\end{equation}
\begin{equation}
    \Angle{t^2, \Sbf_q} = q_2,
\end{equation}
\begin{equation}
    \Angle{t^3, \Sbf_q} = q_3,
\end{equation}

\begin{equation}
    \Angle{t^4, \Sbf_q} = q_2^3 + q_4,
\end{equation}
\begin{equation}
    \Angle{t^5, \Sbf_q} = 2 q_2^2 q_3 + q_5,
\end{equation}
\begin{equation}
    \Angle{t^6, \Sbf_q} = q_2^3 q_3 + q_2 q_3^2 + 2 q_2^2 q_4 + q_6,
\end{equation}
\end{multicols}
\begin{equation}
    \Angle{t^7, \Sbf_q} =
    3 q_2^2 q_3^2 + 2 q_2 q_3 q_4 + 2 q_2^2 q_5 + q_7,
\end{equation}
\begin{equation}
    \Angle{t^8, \Sbf_q} =
    q_2^7 + q_2^4 q_4 + 3 q_2 q_3^3 + 3 q_2^2 q_3 q_4 +
    q_2 q_4^2 + 2 q_2 q_3 q_5 + 2 q_2^2 q_6 + q_8,
\end{equation}
\begin{equation}
    \Angle{t^9, \Sbf_q} =
    4 q_2^6 q_3 + 4 q_2^3 q_3 q_4 + q_3^4 + 6 q_2 q_3^2 q_4 +
    3 q_2^2 q_3 q_5 + 2 q_2 q_4 q_5 + 2 q_2 q_3 q_6 + 2 q_2^2 q_7 + q_9.
\end{equation}
\end{subequations}
\medbreak

\subsubsection{Generating series of balanced binary trees}%
\label{subsubsec:colt_sync_bbaltree}
Let us consider the bud generating system $\BBalTree$ introduced in
Section~\ref{subsubsec:example_bbaltree}. We have
\begin{equation}
    \MultOutIn_{a, \alpha} =
    \begin{cases}
        1 & \mbox{if } (a, \alpha) = (1, 20), \\
        2 & \mbox{if } (a, \alpha) = (1, 11), \\
        1 & \mbox{if } (a, \alpha) = (2, 10), \\
        0 & \mbox{otherwise},
    \end{cases}
\end{equation}
and
\begin{subequations}
\begin{multicols}{2}
    \begin{equation}
        \Gbf_1(\VarY_1, \VarY_2) = \VarY_1^2 + 2 \VarY_1 \VarY_2,
    \end{equation}

    \begin{equation}
        \Gbf_2(\VarY_1, \VarY_2) = \VarY_1.
    \end{equation}
\end{multicols}
\end{subequations}
\medbreak

Since by Proposition~\ref{prop:properties_BBalTree}, $\BBalTree$
satisfies the conditions of
Proposition~\ref{prop:functional_equation_sync}, by this last
proposition and~\eqref{equ:generating_series_sync_functional_equation},
the generating series $\GenS_{\SyncLang\Par{\BBalTree}}$ of
$\SyncLang\Par{\BBalTree}$ satisfies
$\GenS_{\SyncLang\Par{\BBalTree}} = \Fbf_1(t, 0)$ where
\begin{equation}
    \Fbf_1\Par{\VarY_1, \VarY_2} =
    \VarY_1 + \Fbf_1\Par{\VarY_1^2 + 2 \VarY_1 \VarY_2, \VarY_1}.
\end{equation}
This functional equation for the generating series of balanced binary
trees is the one obtained in~\cite{BLL88,BLL97,Knu98,Gir12} by different
methods. As announced in
Section~\ref{subsubsec:synchronous_generating_series}, the
coefficients of $\Fbf_1$ (and hence, those of
$\GenS_{\SyncLang\Par{\BBalTree}}$) can be computed by iteration. This
consists in defining, for any $\ell \geq 0$, the polynomials
$\Fbf_1^{(\ell)}\Par{\VarY_1, \VarY_2}$ as
\begin{equation}
    \label{equ:iteration_functional_equation_balanced_binary_trees}
    \Fbf_1^{(\ell)}\Par{\VarY_1, \VarY_2} :=
    \begin{cases}
        \VarY_1 & \mbox{if } \ell = 0, \\
        \VarY_1
        + \Fbf_1^{(\ell - 1)}\Par{\VarY_1^2 + 2 \VarY_1 \VarY_2,
        \VarY_1} & \mbox{otherwise}.
    \end{cases}
\end{equation}
Since
\begin{equation}
    \Fbf_1\Par{\VarY_1, \VarY_2} =
    \lim_{\ell \to \infty} \Fbf_1^{(\ell)}\Par{\VarY_1, \VarY_2},
\end{equation}
Equation~\eqref{equ:iteration_functional_equation_balanced_binary_trees}
provides a way to compute the coefficients of
$\Fbf_1\Par{\VarY_1, \VarY_2}$. First polynomials
$\Fbf_1^{(\ell)}\Par{\VarY_1, \VarY_2}$ are
\begin{subequations} \small
\begin{multicols}{2}
\begin{equation}
    \Fbf_1^{(0)}\Par{\VarY_1, \VarY_2} = \VarY_1,
\end{equation}

\begin{equation}
    \Fbf_1^{(1)}\Par{\VarY_1, \VarY_2} = \VarY_1 + \VarY_1^2
    + 2\VarY_1\VarY_2,
\end{equation}
\end{multicols}
\begin{equation}
    \Fbf_1^{(2)}\Par{\VarY_1, \VarY_2} = \VarY_1 + \VarY_1^2
    + 2\VarY_1\VarY_2 + 2\VarY_1^3 + 4\VarY_1^2\VarY_2 + \VarY_1^4
    + 4\VarY_1^3\VarY_2 + 4\VarY_1^2\VarY_2^2,
\end{equation}
\begin{multline}
    \Fbf_1^{(3)}\Par{\VarY_1, \VarY_2} = \VarY_1 + \VarY_1^2
    + 2\VarY_1\VarY_2 + 2\VarY_1^3 + 4\VarY_1^2\VarY_2 + \VarY_1^4
    + 4\VarY_1^3\VarY_2 + 4\VarY_1^2\VarY_2^2 + 4\VarY_1^5 \\
    + 16\VarY_1^4\VarY_2
    + 16\VarY_1^3\VarY_2^2
    + 6\VarY_1^6 + 28\VarY_1^5\VarY_2 + 40\VarY_1^4\VarY_2^2
    + 16\VarY_1^3\VarY_2^3 + 4\VarY_1^7 + 24\VarY_1^6\VarY_2 \\
    + 48\VarY_1^5\VarY_2^2
    + 32\VarY_1^4\VarY_2^3 + \VarY_1^8
    + 8\VarY_1^7\VarY_2
    + 24\VarY_1^6\VarY_2^2 + 32\VarY_1^5\VarY_2^3
    + 16\VarY_1^4\VarY_2^4.
\end{multline}
\end{subequations}
\medbreak

Besides, since by Proposition~\ref{prop:properties_BBalTree},
$\BBalTree$ satisfies the conditions of
Theorem~\ref{thm:series_color_types_sync}, by this last theorem
and~\eqref{equ:generating_series_sync_recurrence},
$\GenS_{\SyncLang(\BBalTree)}$ satisfies, for any $n \geq 1$,
\begin{equation}
    \Angle{t^n, \GenS_{\SyncLang(\BBalTree)}}
    = \Angle{\VarY_1^n \VarY_2^0, \Fbf},
\end{equation}
where $\Fbf$ is the series satisfying, for any type
$\alpha \in \Types_{\{1, 2\}}$, the recursive formula
\begin{equation}
    \label{equ:recursive_enumeration_balanced_binary_trees_raw}
    \Angle{\VarY_1^{\alpha_1} \VarY_2^{\alpha_2}, \Fbf}
    = \delta_{\alpha, (1, 0)}
    +
    \sum_{\substack{
        d_1, d_2, d_3 \in \N \\
        \alpha_1 = 2 d_1 + d_2 + d_3 \\
        \alpha_2 = d_2
    }}
    \binom{d_1 + \alpha_2}{d_1}
    2^{d_2}
    \Angle{\VarY_1^{d_1 + d_2} \VarY_2^{d_3}, \Fbf}.
\end{equation}
\medbreak

\section*{Conclusion and perspectives}
In this paper, we have presented a framework for the generation of
combinatorial objects by using colored operads. The described devices
for combinatorial generation, called bud generating systems, are
generalizations of context-free grammars~\cite{Har78,HMU06} generating
words, of regular tree grammars~\cite{GS84,CDGJLLTT07} generating planar
rooted trees, and of synchronous grammars~\cite{Gir12} generating some
treelike structures. We have provided tools to enumerate the objects of
the languages of bud generating systems or to define new statistics on
these by using formal power series on colored operads and several
products on these. There are many ways to extend this work. Here follow
some few further research directions.
\smallbreak

First, the notion of rationality and recognizability in usual formal
power series~\cite{Sch61,Sch63,Eil74,BR88}, in series on
monoids~\cite{Sak09}, and in series of trees~\cite{BR82} are
fundamental. For instance, a series $\Sbf \in \K \AAngle{\Mca}$ on a
monoid $\Mca$ is rational if it belongs to the closure of the set
$\K \Angle{\Mca}$ of polynomials on $\Mca$ with respect to the
addition, the multiplication, and the Kleene star operations.
Equivalently, $\Sbf$ is rational if there exists a $\K$-weighted
automaton accepting it. The equivalence between these two properties
for the rationality property is remarkable. We ask here for the
definition of an analogous and consistent notion of rationality for
series on a colored operad $\Cca$. By consistent, we mean a property
of rationality for $\Cca$-series which can be defined both by a
closure property of the set $\K \Angle{\Cca}$ of the polynomials on
$\Cca$ with respect to some operations, and, at the same time, by an
acceptance property involving a notion of a $\K$-weighted automaton
on~$\Cca$. The analogous question about the definition of a notion of
recognizable series on colored operads also seems worth investigating.
\smallbreak

A second research direction fits mostly in the contexts of computer
science and compression theory. A \Def{straight-line grammar} (see for
instance~\cite{ZL78,SS82,Ryt04}) is a context-free grammar with a
singleton as language. There exists also the analogous natural
counterpart for regular tree grammars~\cite{LM06}. One of the main
interests of straight-line grammars is that they offer a way to compress
a word (resp. a tree) by encoding it by a context-free grammar (resp. a
regular tree grammar). A word $u$ can potentially be represented by a
context-free grammar (as the unique element of its language) with less
memory than the direct representation of $u$, provided that $u$ is made
of several repeating factors. The analogous definition for bud
generating systems could potentially be used to compress a large variety
of combinatorial objects. Indeed, given a suitable monochrome operad
$\Oca$ defined on the objects we want to compress, we can encode an
object $x$ of $\Oca$ by a bud generating system $\Bca$ with $\Oca$ as
ground operad and such that the language (or the synchronous language)
of $\Bca$ is a singleton $\{y\}$ and $\Prune(y) = x$. Hence, we can hope
to obtain a new and efficient method to compress arbitrary combinatorial
objects.
\smallbreak

Let us finally describe a third extension of this work. \Def{Pros} are
algebraic structures which naturally generalize operads. Indeed, a pro
is a set of operators with several inputs and several outputs, unlike in
operads where operators have only one output (see for
instance~\cite{Mar08}). Surprisingly, pros appeared earlier than operads
in the literature~\cite{McL65}. It seems fruitful to translate the main
definitions and constructions of this work (as {\em e.g.,} bud operads,
bud generating systems, series on colored operads, pre-Lie and
composition products of series, star operations, {\em etc.}) with pros
instead of operads. We can expect to obtain an even more general class
of grammars and obtain a more general framework for combinatorial
generation.
\medbreak

\begin{footnotesize}
\bibliographystyle{alpha}
\bibliography{Bibliography}
\end{footnotesize}

\end{document}